%% file: deconvolNP.tex
\documentclass
	{article}
\usepackage[utf8]{inputenc}
\usepackage[english]{babel}
\usepackage{authblk}
\usepackage{amssymb,amsmath,amsthm,bm}
\usepackage{times,dsfont}
\usepackage[colorlinks,linkcolor=black,citecolor=black,urlcolor=black]{hyperref}
\usepackage{fancyhdr}
\usepackage{geometry}
\title{Deconvolution with unknown noise distribution is possible for multivariate signals}
\date{}

\author[$\star$]{\'Elisabeth Gassiat}
\author[$\dag$]{Sylvain Le Corff}
\author[$\ddag$]{Luc Leh\'ericy}
\affil[$\star$]{{\small Universit\'e Paris-Saclay, CNRS, Laboratoire de math\'ematiques d'Orsay, 91405, Orsay, France.}}
\affil[$\dag$]{{\small Samovar, T\'el\'ecom SudParis, d\'epartement CITI, TIPIC, Institut Polytechnique de Paris, Palaiseau, France.}}
\affil[$\ddag$]{{\small Laboratoire J. A. Dieudonné, Universit\'e C\^ote d'Azur, CNRS, 06100, Nice, France.}}

\newcommand{\M}{\mathcal{M}}

\newcommand{\N}{\mathbb{N}}
\newcommand{\R}{\mathbb{R}}
\newcommand{\C}{\mathbb{C}}

\newcommand{\cbeta}{c_{\beta}}

\newcommand{\com}[1]{\textcolor{black}{#1}}
\newcommand{\E}{\mathbb{E}}

\newcommand{\po}{\mathbb{P}}
\renewcommand{\tilde}{\widetilde}
\renewcommand{\hat}{\widehat}
\newcommand\indicator{\mathbf1}
\newcommand\ind{\indicator}

\DeclareMathOperator*{\argmin}{arg\,min}

\newcommand\noisedist{Q}
\newcommand\transk{R}
\newcommand\noisedistbis{\widetilde{Q}}
\newcommand\transkbis{\widetilde{R}}
\newcommand\eqsp{\,}
\newcommand\rmd{\mathrm{d}}
\newcommand{\rme}{\mathrm{e}}
\newcommand\neighborhood[2]{\mathsf{B}^{#1}_{#2}}

\newcommand{\Mfrak}{{\mathfrak{M}}}
\newcommand{\Bbf}{{\mathbf{B}}}
\newcommand{\Lfrak}{{\mathfrak{L}}}
\newcommand{\rk}{{\text{rk}}}

\newcommand{\Acal}{{\mathcal{A}}}

\newcommand{\Fcal}{{\mathcal{F}}}
\newcommand{\Gcal}{{\mathcal{G}}}
\newcommand{\Hcal}{{\mathcal{H}}}

\newcommand{\Lbf}{{\mathbf{L}}}

\newcommand{\Mcal}{{\mathcal{M}}}

\newcommand{\Pbf}{{\mathbf{P}}}

\newcommand{\Qbf}{{\mathbf{Q}}}

\newcommand{\Sbf}{{\mathbf{S}}}

\newcommand{\Zcal}{{\mathcal{Z}}}

\newcommand{\bfX}{{\mathbf{X}}}
\newcommand{\Xbf}{{\mathbf{X}}}
\newcommand{\bfY}{{\mathbf{Y}}}
\newcommand{\bfeps}{\bm{\varepsilon}}

\newcommand{\one}{{\mathbf1}}

\newcommand{\sfc}{\mathsf{c}}

\newcommand{\rvline}{\hspace*{-\arraycolsep}\vline\hspace*{-\arraycolsep}}

\renewcommand{\geq}{\geqslant}
\renewcommand{\leq}{\leqslant}

\usepackage{booktabs} 
\usepackage{colortbl} 
\usepackage{xcolor} 
\usepackage{xfrac}
\usepackage{enumerate}

\newtheorem{theo}{Theorem}
\newtheorem{prop}[theo]{Proposition}
\newtheorem{cor}[theo]{Corollary}
\newtheorem{lem}[theo]{Lemma}
\newtheorem{conj}[theo]{Conjecture}

\newcounter{hypH}
\newenvironment{hypH}{\refstepcounter{hypH}\begin{itemize}
\item[\textbf{H\arabic{hypH}}]}{\end{itemize}}

\begin{document}

\maketitle

\begin{abstract}
This paper considers the deconvolution problem in the case where the target signal is multidimensional and no information is known about the noise distribution. More precisely, no assumption is made on the noise distribution and no samples are available to estimate it: the deconvolution problem is solved based only on the corrupted signal observations. We establish the identifiability of the model up to translation when the signal has a Laplace transform with an exponential growth smaller than $2$ and when it can be decomposed into two dependent components. Then, we propose an estimator of the probability density function of the signal without any 
assumption on the noise distribution. As this estimator depends of the lightness of the tail of the signal distribution which is usually unknown, a model selection procedure is proposed to obtain an adaptive estimator in this parameter with the same rate of convergence as the estimator with a known tail parameter. Finally, we establish a lower bound on the minimax rate of convergence that matches the upper bound.
\end{abstract}



\section{Introduction}


Estimating the distribution of a signal corrupted by some additive noise, referred to as solving the {\em deconvolution problem}, is a long-standing challenge in nonparametric statistics. In such problems, the observation $\bfY$ is given by
\begin{equation}
\label{eq:model:deconvolution}
\bfY = \bfX + \bfeps\eqsp,
\end{equation}
where $\bfX$ is the signal and $\bfeps$ is the noise. Recovering the distribution of the signal using 
data contaminated by additive noise is an all-pervasive problem in all fields of statistics, see \cite{meister:2009} and the references therein. It has been applied in a large variety of disciplines and has stimulated a great research interest for instance in signal processing \cite{moulines1997maximum,attias1998blind}, in image reconstruction \cite{kundur1996blind,campisi2017blind} or in astronomy \cite{starck2002deconvolution}. 

Although a great deal of research effort has been devoted to design efficient estimators of the distribution of the signal and to derive optimal convergence rates, the results available in the literature suffer from a crucial limitation: they assume that the distribution of the noise is known. Estimators based on Fourier transforms are the most widespread in this setting as convolution with a known error density translates into a multiplication of the Fourier transform of the signal by the Fourier transform of the noise. However, this assumption may have a significant impact on the robustness of deconvolution estimators as pointed out in \cite{meister2004effect} where the author established that the mean integrated squared error of such an estimator can grow to infinity when the noise distribution is misspecified. 

The aim of this paper is to solve the deconvolution problem without any assumption on the noise distribution and based solely on a sample of observations $\bfY_1, \dots, \bfY_n$. In particular, we do not assume that some samples with the same distribution as $\bfeps$ are available as in \cite{johannes:2009,lacour:comte:2010}.
We prove this is possible as soon as the signal $\bfX$ (i) has a distribution with light enough tails and (ii) has at least two dimensions and may be decomposed into two subsets of random variables which satisfy some weak dependency assumption. We then propose an estimator of the density of its distribution which is shown to be minimax adaptive for the mean integrated squared error.


The main reason why it becomes possible to solve the deconvolution problem in this multivariate setting is the structural difference between signal and noise in term of dependence structure of the two components: the signal has dependent components and the noise has independent components. We prove that such a hidden structure may be discovered based solely on observations $\bfY_1, \dots, \bfY_n$.
A first step to establish the identifiability in deconvolution without any assumption on the noise was obtained by \cite{gassiat:rousseau:2016} with a dependency assumption on the signal, but under the restrictive assumption that the signal takes a finite number of values. This identifiability result was extended
recently by \cite{gassiat:lecorff:lehericy:2019} who proved the identifiability up to translation of the distributions of the signal and of the noise when the hidden signal is a 
hidden stationary Markov chain 
independent of the noise. 
Building upon these ideas,
the first part of our paper establishes the identifiability up to translation of the deconvolution model when the signal $\bfX$ which lies in $\R^d$,   $d\geq 2$, can be decomposed into two dependent components $X^{(1)}\in\R^{d_1}$,  $d_1 \geq 1$, and $X^{(2)}\in\R^{d_2}$,  $d_2 \geq 1$, with $d_1+d_2 = d$:
\begin{equation}
\label{eq:model2:deconvolution}
\bfY = \begin{pmatrix} Y^{(1)}\\Y^{(2)}\end{pmatrix} = \begin{pmatrix} X^{(1)}\\X^{(2)}\end{pmatrix} + \begin{pmatrix} \varepsilon^{(1)}\\ \varepsilon^{(2)}\end{pmatrix} = \bfX + \bfeps \eqsp.
\end{equation}
The identifiability up to translation of the law of $\bfX\in\R^{d_1+d_2}$ and of $\bfeps\in\R^{d_1+d_2}$ based on the law of $\bfY$ when the noise is independent of the signal only requires that the Laplace transform of the signal has an exponential growth smaller than $2$ and some dependency assumption between the two components of the signal. 

The second objective of this paper is to propose a rate optimal estimator of the probability density function of $\bfX$ without any assumptions on the noise distribution. In the pioneering works on deconvolution for i.i.d. data, the distribution of $\bfX$ is recovered by filtering the received observations to compensate for the convolution using Fourier inversion and Kernel based methods, see \cite{MR1033106,MR1047309,MR1054861} for some early nonparametric deconvolution methods and \cite{MR997599,MR1126324} for minimax rates. On the other hand, more recent works were dedicated to multivariate deconvolution problems such as \cite{comte2013anisotropic} for kernel density estimators, \cite{sarkar2018bayesian} for a Bayesian approach or \cite{eckle:bissantz:dette:2016} for a multiscale based inference. In all these works, deconvolution is solved under two restrictive assumptions: (a) the distribution of the noise is assumed to be known and (b) this distribution is assumed to be such that its Fourier transform is nowhere vanishing.

An important step toward solving the deconvolution problem without such restrictions on the noise distribution was achieved in \cite{meister2007deconvolving} for signals in $\R$ with a probability density function supported on a compact subset of $\R$. In \cite{meister2007deconvolving}, the estimation procedure only requires the Fourier transform of the noise to be known on a compact interval around 0. The procedure relies first on recovering as usual the Fourier transform of the signal by direct inversion on the compact interval where the noise distribution is known, and by choosing a polynomial expansion on this compact interval. Then, the Fourier transform is extended to larger intervals before using a Fourier inversion to provide a probability density estimator. 
Under standard smoothness assumptions, \cite{meister2007deconvolving} established an upper bound for the mean integrated squared error 
which is shown to be optimal under a few additional assumptions.

In this paper, we propose an estimation procedure inspired from our identifiability proof. We provide an identification equation on Fourier transforms which can be used to build a contrast function to be minimized on possible estimators for the unknown Fourier transform of the distribution of the signal. Once an estimator of the Fourier transform of the signal on a neighborhood of 0 is available, we use polynomial expansions of this estimator as in~\cite{meister2007deconvolving} to extend it to $\R^{d_1+d_2}$ before using a Fourier inversion to obtain an estimator of the density.
One of the main hurdles to overcome is then to relate the value of the contrast function to the error on the Fourier transform. Under common smoothness assumptions, we provide rates of convergence for the estimator of the probability density function of $X$ depending on the lightness of its tail. Both the regularity and the tail lightness have an impact on the rates of convergence. Surprisingly, while this estimation procedure does not require any prior knowledge on the noise, we obtain the same rates as in~\cite{meister2007deconvolving} when the signal distribution has a compact support: not knowing the noise distribution does not affect these rates.

We then propose a model selection method to obtain an estimator that is rate adaptive to the unknown lightness of the tail. Finally, we establish a lower bound on the minimax rate of convergence that matches the upper bound. Minimax rates of convergence in deconvolution problems may be found in
\cite{MR1126324}, \cite{butucea:tsybakov:2008a}, \cite{butucea:tsybakov:2008b} and in \cite{meister:2009}. In most works on deconvolution, not only the distribution of the noise is assumed to be known (or estimated for instance as in \cite{johannes:2009} and \cite{lacour:comte:2010}) but the rates of convergence depend on the decay of its Fourier transform (ordinary or super smooth). It is interesting to note that in our context where the noise is completely unknown, the minimax rate of convergence depends only on the signal and not on the noise.

The paper is organized as follows. Section~\ref{sec:identifiability:general} displays the general identifiability result which establishes that the distributions of the signal and of the noise can be recovered from the observations up to a translation indeterminacy. 
This general result allows to identify submodels as illustrated in Section~\ref{sec:examples} with several common statistical frameworks.
Section~\ref{sec:estimation} describes the adaptive estimation procedure and provides convergence rates. Section~\ref{sec:lower} states the lower bound on the minimax rates of convergence and Section~\ref{sec:further} suggests a few possibilities for future works and settings in which our results may contribute significantly.
All proofs are postponed to the appendices.

\section{Identifiability results}

\subsection{General theorem}
\label{sec:identifiability:general}
The following assumption is assumed to hold throughout the paper.
\begin{hypH}
\label{assum:model}
The signal $\bfX$ belongs to $\C^d$ with $d\geqslant 2$ and the observation model is given by  \eqref{eq:model2:deconvolution}  in which $\bfeps$ is independent of $\bfX$ and $\varepsilon^{(1)}$ is independent of $\varepsilon^{(2)}$.
\end{hypH}

Consider model~\eqref{eq:model2:deconvolution} in which
$\bfeps$ is independent of $\bfX$ 
and $\varepsilon^{(1)}$ is independent of $\varepsilon^{(2)}$. Let $\po^{}_{R,\noisedist}$ be the distribution of $\bfY$ when $\bfX$ has distribution $R$ and for $i\in\{1,2\}$, $\varepsilon^{(i)}$ has distribution $\noisedist^{(i)}$, with $\noisedist= \noisedist^{(1)} \otimes \noisedist^{(2)}$. Denote by $R^{(1)}$ the distribution of $X^{(1)}$ and by $R^{(2)}$ the distribution of $X^{(2)}$. For any $\rho \geq 0$ and any integer $p \geq 1$, let $\M^p_\rho$ be the set of positive measures $\mu$ on $\R^{p}$ such that 
there exist $A,B>0$ satisfying, for all $\lambda\in \R^{p}$,
\begin{equation*}
\int\exp \left(\lambda^\top x\right)\mu (\rmd x)
\leq A \exp \left( B \|\lambda\|^\rho\right)\eqsp,
\end{equation*}
where for a vector $\lambda$ in a Euclidian space, $\|\lambda\|$ denotes its euclidian norm. 
When $R \in \Mcal_\rho^d$, the characteristic function of $R$ can be extended into a multivariate analytic function denoted by
\begin{eqnarray*}
\Phi_{R}: \C^{d_1}\times \C^{d_2} &\longrightarrow& \C\\
(z_1,z_2)&\longmapsto& \int \exp \left(iz_1^\top x_1 + i z_2^\top x_2\right)R(\rmd x_1,\rmd x_2)\eqsp.
\end{eqnarray*}
Let us now introduce the structural assumption on the signal $\bfX$. Note that no assumption other than H\ref{assum:model} is made on the noise $\bfeps$, and that assumption H\ref{assum:} may be understood as a dependency assumption between the  components $X^{(1)}$ and $X^{(2)}$ of $\bfX$ as discussed below.
\begin{hypH}
\label{assum:}
For any $z_{0}\in \C^{d_1}$, 
$z \mapsto \Phi_{R}(z_{0},z)$
is not the null function and for any $z_{0}\in \C^{d_2}$, 
$z \mapsto \Phi_{R}(z,z_{0})$
is not the null function.
\end{hypH}
Assumption H\ref{assum:} means 
that 
for any $z_{1}\in \C^{d_1}$, there exists $z_2\in\C^{d_2}$ such that $\Phi_{R}(z_{1},z_2) \neq 0$ and for any $z_{2}\in \C^{d_2}$, there exists $z_1\in\C^{d_1}$ such that $\Phi_{R}(z_{1},z_2) \neq 0$.

In the following, the assertion \emph{$R = \tilde{R}$ and $\noisedist = \noisedistbis$ up to translation} means that there exists $m=(m_1,m_2)\in \R^{d_1}\times \R^{d_2}$ such that if $X$ has distribution $R$ and for $i\in\{1,2\}$, $\varepsilon_{i}$ has distribution $\noisedist_{i}$, then $(X_{i}-m_{i})_{i\in\{1,2\}}$ has distribution $ \tilde{R}$ and for $i\in\{1,2\}$, $\epsilon_{i}+m_{i}$ has distribution $\noisedistbis_{i}$.
 
\begin{theo}
\label{theoident1}
Assume that $R$ and $\tilde{R}$ are probability distributions on $\R^d$ which satisfy assumption H\ref{assum:}. Assume also that there exists $\rho <2$ such that 
$R$ and $\tilde{R}$ are in $\M^{d}_\rho$. 
Then, $\po^{}_{R,\noisedist}=\po^{}_{\tilde{R},\noisedistbis}$ implies that $R= \tilde{R}$ and $\noisedist = \noisedistbis$ up to translation.
\end{theo}

%

One way to fix the ``up to translation" indeterminacy when the noise has a first order moment is to assume that $\mathbb{E}[\bfeps]=0$.
The proof of Theorem \ref{theoident1} is postponed to Appendix~\ref{app:proof:ident}.

\paragraph{Comments on the assumptions of Theorem \ref{theoident1}.} First of all, Theorem~\ref{theoident1} involves no assumption at all on the noise distribution. This noise can be deterministic and there is no assumption on the set where its characteristic function vanishes. In addition, there is no density or singularity assumption on the distribution of the hidden signal. The signal may have an atomic or a continuous distribution, and no specific knowledge about this is required. The only assumptions are on the tail of the signal distribution and assumption H\ref{assum:} which, as discussed below, is a dependency assumption.

The assumption that $R\in \M_{\rho}^d$ is an assumption on the tails of the distribution of $\bfX$. If $R$ is compactly supported, then $R\in \M_{1}^d$, and if a probability distribution is in $\M_{\rho}^d$ for some $\rho$, then $\rho \geq 1$ except in case it is a Dirac mass at point $0$. The assumption $\rho < 2$ means that $R$ is required to have tails lighter than that of Gaussian distributions. It is useful to note that $R$ is in $\M^{d}_\rho$ for some $\rho$ if and only if 
$R^{(1)}$ is in $\M^{d_1}_\rho$ for some $\rho$ and $R^{(2)}$ is in $\M^{d_2}_\rho$ for some $\rho$. 

Let us now comment assumption H\ref{assum:}. Hadamard's factorization theorem states that entire functions are completely determined by their set of zeros up to a multiplicative indeterminacy which is the exponential of a polynomial with degree at most the exponential growth of the function (here $\rho$). If $R\in \M_{\rho}$ for some $\rho <2$, then a consequence of Hadamard's factorization theorem (arguing variable by variable) is that $\Phi_{R}\left(\cdot \right)$ has no complex zeros if and only if $R\in \M_{\rho}$ is a dirac mass. 
Since we are interested  in non deterministic signals, in general $\Phi_{R}(\cdot,\cdot)$, $\Phi_{R}(\cdot,0)$ and $\Phi_{R}(0,\cdot)$ will have complex zeros. 
Now, if the variables $X^{(1)}$ and $X^{(2)}$ are independent, then for all $z_{1}\in\C^{d_1}$ and $z_{2}\in\C^{d_2}$,
$\Phi_{R}\left(z_{1},z_{2}\right)=\Phi_{R}\left(z_{1},0\right)\Phi_{R}\left(0,z_{2}\right)$, so that $\Phi_{R}(z_1,\cdot)$ is identically zero as soon as $z_1$ is a complex zero of $\Phi_{R}\left(\cdot,0\right)$. 
Thus, assumption H\ref{assum:} implies that the variables $X^{(1)}$ and $X^{(2)}$ are not independent except if they are deterministic.
Moreover, if  for $i\in\{1,2\}$, $X^{(i)}$ can be decomposed as $X^{(i)}=\tilde{X}^{(i)}+\eta_{i}$, with $\eta_{1}$ and $\eta_{2}$ independent variables independent of $\tilde{\bfX}=(\tilde{X}^{(1)},\tilde{X}^{(2)})$, and if for some $z_{1}$, $\E[e^{iz_{1}^\top\eta_{1}}]=0$ or for some $z_{2}$, $\E[e^{iz_{2}^\top\eta_{2}}]=0$, then  H\ref{assum:} does not hold. In other words, H\ref{assum:} can hold only if all the additive noise has been removed from $\mathbf X$. Here, additive noise means a random variable with independent components.When the components $X^{(1)}$ and $X^{(2)}$ of the signal have each a finite support set of cardinality $2$, Assumption H\ref{assum:} is even equivalent to the fact that $X^{(1)}$ and $X^{(2)}$ are not independent.

Other examples in which assumption H\ref{assum:} holds are provided in Section \ref{sec:examples}, showing that assumption H\ref{assum:} is a mild assumption which may hold for a large class of multivariate signals with dependent components.



\subsection{Identification of structured submodels}
\label{sec:examples}

This section displays examples to which Theorem~\ref{theoident1} applies, and in particular, for each model, we explicitize conditions which ensure that assumption H\ref{assum:} holds.
This means of course that such models are identifiable. But, since they are submodels of the general model, it also means that they may be recovered in this larger general model.
Further examples that could be investigated are discussed in Section \ref{sec:further}.

\subsubsection{
Noisy Independent Component Analysis}


Independent Component Analysis assumes that $\bfY \in \R^d$ is a random vector such that there exists an unknown integer $q \geq 1$, an unknown matrix $A$ of size $d \times q$, and two independent random vectors $\mathbf{S} \in \R^q$ and $\bfeps \in \R^d$ such that 
\begin{equation}
\label{eq:ica:noisy}
\bfY = A\mathbf{S} + \bfeps\eqsp,
\end{equation}
where all coordinates of the signal $\mathbf{S}$ are independent, centered and with variance one and all coordinates of the noise $\bfeps$ are independent. The statistical challenge 
lies in estimating $A$ and the probability distribution of $\mathbf{S}$ while only $Y$ is observed. The {\em noise free} formulation of this problem, i.e. $\bfY = A\mathbf{S} $, was proposed in the signal processing litterature, see for instance \cite{jutten:1991}. The identifiability of the noise free linear independent component analysis has been established in \cite{comon:1994,eriksson:koivunen:2004} under the following (sufficient conditions).
\begin{itemize}
\renewcommand{\labelitemi}{-}
\item The components $S_i$, $1\leq i \leq q$, are not Gaussian random variables (with the possible exception of one component).
\item $d\geq q$, i.e. the number of observations is greater than the number of independent components.
\item The matrix $A$ has full rank.
\end{itemize}
A noisy extension of the ordinary ICA model which implies further identifiability issues was considered for instance in \cite{moulines1997maximum}. A correct identification of the mixing matrix $A$ can be obtained by assuming that the additive noise is Gaussian and independent of the signal sources which are non-Gaussian, see for instance \cite{hyvarinen:karhunen:oja:2000}. In our paper, identifiability of the ICA model with unknown additive noise is established using Theorem~\ref{theoident1} under some assumptions (discussed below). In the following, for any subset $I$ of $\{1,\ldots,d\}$ and any matrix $B$ of size $d \times q$, let $B_I$ denote the $|I|\times q$ matrix whose lines are the lines of $B$ with index in $I$, where $|C|$ is the number of element of any finite set $C$.
\begin{cor}
\label{ident:ICAnoisy}
Let $A$ and $\tilde A$ be two matrices of size $d \times q$. Assume that there exists a partition $I\cup J = \{1,\ldots,d\}$ such that all columns of $A_I$, $\tilde A_I$, $A_J$ and $\tilde A_J$ are nonzero. Assume also that $(S_j)_{1\leq j\leq q}$ (resp. $(\tilde S_j)_{1\leq j\leq q}$) are independent and that there exists $\rho <2$ such that the distributions of all $S_j$ (resp. $\tilde S_j$) are in $\M^1_\rho$. Denote by $Q$ (resp. $\tilde Q$) the distribution of $\bfeps$ (resp. $\tilde \bfeps$) and by $R$ (resp. $\tilde R$) the distribution of $A\Sbf$ (resp. $\tilde A\tilde \Sbf$) in \eqref{eq:ica:noisy}. Then, $\po_{R,P}=\po_{\tilde R,\tilde P}$ implies that $R= \tilde{R}$ and $\noisedist = \noisedistbis$ up to translation.
\end{cor}
Corollary \ref{ident:ICAnoisy} is proved in Section \ref{sec:proofCor2}. Apart from the assumption that the independent components of the signal have distribution with light tails, the main assumption is that the observation $\bfY$ may be splitted in two known parts so that the corresponding lines of the matrix $A$ have a non zero entry in each column. 
Although this assumption is not common in the ICA literature, as explained in \cite[Section~1.1.3]{pfister:2019}, a wide range of applications require to design source separation techniques to deal with grouped data. Identifiability of such a group structured ICA is likely to rely on specific assumptions and we propose in Corollary~\ref{ident:ICAnoisy} a set of assumptions which allow to apply Theorem~\ref{theoident1}.

\subsubsection{
Repeated measurements} 

In deconvolution problems with repeated measurements, the observation model is
\begin{equation}
\label{eq:repeat:meas}
Y^{(1)} = X^{(1)} + \varepsilon^{(1)}\quad\mathrm{and}\quad Y^{(2)} = X^{(1)} + \varepsilon^{(2)}\eqsp,
\end{equation}
where $X^{(1)}$ has distribution $R^{(1)}$ on $\R^{d_1}$ and is independent of $\bfeps = (\varepsilon^{(1)},\varepsilon^{(2)})^\top$ where $\varepsilon^{(1)}$ is independent of $\varepsilon^{(2)}$ and $\bfeps$ has distribution $Q$, see \cite{MR2396811} for a detailed description of such models and all the references therein for the numerous applications. Let $R$ be the distribution of $(X^{(1)},X^{(1)})^\top$ on $\R^{2d_1}$.
\begin{cor}
Assume that there exists $\rho <2$ such that $R^{(1)}$ and $\tilde{R}^{(1)}$ are in $\M^{d_1}_\rho$. Then, $\po^{}_{R,\noisedist}=\po^{}_{\tilde{R},\noisedistbis}$ implies that $R= \tilde{R}$ and $\noisedist = \noisedistbis$ up to translation.
\end{cor}
\begin{proof}
Assumption H\ref{assum:} holds 
since $\Phi_{R}(z_{1},z_{2})=\Phi_{R^{(1)}}(z_{1}+z_{2})$ for all $z_{1}\in \C^{d_{1}}$ and $z_{2}\in \C^{d_{1}}$, and $\Phi_{R^{(1)}}$ can not be identically zero since $\Phi_{R^{(1)}}(0)=1$.  We then apply Theorem~\ref{theoident1}. 
\end{proof}
 
Therefore, deconvolution with at least two repetitions is identifiable without any assumption on the noise distribution, under the mild assumption that the distribution of the variable of interest has light tails. 
The model may also contain outliers with unknown probability and still be identifiable. 

Corollary~\ref{cor:repeated} generalizes \cite[Lemma~1]{kotlarski:1967}, in which $\bfY$ is assumed to have a non vanishing characteristic function, which implies that  the characteristic functions of $X^{(1)}$ and of the noise are not vanishing everywhere. 
Identifiability of model~\eqref{eq:repeat:meas} has been proved by \cite{MR1625869} under the assumption that the characteristic functions of $X^{(1)}$ and of the noise are not vanishing everywhere. In \cite{MR2396811}, kernel estimators where proved equivalent to those for deconvolution with known noise distribution when $X^{(1)}$ has a real characteristic function and for ordinary smooth errors and signal.

\subsubsection{
Errors in variable regression models}

The observations of errors in variable regression models are defined as
\begin{equation}
\label{eq:repeat:errorsinvariables}
Y^{(1)} = X^{(1)} + \varepsilon^{(1)}\quad\mathrm{and}\quad Y^{(2)} = g(X^{(1)}) + \varepsilon^{(2)}\eqsp,
\end{equation}
where $g:\R^{d_1} \rightarrow \R^{d_2}$, $X^{(1)}$ has distribution $R^{(1)}$ on $\R^{d_1}$ and is independent of $\bfeps = (\varepsilon^{(1)},\varepsilon^{(2)})^\top$, $\varepsilon^{(1)}$ is independent of $\varepsilon^{(2)}$ and $\bfeps$ has distribution $\noisedist$. Let $R$ be the distribution of $(X^{(1)},g(X^{(1)}))$ on $\R^{d_1+d_2}$. If the distribution of $(X^{(1)},g(X^{(1)}))$ is identified, then its support is identified and the support of $(X^{(1)},g(X^{(1)}))$ is the graph of the function $g$ so that $g$ is identified on the support of the distribution of $X^{(1)}$.
\begin{cor}
Assume that there exists $\rho <2$ such that $R^{(1)}$ and $\tilde{R}^{(1)}$ are in $\M^{d_1}_\rho$ and that $R^{(2)}$ and $\tilde{R}^{(2)}$ are in $\M^{d_2}_\rho$. Assume also that the supports of $X^{(1)}$ and $g( X^{(1)})$ have a nonempty interior and that $g$ is one-to-one on a subset of the support of $X_1$ with nonempty interior. Then, $\po^{}_{R,\noisedist}=\po^{}_{\tilde{R},\noisedistbis}$ implies that $R= \tilde{R}$ and $\noisedist = \noisedistbis$ up to translation.
\end{cor}
This identifiability relies on weaker assumptions on the errors in variable regression models than in \cite{MR2396811} where the noise distribution is assumed to be ordinary-smooth (which implies in particular that its Fourier transform does not vanish on the real line) and where the distribution of $X^{(1)}$ is assumed to have a probability density with respect to the Lebesgue measure on $\mathbb{R}$. In \cite{MR3174611}, the authors also assumed a nowhere vanishing Fourier transform of the noise distribution and that the distribution of $X^{(1)}$ admits a probability density with respect to the Lebesgue measure uniformly bounded and supported on an open interval. In this setting (more restrictive on the noise and with different restrictions on the signal), the identification result in \cite{MR3174611} is not comparable to ours.

\begin{proof}
The proof boils down to establishing that Assumption H\ref{assum:} holds to apply Theorem~\ref{theoident1}. If Assumption H\ref{assum:} does not hold, then either there exists $z_{0}\in \C^{d_1}$ such that for all $z\in \C^{d_2}$, $\E[\rme^{z_{0}^\top X^{(1)}+z^\top g(X^{(1)})}]=0$, or there exists $z_{0}\in \C^{d_2}$ such that for all $z\in \C^{d_1}$, $\E[\rme^{z^\top X^{(1)}+z_{0}^\top g(X^{(1)})}]=0$. In the last case, since the support of $X^{(1)}$ has a nonempty interior, this is equivalent to
$\E[\rme^{z_{0}^\top g(X^{(1)})}\vert X^{(1)}]=0$, which means that $e^{z_{0}^\top g(X^{(1)})}=0$, which is impossible. Thus, since the support of $g(X^{(1)})$ has a nonempty interior (which is the case for instance if $g$ is a continuous function), H\ref{assum:} does not hold if and only if for some $z_{0}$,
$\E\left[\rme^{z_{0}^\top X^{(1)}} \middle\vert g(X^{(1)}) \right]=0$. The error in variables regression model is then identifiable without knowing the distribution of the noise as soon as for all $z_{0}$,
\begin{equation}
\label{eq:errors}
\E\left[\rme^{z_{0}^\top X^{(1)}} \middle\vert g(X^{(1)}) \right]\neq 0\eqsp.
\end{equation}
When $g$ is one-to-one on a subset of the support of $X^{(1)}$ with nonempty interior, for all $z_{0}$,~\eqref{eq:errors} is verified and the model is identifiable. 
\end{proof}

\section{Upper bounds}
\label{sec:estimation}

In this section, we propose an estimator of the signal density that is adaptive in the tail parameter $\rho$ and we study its rate of convergence. We first explain in Section \ref{subsec:estimateur} the construction of the estimator for a fixed tail parameter. We then study in Section \ref{subsec:rates} the rates of convergence for the estimators with fixed tail parameter and give an upper bound for the maximum integrated squared error over a class of densities with fixed regularity and tail parameter. We provide in Section \ref{subsec:adaptive} a model selection method to choose the tail parameter based solely on data $\bfY_{1},\ldots,\bfY_{n}$ and prove that the resulting estimator is rate adaptive over the previously considered classes of regularity and tail parameter. We further study in Section \ref{sec:lower} a lower bound of the minimax rate indicating that our final estimator is rate minimax adaptive.

In the following, the unknown distribution of the signal is denoted $R^\star$ and we assume it admits a density $f^\star$ with respect to the Lebesgue measure. Likewise, the unknown distribution of the noise is written $Q^\star$.

\subsection{Estimation procedure}
\label{subsec:estimateur}

Define, for any positive integer $p$ and any $\nu>0$, $\neighborhood{p}{\nu} = [-\nu,\nu]^p$. For all positive integer $p$ and any $\nu > 0$, write $\Lbf^2(\neighborhood{p}{\nu})$ the set of square integrable functions on $\neighborhood{p}{\nu}$ (possibly taking complex values) with respect to the Lebesgue measure.
The first step of our procedure is to estimate the Fourier transform of $f^\star$. 
For all $\nu>0$ and all measurable and bounded functions $\phi : \neighborhood{d_1}{\nu}\times \neighborhood{d_2}{\nu} \to \C$, define
\begin{multline*}
M(\phi; \nu | R^\star, Q^\star)
	= \int_{\neighborhood{d_1}{\nu}\times \neighborhood{d_2}{\nu}} | \phi(t_1,t_2) \Phi_{R^\star}(t_1,0) \Phi_{R^\star}(0,t_2) - \Phi_{R^\star}(t_1,t_2) \phi(t_1,0) \phi(0,t_2) |^2 \\
		| \Phi_{Q^{\star,(1)}}(t_1) \Phi_{Q^{\star,(2)}}(t_2) |^2 \rmd t_1 \rmd t_2\eqsp,
\end{multline*}
where $\Phi_{Q^{\star,(1)}}$ (resp. $\Phi_{Q^{\star,(2)}}$) is the Fourier transform of the (unknown) distribution $Q^{\star,(1)}$ of $\varepsilon_1$ (resp. $Q^{\star,(2)}$ of $\varepsilon_2$). This contrast function is inspired by the identifiability proof, see equation~\eqref{chara5bis}. Indeed, following the identifiability proof, we know that for all $Q^\star$, if $R^\star$ satisfies the assumptions of Theorem \ref{theoident1}, and if
$\phi$ is a multivariate analytic function satisfying Assumption H\ref{assum:}, such that there exists $A,B>0$ and $\rho\in (0,2)$ such that for all $(z_{1},z_{2})\in \R^{d_1}\times \R^{d_2}$, $|\phi(iz_{1},iz_{2})|\leq A \exp (B\|(z_{1},z_{2})\|^{\rho})$ and such that for all $z\in\R^{d}$, $\overline{\phi(z)}=\phi(-z)$, then for any $\nu >0$, 
\begin{equation}
\label{equa:ident}
M(\phi; \nu | R^\star,Q^\star)=0\;\text{ if and only if }\;\phi=\Phi_{R^\star}.
\end{equation}

In practice, $R^\star$ and $Q^\star$ are unknown, so the estimator is defined by minimizing an empirical counterpart of $M$ over classes of analytic functions to be chosen later. Choose first some fixed $\nu_{\text{est}} > 0$ and for all $n \geq 0$, define
\begin{equation*}
M_{n}(\phi) = \int_{\neighborhood{d_1}{\nu_{\text{est}}}\times \neighborhood{d_2}{\nu_{\text{est}}}} | \phi(t_1,t_2) \tilde\phi_n(t_1,0) \tilde\phi_n(0,t_2) - \tilde \phi_n(t_1,t_2) \phi(t_1,0) \phi(0,t_2) |^2 \rmd t_1 \rmd t_2\eqsp,
\end{equation*}
where for all $(t_1,t_2)\in\C^{d_1}\times \C^{d_2}$,
\begin{equation*}
\tilde \phi_n(t_1,t_2) = \frac1{n}\sum_{\ell=1}^{n}\rme^{it_1^\top Y_{\ell}^{(1)} + it_2^\top Y_{\ell}^{(2)}}\eqsp.
\end{equation*}
For all $ i \in \N^d$ and all analytic function $\phi$ defined on $\C^d$, write $\partial^i \phi$ the partial derivative of order $i$ of $\phi$: for all $x\in\C^d$, $\partial^i \phi(x) = (\partial^{i_1}/\partial_{x_1})\ldots(\partial^{i_d}/\partial_{x_d}) \phi(x)$. For all $\kappa > 0$ and $S < \infty$, let 
\begin{equation}
\label{eq:def:funcset}
\Upsilon_{\kappa,S} = \left\{
\phi \text{ analytic } \text{s.t. } \forall z\in\R^{d},\overline{\phi(z)}=\phi(-z), \phi(0) = 1 \text{ and } \forall i \in \N^d \setminus \{0\}, \left| \frac{\partial^{i} \phi(0)}{\prod_{a=1}^d i_a!}\right| \leq \frac{S^{\|i\|_1}}{\|i\|_1^{\kappa \|i\|_1}} \right\}
\end{equation}
where $\|i\|_1 =\sum_{a=1}^{d} i_{a}$, and 
\begin{equation}
\label{eq:def:funcsetdiff}
\Gcal_{\kappa,S} = \{ \phi - \phi' : \phi, \phi' \in \Upsilon_{\kappa,S} \}\eqsp.
\end{equation}
Note that for all $\kappa > 0$ and $S < \infty$, the elements of $\Upsilon_{\kappa,S}$ are equal to their Taylor series expansion.
As shown in the following Lemma, the sets $\Upsilon_{\kappa,S}$ and $\Mcal_{1/\kappa}^d$ are equivalent in that the set of all characteristic functions in $\bigcup_S \Upsilon_{\kappa,S}$ is the set of characteristic functions of probability measures in $\Mcal^d_{1/\kappa}$. Its advantage over $\Mcal^d_{1/\kappa}$ is the more convenient characterization of its elements $\phi$ in terms of their Taylor expansion.
\begin{lem}
\label{lem_lien_Mrho_Upsilon}
For each $\rho \geq 1$ and probability measure $\mu \in \Mcal^d_\rho$, there exists $S>0$ such that $\lambda \mapsto \int \exp\left(i \lambda^\top x\right)\mu(\rmd x)$ is in $\Upsilon_{1/\rho,S}$. Conversely, for all $\kappa > 0$, there exists a constant $c$ such that for any $S > 0$ and for any probability measure $\mu$ on $\R^d$ such that $\lambda \mapsto \int \exp\left(i \lambda^\top x\right)\mu(\rmd x)$ is in $\Upsilon_{\kappa,S}$, $\mu$ satisfies for all $\lambda\in \R^{p}$,
\begin{equation*}
\int\exp \left(\lambda^\top x\right)\mu (\rmd x)
	\leq c \left(1 + (S \|\lambda\|)^{\frac{d+1}{\kappa}} \right)
	\exp \left( \kappa (S \|\lambda\|)^{1/\kappa}\right)\eqsp.
\end{equation*}
In particular, $\mu \in \Mcal_{1/\kappa}^d$.
\end{lem}
\begin{proof}
The proof is postponed to Appendix~\ref{sec_preuve_lem_lien_Mrho_Upsilon}.
\end{proof}

Let now $\Hcal$ be a set of functions $\R^d \rightarrow \C^d$ such that all elements of $\Hcal$ satisfy H\ref{assum:} and which is closed in $\Lbf^{2}([-\nu_{\text{est}},\nu_{\text{est}}]^d)$. For all $\kappa>0$, $n\geq 1$, the Fourier transform $\Phi_{R^\star}$ of the distribution of $X$ is estimated by \begin{equation}
\label{eq_def_phihat}
\widehat \phi_{\kappa,n} \in \underset{\phi \in \Upsilon_{\kappa,S}
 \cap \Hcal}{\argmin} M_{n}(\phi)\eqsp.
\end{equation}
To address possible measurability issues, note that we could take $\widehat \phi_{\kappa,n}$ as a measurable function such that $M_n(\widehat \phi_{\kappa,n}) \leq \inf_{\phi \in \Upsilon_{\kappa,S} \cap \Hcal} M_{n}(\phi) + 1/n$, and all following results would still hold.

Consistency of $\widehat \phi_{\kappa,n}$ will follow from (\ref{equa:ident}) and the compactness of $\Upsilon_{\kappa,S}\cap \Hcal$.
Now, to get an estimator of the density $f^\star$, there remains to perform Fourier inversion. First, we shall truncate the polynomial expansion of $\widehat \phi_{\kappa,n}$. 
For all $m \in \N$, let $\C_m[X_1, \dots, X_d]$ be the set of multivariate polynomials in $d$ indeterminates with (total) degree $m$ and coefficients in $\C$.
In the following, if $\phi$ is an analytic function defined in a neighborhood of $0$ in $\C^d$ written as $\phi: x \mapsto \sum_{i \in \N^d} c_i \prod_{a=1}^d x_a^{i_a}$, define its truncation on $\C_m[X_1, \dots, X_d]$ as 
\begin{equation}
\label{eq:def:tronc}
T_m \phi: x\mapsto \!\!\!\sum_{i \in \N^d : \|i\|_1 \leq m} c_i \prod_{a=1}^d x_a^{i_a}\eqsp.
\end{equation}
Then, for some integer
 $m_{\kappa,n}$ (to be chosen later), the estimator of $f^\star$ is defined as follows:
\begin{equation}
\label{eq:def:fhat}
\hat{f}_{\kappa,n}(x) =
 \frac1{(2\pi)^d} \int_{B_{\omega_{\kappa,n}}^{d_1}\times B_{\omega_{\kappa,n}}^{d_2}} \exp(-i t^\top x)\left(T_{m_{\kappa,n}} \hat{\phi}_{\kappa,n}\right)(t) \rmd t \eqsp,
\end{equation}
for some $\omega_{\kappa,n}>0$ (to be chosen later).

\subsection{Rates of convergence}
\label{subsec:rates}

In this section, we explain how to choose $(m_{\kappa,n})_{\kappa, n}$ and $(\omega_{\kappa,n})_{\kappa, n}$ to obtain the rate of convergence of $\hat{f}_{\kappa,n}$ to $ f^\star$ in $\Lbf^2(\R^{d_1} \times \R^{d_2})$.
For any $\kappa \in (0, 1]$, define
\begin{equation}
\label{eq:mkappa}
m_{\kappa,n} = \left\lfloor \frac{1}{8\kappa} \frac{\log n}{\log \log (n/4)} \right\rfloor 
\end{equation}
and
\begin{equation}
\label{eq:omegakappa}
\omega_{\kappa,n} = c_{\kappa} m_{\kappa,n}^\kappa / S
\end{equation}
for some constant $c_{\kappa} \leq \nu_{\text{est}}\wedge 2\kappa \exp(-(3d+5)/2)$.
We will also need to control the regularity of the target density $f^\star$ as in the following assumption.
\begin{hypH}
\label{assum:phistar}
We say that $\Phi_{R^\star}$ satisfies H\ref{assum:phistar} for the constants $\beta, \cbeta>0$ if
\begin{equation*}
\int_{\R^{d_1}\times \R^{d_2}} |\Phi_{R^\star}(t)|^2 (1 + \|t\|^2)^\beta \rmd t \leq c_\beta\eqsp.
\end{equation*}
\end{hypH}
For all $\kappa$, $S>0$, $\beta >0$, $c_{\beta}>0$, $\nu>0$, $c_\nu>0$ and $c_{Q}>0$, let
\begin{itemize}
\item $\Psi(\kappa,S,\beta,\cbeta)$ be the set of functions in $\Upsilon_{\kappa,S}$ that can be written as $\Phi_R$ for some probability measure $R$ on $\R^d$ and that satisfy H\ref{assum:phistar} for $\beta$, $\cbeta$.

\item $\Qbf(\nu,c_\nu,c_Q)$ be the class of probability measures of the form $Q^{(1)} \otimes Q^{(2)}$ where $Q^{(1)}$ (resp. $Q^{(2)}$) is a probability measure on $\R^{d_1}$ (resp. $\R^{d_2}$) such that $|\Phi_{Q^{(1)}}| \geq c_\nu$ on $[-\nu,\nu]^{d_1}$ and $|\Phi_{Q^{(2)}}| \geq c_\nu$ on $[-\nu,\nu]^{d_2}$, and such that if $\varepsilon$ is a random variable with distribution $Q$, then $\E[\|\varepsilon\|^2] \leq c_Q$.
\end{itemize}

\begin{theo}
\label{th:fhat}
For all $\kappa_0 > 1/2$, $S>0$, $\beta >0$, $c_Q>0$, $c_\nu>0$ and $\cbeta>0$, for all $\nu\in [(d+4/3)\rme /S,\nu_{\text{est}}]$,
\begin{equation*}
\label{eq:finalrate}
\limsup_{n \rightarrow +\infty}
\sup_{\kappa \in [\kappa_0, 1]}\!\!\!\!\!\!\!\!
\underset{R^\star \,:\, \Phi_{R^\star} \in \Psi(\kappa,S,\beta,\cbeta) \cap \Hcal}{\sup_{Q^\star \in \Qbf(\nu,c_\nu,c_Q)}}
\!\!\!\!\!\!\!\!\E_{R^\star,Q^\star} \left[
	\sup_{\kappa' \in [\kappa_0,\kappa]} \left\{
	\left( \frac{\log n}{\log \log n} \right)^{2 \kappa' \beta} \!\!\| \hat{f}_{\kappa',n} - f^\star \|_{\Lbf^2(\R^{d_1} \times \R^{d_2})}^2 \right\}\right]
	< +\infty \eqsp,
\end{equation*}
where $\Hcal$ is introduced in the definition of $\widehat \phi_{\kappa,n}$, see \eqref{eq_def_phihat}.
\end{theo}

\begin{proof}
The proof is postponed to Section~\ref{proof:theo:conv}.
\end{proof}
It is important to note that the procedure does not require the knowledge of $\nu$, which leads to the rate of convergence $\left(\log n/\log \log n\right)^{-2 \kappa \beta}$ without any prior knowledge about the distribution of the noise, since for any $\nu_{\text{est}}>0$, there exists $\nu\in [(d+4/3)\rme /S,\nu_{\text{est}}]$ such that $|\Phi_{Q^{(1)}}| \geq c_\nu$ on $[-\nu,\nu]^{d_1}$ and $|\Phi_{Q^{(2)}}| \geq c_\nu$ on $[-\nu,\nu]^{d_2}$ provided that $S$ is large enough.
Also, the assumption $\Phi_{R^\star} \in \Upsilon_{\kappa^\star,S}$ is not restrictive since by Lemma~\ref{lem_lien_Mrho_Upsilon}, $f^\star \in \Mcal_{\rho}^d$ implies $\phi^\star \in \Upsilon_{1/\rho,S}$ for some $S > 0$.
The assumption $\kappa_0 > 1/2$ is required only to apply Theorem~\ref{theoident1} and corresponds to the assumption $\rho < 2$. If the identifiability theorem held for a wider range of $\rho$, Theorem~\ref{th:fhat} would be valid for the corresponding range of $\kappa$ without any change in the proofs. The proof of Theorem~\ref{th:fhat} can be decomposed into the following steps.
\begin{enumerate}[(i)]
\item {\bf Consistency.} The fist step consists in proving that there exists a constant $c$ which depends on $\kappa$, $S$, $d$ and $\nu_{\text{est}}$ such that for all $n \geq 1$ and all $x > 0$, with probability at least $1-4\rme^{-x}$,
\begin{equation*}
\sup_{\phi \in \Upsilon_{\kappa,S}} |M_n(\phi) - M(\phi ; \nu_{\text{est}} | R^\star, Q^\star)| \\
\leq c\left(\sqrt{\frac{1}{n}} \vee \sqrt{\frac{x}{n}} \vee \frac{x}{n}\right)\eqsp.
\end{equation*}
This result is established in Lemma~\ref{lem_controle_deviations_M}. A key observation will be that for any $\nu \leq \nu_{\text{est}}$ and any $\phi$,
\begin{equation*}
M(\phi ; \nu | R^\star, Q^\star)
\leq M(\phi ; \nu_{\text{est}} | R^\star, Q^\star).
\end{equation*}
This is enough to establish that, for any $\nu \leq \nu_{\text{est}}$, all convergent subsequences of $(\hat \phi_{\kappa,n})_{n\geq 1}$ have limit $\Phi_{R^\star}$ in $\Lbf^2(\neighborhood{d_1}{\nu} \times \neighborhood{d_2}{\nu})$, provided $\Phi_{R^\star} \in \Upsilon_{\kappa,S}$.
Since $\Upsilon_{\kappa,S}$ is a compact subset of $\Lbf^2(\neighborhood{d_1}{\nu}\times \neighborhood{d_2}{\nu})$, this implies that $(\hat \phi_{\kappa,n})_{n\geq 1}$ is a consistent estimator of $\Phi_{R^\star}$ uniformly in $\kappa$ and $R^{\star}$.

\item {\bf Rates for the estimation of $\Phi_{R^\star}$.} Then, for fixed $\nu\in [(d+4/3)\rme /S,\nu_{\text{est}}]$, for $h$ in a neighborhood of 0 in $\Lbf^2(\neighborhood{d_1}{\nu} \times \neighborhood{d_2}{\nu})$, the risk $M(\Phi_{R^\star} + h ; \nu | R^\star, Q^\star)$ is lower bounded as follows:
\begin{equation}
\label{eq:lowerh4}
M(\Phi_{R^\star} + h ; \nu | R^\star, Q^\star) \geq c\|h\|_{\Lbf^2(\neighborhood{d_1}{\nu}\times \neighborhood{d_2}{\nu})}^{4}\eqsp,
\end{equation} 
where $c$ depends on $d$ and $\nu$. This result is established in Proposition~\ref{prop:compromis:M} in Appendix~\ref{sec:upperboundfourier} and is obtained by decomposing $M(\Phi_{R^\star} + h ; \nu | R^\star, Q^\star)$ into two terms, the first one involving the $\Lbf^2(\neighborhood{d_1}{\nu} \times \neighborhood{d_2}{\nu})$ norm of $h(\cdot,0)h(0,\cdot)$ and the second part involving the $\Lbf^2(\neighborhood{d_1}{\nu} \times \neighborhood{d_2}{\nu})$ norm of a linear term in $h$. The main challenge to 
prove equation~\eqref{eq:lowerh4} is 
to establish a lower bound of the first term and an upper bound of the second term for $h$ in a neighorhood of $0$ in $\Lbf^2(\neighborhood{d_1}{\nu} \times \neighborhood{d_2}{\nu})$. Obtaining these two bounds requires many technicalities and 
they need to be balanced sharply to establish \eqref{eq:lowerh4}. Then, we show in Proposition~\ref{prop:phihat} that there exist constants $c_1$, $c_2$ and $c_3$ which depend on $\kappa_0$, $\nu$, $S$, $d$ and $\E[\|\bfY\|^2]$ such that for all $x \geq 1$, for all $n \geq (1 \vee x c_1) / c_2$, with probability at least $1 - 4 \rme^{-x}$,
\begin{equation}
\label{eq:power4}
\sup_{\kappa \in [\kappa_0, \kappa^\star]}
\| \widehat{\phi}_{\kappa,n} - \Phi_{R^\star} \|_{\Lbf^2(\neighborhood{d_1}{\nu}\times \neighborhood{d_2}{\nu})} \leq c_3 \left(\sqrt{\frac{x}{n}} \vee \frac{x}{n} \right)^{1/4} \eqsp. 
\end{equation}

\item {\bf Rates for the estimation of $f^\star$.} Then, using assumption H\ref{assum:phistar}, the error term $\| \hat{f}_{\kappa,n} - f^\star \|_{\Lbf^2(\R^{d_1}\times \R^{d_2})}^2 $ is upper bounded based on the Fourier inversion \eqref{eq:def:fhat} as follows
\begin{equation*}
\| \hat{f}_{\kappa,n} - f^\star \|_{\Lbf^2(\R^{d_1}\times \R^{d_2})}^2 \leq \frac1{2\pi^2} \| T_{m_{\kappa,n}} \widehat{\phi}_{\kappa,n} - \Phi_{R^\star} \|_{\Lbf^2(\neighborhood{d_1}{\omega_{\kappa,n}}\times\neighborhood{d_2}{\omega_{\kappa,n}})}^2 
		+ \frac1{2\pi^2} \frac{C}{(1 + \omega_{\kappa,n}^2)^\beta}\eqsp.
\end{equation*}
This allows to establish Theorem~\ref{th:fhat} by controlling the error between $T_{m_{\kappa,n}} \widehat{\phi}_{\kappa,n}$ and the truncation of $\phi^\star$ in $\C_{m_{\kappa,n}}[X_1,\ldots,X_d]$ using Legendre polynomials, and the distance between functions in $ \Upsilon_{\kappa,S}$ and their truncations in $\C_{m_{\kappa,n}}[X_1,\ldots,X_d]$.
\end{enumerate}

{\bf Comments on the practical computation of the estimator.} 
In practice computing the minimum over the infinite dimensional set defined in \eqref{eq_def_phihat} requires to introduce a truncation parameter. In other words, instead of minimizing $M_n$ over all elements $\phi$ of $\Upsilon_{\kappa,S} \cap \Hcal$, we would minimize it over all $T_m \phi$, where $m$ is the so-called truncation parameter. This truncation has no impact on the result proved in Theorem~\ref{th:fhat}, i.e. on the rates of convergence derived in this paper, as long as this truncation parameter is chosen sufficiently large with respect to $m_{\kappa,n}$ to obtain the rates for the estimation of $\Phi_{R^\star}$: as observed just after equation~\eqref{eq_def_phihat}, the resulting is an approximate minimizer of $M_n$. In the case where this new truncation parameter is at least greater than $2m_{\kappa,n}$, this allows in \eqref{eq:power4} to control the additional bias term and to balance it with the term $(\sqrt{x/n} \vee x/n)^{1/4}$.
Although the estimator may be adapted to allow a practical computation, this does not ensure a stable and numerically efficient result in real life learning frameworks. Moreover, designing a set $\Hcal$ that is closed in $\Lbf^{2}([-\nu_{\text{est}},\nu_{\text{est}}]^d)$ and whose elements satisfy H\ref{assum:} that is in addition rich enough for Theorem~\ref{th:fhat} to hold for a wide choice of $R^\star$ is complex and would be a significant practical contribution. Designing an efficient and stable implementation of the proposed algorithm is a challenge on its own and is left for future works, as described in Section~\ref{sec:further}.
The focus of this paper is to derive theoretical properties of the deconvolution estimator without any assumption on the noise distribution.

\subsection{Adaptivity in \texorpdfstring{$\kappa$}{kappa}}
\label{subsec:adaptive}

In Section \ref{subsec:rates}, we studied estimators built using the tail parameter $\kappa$. Unfortunately this tail parameter is typically unknown in practice.
We now propose a model selection data-driven procedure to choose $\kappa$, and we prove that the resulting estimator converges to the rate corresponding to the largest $\kappa$ such that $\Phi_{R^\star} \in \Upsilon_{\kappa,S}$ for some $S>0$.

Our strategy is based on Goldenshluger and Lepski's methodology~(\cite{goldenshluger2008, goldenshluger2013}, see also~\cite{BLR16Lepski} for a very clear introduction). Like in all model selection problems, the core idea is to perform a careful bias-variance tradeoff to select the right $\kappa$. While a variance bound is readily available thanks to Theorem~\ref{th:fhat}, the bias is not so easily accessible. Goldenshluger and Lepski's methodology give a way to compute a proxy of the bias, thus allowing selection of a proper $\hat{\kappa}$.

The variance bound (which can also be seen as a penalty term) will be taken as
\begin{equation*}
\sigma_{n}(\kappa') = c_\sigma \left( \frac{\log n}{\log \log n} \right)^{-\kappa' \beta}
\end{equation*}
for all $\kappa' \in [\kappa_0, 1]$, for some constant $c_\sigma>0$. While the selection procedure works as soon as this constant $c_\sigma$ is large enough, the exact threshold depends on the true parameters. This is a usual problem of selection procedures based on penalization: the penalty is typically known only up to a constant. Heuristics such as the slope heuristics or dimension jump heuristics have been proposed to solve this issue and proved to work in several settings, see~\cite{baudry2012slope} and references therein.

The proxy for the bias is defined for all $\kappa' \in [\kappa_0,1]$ as
\begin{equation*}
A_n(\kappa') = 0 \vee \sup_{\kappa'' \in [\kappa_0,\kappa']} \left\{ \| \hat{f}_{\kappa'',n} - \hat{f}_{\kappa',n} \|_{\Lbf^2(\R^{d_1} \times \R^{d_2})} - \sigma_{n}(\kappa'') \right\}\eqsp.
\end{equation*}
Finally, the tail parameter is selected as
\begin{equation*}
\hat{\kappa}_{n} \in \argmin_{\kappa' \in [\kappa_0, 1]} \{ A_n(\kappa') + \sigma_{n}(\kappa')\}\eqsp.
\end{equation*}
When $\Phi_{R^\star} \in \Upsilon_{\kappa,S}$, $\hat{f}_{\hat{\kappa}_{n},n}$ reaches the same rate of convergence as $\hat{f}_{\kappa,n}$ for the integrated square risk:

\begin{theo}
\label{theo:adaptrate}
For all $\kappa_0 > 1/2$, $S>0$, $\beta>0$, $c_Q>0$, $c_\nu>0$ and $\cbeta>0$, there exists $c_\sigma>0$ such that for all $\nu\in [(d+4/3)\rme /S,\nu_{\text{est}}]$, if $\sigma_n(\kappa') \geq c_\sigma (\log n / \log \log n)^{-\kappa' \beta}$ for all $\kappa' \in [\kappa_0,1]$,
\begin{equation*}
\label{eq:adaptrate}
\limsup_{n \rightarrow +\infty}
\sup_{\kappa \in [\kappa_0, 1]}
\underset{Q^\star \in \Qbf(\nu,c_\nu,c_Q)}{\sup_{R^\star \,:\, \Phi_{R^\star} \in \Psi(\kappa,S,\beta,\cbeta) \cap \Hcal}}
\left( \frac{\log n}{\log \log n} \right)^{2 \kappa \beta} \E_{R^\star,Q^\star} \left[
	 \| \hat{f}_{\hat{\kappa}_n,n} - f^\star \|_{\Lbf^2(\R^{d_1} \times \R^{d_2})}^2 \right]
	< +\infty \eqsp,
\end{equation*}
where $\Hcal$ is introduced in the definition of $\widehat \phi_{\kappa,n}$, see \eqref{eq_def_phihat}.
\end{theo}
The proof of Theorem \ref{theo:adaptrate} is detailed in Section \ref{sec:theoadapt}. It is a consequence of deviation upper bounds developed for proving Theorem~\ref{th:fhat} showing that if $\Phi_{R^\star} \in \Upsilon_{\kappa,S}$, with probability at least $1 - 4/n$, for all $\kappa' \in [\kappa_0,\kappa]$, $\| \hat{f}_{\kappa',n} - f^\star \|_{\Lbf^2(\R^{d_1} \times \R^{d_2})} \leq \sigma_{n}(\kappa')$.

\section{Lower bounds}
\label{sec:lower}

\input{lower.tex}

\section{Conclusion and future works}
\label{sec:further}

Recently, in \cite{belomestny:goldenshluger:2019}, the authors summarized the standard assumptions on the noise distribution and their implications on the minimax risk of the estimator of the signal distribution. In particular, they pointed out that obtaining assumptions under which standard rates of convergence can be established when the Fourier transform of the noise can vanish have not received satisfactory solutions in the existing literature. 
In the direction of weakening the assumptions on the noise, such limitation has been completely overcome
 in this paper, where we propose a general optimal rate which depends on the lightness of the tail distribution of the signal which extends the work of \cite{meister2007deconvolving} to multivariate signals without the compact support assumption and with no prior knowledge on the noise distribution. The optimal rate of convergence in our setting does not depend at all on the unknown noise. In another direction, one could try to find if it is possible, in the context of unknown noise, to recover noise dependent minimax risk by restricting the set of possible unknown noises. One way could be to make in our methodology $\nu=\nu_{\text{est}}$ go to infinity and to study the square integrated risk with $c_{\nu}$ having a precise decreasing behavior. This can not be directly obtained by the proofs in this work in which we use the fact that $\nu$ is finite to derive equation~\eqref{eq_consistence_unif} which is itself a basic stone to establish Proposition~\ref{prop:phihat}.
 
There are numerous avenues for future works. We specifically chose to focus on the theoretical properties of the deconvolution estimator obtained from the risk function $M_n$ without assumption on the noise distribution, leaving mainly open the question of designing efficient numerical solutions. Recently, in this unknown noise setting, \cite{gassiat:lecorff:lehericy:2019} provided two algorithms to compute nonparametric estimators of the law of the hidden process in a general state space translation model, i.e. when the hidden signal is a Markov chain. More thorough and scalable practical solutions remain to be developed. Although the estimator proposed in this paper enjoys interesting theoretical properties, designing a stable and numerically efficient algorithm remains mainly an open problem.

In a more applied perspective, the recent emergence of blind spot neural networks such as \cite{noise2self:2019} or \cite{noise2void:2019} represent a breakthrough in the field of blind image denoising. In these papers, the authors manage to improve state-of-the-art performance in signal prediction using mainly local (spatially) dependencies on the signal and assuming that the noise components are independent. Our results which in addition do not require any assumption on the noise are likely to provide new architectures or new loss functions to extend such works.

We are particularly interested in applying our results to widespread models such as noisy independent component analysis and nonlinear component analysis, see for instance \cite{khemakhem:2020}. As mentionned in \cite{pfister:2019}, a wide range of applications require to design source separation techniques to deal with grouped data and structured signals. The identifiability of such a group structured ICA is likely to rely on specific assumptions similar to the one derived in our paper which should provide new insights to derive numerical procedures. Additive index models studied in \cite{lin2007identifiability,yuan:2011} could also benefit from this work to weaken the assumptions on the signal and on the functions involved in the mixture defining the observation.

As underlined in Section \ref{sec:examples}, submodels may be identified in the larger general deconvolution model studied in this paper. It could be of interest to study statistical testing of such structured submodels, for instance using the minimax non parametric hypothesis testing theory.

In another line of works referred to as topological data analysis (TDA), see \cite{CM17}, \cite{chazal2017robust}, the aim is at providing mathematical results and methods to infer, analyze and exploit the complex topological and geometric structures underlying data. Despite fruitful developments, geometric inference from noisy data remains a theoretical and practical widely open problem. Although they appear to be concentrated around geometric shapes, real data are often corrupted by noise and outliers. Quantifying and distinguishing topological/geometric noise, which is difficult to model or unknown, from topological/geometric signal to infer relevant geometric structures is a subtle problem. Our paper is likely to apply to multidimensional signals supported on manifolds and opens the way to find strategies to infer relevant topological and geometric information of signals additively corrupted with totally unknown noise. One way to proceed is to use the distance to measure strategy developed in \cite{CDSM11} which shows that it is possible to build robust methods to estimate geometric and topological parameters of supports of probability distribution from perturbed versions of it in Wasserstein's metric. 
This is the subject of an ongoing research project.

\appendix

\section{Proof of Theorem~\ref{th:fhat}}
\label{proof:theo:conv}

For any discrete set $A$, $|A|$ denotes the number of elements in $A$. For any matrix $B$, $\|B\|_F$ denotes the Frobenius norm of $B$ and $B^\top$ the transpose matrix of $B$. 
\subsection{Uniform consistency}

The risk function at $\widehat \phi_{\kappa,n}$ satisfies, by definition, for all $R^\star$ and all $Q^\star$ in $\Upsilon_{\kappa,S} \cap \Hcal$,
\begin{align}
\nonumber
M(\widehat \phi_{\kappa,n} ; \nu_{\text{est}} | R^\star, Q^\star)
	&\leq M_n(\widehat \phi_{\kappa,n}) + \!\!\!\sup_{\phi \in \Upsilon_{\kappa,S} } \!\!\!\left|M_n(\phi) - M(\phi ; \nu_{\text{est}} | R^\star, Q^\star) \right|\eqsp,\\ \nonumber
	&\leq M_n(\Phi_{R^\star}) + \!\!\!\sup_{\phi \in \Upsilon_{\kappa,S}} \!\!\!\left|M_n(\phi) - M(\phi ; \nu_{\text{est}} | R^\star, Q^\star) \right|\eqsp, \\
	\label{eq_maj_M_hatphi}
	&\leq \left|M_n(\Phi_{R^\star}) - M(\Phi_{R^\star} ; \nu_{\text{est}} | R^\star, Q^\star) \right| + \!\!\!\!\sup_{\phi \in \Upsilon_{\kappa,S} } \!\!\!\!\left|M_n(\phi) - M(\phi ; \nu_{\text{est}} | R^\star, Q^\star) \right|\eqsp.
\end{align}
Lemma~\ref{lem_controle_deviations_M} provides a control on the deviation $|M_n(\phi) - M(\phi ; \nu_{\text{est}} | R^\star, Q^\star)|$ for $\phi \in \Upsilon_{\kappa,S}$.
\begin{lem}
\label{lem_controle_deviations_M}
There exists a numerical constant $c_M$ and a constant $x_0$ depending only on $d$ (for instance $x_0 = \sup_{\kappa > 0} (\frac{d+4/3}{\kappa})^\kappa$) such that the following holds. For all $C' > 0$, $n \geq 1$, $x > 0$, and probability measures $R^\star$ and $Q^\star$ on $\R^d$ such that $\E_{R^\star,Q^\star}[\|\bfY\|^2] \leq C'$, with probability at least $1-4e^{-x}$ under $\po_{R^\star,Q^\star}$, for all $\kappa>0$ and $S>0$,
\begin{multline*}
\sup_{\phi \in \Upsilon_{\kappa,S}}
|M_n(\phi) - M(\phi;\nu_{\text{est}} | R^\star,Q^\star)| \\
\leq c_M \nu_{\text{est}}^d (S\nu_{\text{est}} \vee x_0)^{4\frac{d+1}{\kappa} } \exp\left( 4\kappa (S\nu_{\text{est}} \vee x_0)^{1/\kappa} \right)
	\left[d \sqrt{\frac{1 \vee \nu_{\text{est}}^2 d C'}{n}} \vee \sqrt{\frac{x}{n}} \vee \frac{x}{n}\right]\eqsp.
\end{multline*}
In particular, for all $\kappa_0 \in (0,1]$, $S>0$ and $C'>0$, there exists a constant $c$ such that for all $n \geq 1$ and $x>0$, for all $\nu \leq \nu_{\text{est}}$,
\begin{equation}
\label{eq_vitesse_ecart_Mn_M}
\underset{Q^\star \,:\, \E_{R^\star,Q^\star}[\|\bfY\|^2] \leq C'}{\sup_{R^\star \,:\, \Phi_{R^\star} \in \Upsilon_{\kappa_0,S}}}
	\po_{R^\star,Q^\star} \left(
		\sup_{\kappa \in [\kappa_0,1]}
		M(\widehat \phi_{\kappa,n} ; \nu | R^\star, Q^\star) \geq c \left( \sqrt{\frac{x}{n}} \vee \frac{x}{n} \right) \right)
	\leq 4 e^{-x}\eqsp.
\end{equation}
(Even though it is not visible in the notations, $S$ is involved in the definition of $\hat{\phi}_{\kappa,n}$.)
\end{lem}
\begin{proof}
The proof of the first inequality is postponed to Section~\ref{sec_proof_lem_controle_deviations_M}  in the supplementary material.

The second follows from taking the supremum over all $\kappa \in [\kappa_0,1]$ first together with equation~\eqref{eq_maj_M_hatphi} and the key observation that for all $\nu \leq \nu_{\text{est}}$,
\begin{equation*}
M(\widehat \phi_{\kappa,n} ; \nu | R^\star, Q^\star)
\leq M(\widehat \phi_{\kappa,n} ; \nu_{\text{est}} | R^\star, Q^\star).
\end{equation*}
\end{proof}

Since $\sup_{R^\star : \Phi_{R^\star} \in \Upsilon_{\kappa,S}} \E_{R^\star}[\|\bfX\|^2]$ is bounded by a constant that depends only on $\kappa$ and $S$, assuming $\E_{Q^\star}[\|\varepsilon\|^2] \leq C''$ and $\Phi_{R^\star} \in \Upsilon_{\kappa,S}$ ensures $\E_{R^\star,Q^\star}[\|\bfY\|^2] \leq C'$ for some constant $C'$ depending on $\kappa$, $S$ and $C''$. Thus, assuming $\E_{Q^\star}[\|\varepsilon\|^2] \leq C''$ and $R^\star \in \Upsilon_{\kappa,S}$ is enough to apply the above lemma.

For any $\nu >0$, by the proof of Theorem~\ref{theoident1} and Lemma~\ref{lem_lien_Mrho_Upsilon}, if $\Phi_{R^\star} \in \Upsilon_{\kappa,S} \cap \Hcal$, the only zero of the contrast function $\phi \mapsto M(\phi; \nu | R^\star,Q^\star)$ on $\Upsilon_{\kappa,S} \cap \Hcal$ is $\phi = \Phi_{R^\star}$ as soon as $1/\kappa < 2$ since all functions in $\Hcal$ satisfy H\ref{assum:}.
Moreover, the mapping $(\phi,\Phi_{R^\star},\Phi_{Q^\star}) \in \Lbf^\infty(\neighborhood{d_1}{\nu} \times \neighborhood{d_2}{\nu_\text{est}})^3 \mapsto M(\phi;\nu|R^\star,Q^\star)$ is continuous and for all $\kappa>0$, $S>0$ and $C'>0$, the sets $\Upsilon_{\kappa,S}$ and $\{\Phi_Q : Q \text{ s.t. } \E_Q[\|\varepsilon\|^2] \leq C'\}$ are compact in $\Lbf^\infty(\neighborhood{d_1}{\nu} \times \neighborhood{d_2}{\nu})$ by Arzelà–Ascoli's theorem (the second derivative of $\Phi_Q$ is bounded by the second moment of $Q$ and likewise for $\Phi_R$, so these sets are uniformly equicontinuous and all of their elements have value 1 at zero).
Thus, for all $\kappa>1/2$, $S>0$, $C'>0$ and $\eta>0$,
\begin{equation*}
\inf_{\substack{
	\phi,\Phi_{R^\star} \in \Upsilon_{\kappa,S} \cap \Hcal \\
	\|\phi - \Phi_{R^\star}\|_{\Lbf^2(\neighborhood{d}{\nu})} > \eta \\
	Q^\star \,:\, \E_{Q^\star}[\|\varepsilon\|^2] \leq C'	
	}} \hspace{-.5cm} M(\phi;\nu | R^\star,Q^\star) > 0 \eqsp.
\end{equation*}
Fix now $\eta \in (0,\nu_{\text{est}}]$. This equation and Lemma~\ref{lem_controle_deviations_M} together with the fact that the family $(\Upsilon_{\kappa,S})_\kappa$ is nonincreasing in $\kappa$ ensure that for all $\kappa_0 \in (1/2,1]$ and all $S>0$, $C'>0$ and $\eta>0$, there exists $c>0$ such that for all $n \geq 1$, $x \in (0, cn]$ and $\kappa \in [\kappa_0,1]$,
\begin{equation}
\label{eq_consistence_unif}
\underset{Q^\star \,:\, \E_{R^\star,Q^\star}[\|\bfY\|^2] \leq C'}{\sup_{R^\star \,:\, \Phi_{R^\star} \in \Upsilon_{\kappa,S} \cap \Hcal}}
\po_{R^\star,Q^\star} \left( \sup_{\kappa' \in [\kappa_0,\kappa]} \|\hat{\phi}_{\kappa',n} - \Phi_{R^\star}\|_{\Lbf^2(\neighborhood{d}{\nu})} \geq \eta \right) \leq 4e^{-x} \eqsp.
\end{equation}
In particular, the family of estimators $(\hat{\phi}_{\kappa,n})_{\kappa}$ is $\Lbf^2(\neighborhood{d}{\nu})$-consistent uniformly in $\kappa \in [\kappa_0,1]$, and uniformly in the true parameters $R^\star$ and $Q^\star$.

\subsection{Upper bound for the estimator of the Fourier transform of the signal distribution}
\label{sec:upperboundfourier}

Recall, for all bounded and measurable functions $h: \neighborhood{d_1}{\nu}\times \neighborhood{d_2}{\nu} \to \C$, for any $\nu >0$ and any probability measures $R^\star$ and $Q^\star$ on $\R^d$,
\begin{multline*}
M(\Phi_{R^\star} + h ; \nu | R^\star,Q^\star) \\
	= \int_{\neighborhood{d_1}{\nu}\times \neighborhood{d_2}{\nu}} | h(t_1,t_2) \Phi_{R^\star}(t_1,0) \Phi_{R^\star}(0,t_2)
		- \Phi_{R^\star}(t_1,t_2) h(t_1,0)\Phi_{R^\star}(0,t_2) \\
		\hspace{2cm} - \Phi_{R^\star}(t_1,t_2) \Phi_{R^\star} (t_1,0)h(0,t_2)
		- \Phi_{R^\star}(t_1,t_2) h(t_1,0) h(0,t_2) |^2 \\
	| \Phi_{Q^{\star,(1)}}(t_1) \Phi_{Q^{\star,(2)}}(t_2) |^2 \rmd t_1 \rmd t_2\eqsp. 
\end{multline*}
Recall that for all $Q \in \Qbf(\nu,c_\nu,C')$, $\inf_{\neighborhood{d_1}{\nu}} |\Phi_{Q^{(1)}}|
		\wedge \inf_{\neighborhood{d_2}{\nu}} |\Phi_{Q^{(2)}}| \geq c_\nu$.
Using that for all $(a,b)\in \R$, $(a-b)^2 \geq a^2/2 - b^2$ and $\|\Phi_{Q^{\star,(1)}}\|_\infty = \|\Phi_{Q^{\star,(2)}}\|_\infty = \| \Phi_{R^\star} \|_\infty = 1$ yields for all probability measures $R^\star$ and $Q^\star$ on $\R^d$ such that $Q^\star \in \Qbf(\nu,c_\nu,C')$,
\begin{equation}
\label{eq:lowerbound:M}
M(\Phi_{R^\star}+ h ; \nu | R^\star,Q^\star) \geq c_\nu^4 M^\text{lin}(h,\Phi_{R^\star}; \nu) / 2 - c_\nu^4 \| h(\cdot,0) h(0,\cdot) \|_{\Lbf^2(\neighborhood{d_1}{\nu}\times \neighborhood{d_2}{\nu})}^2\eqsp,
\end{equation}
where 
\begin{multline}
M^\text{lin}(h,\phi; \nu) = \int_{\neighborhood{d_1}{\nu}\times \neighborhood{d_2}{\nu}} | h(t_1,t_2) \phi(t_1,0) \phi(0,t_2)- \phi(t_1,t_2) h(t_1,0) \phi(0,t_2) \label{eq:def:mlin}\\ 
- \phi(t_1,t_2) \phi(t_1,0) h(0,t_2) |^2 \rmd t_1 \rmd t_2\eqsp.
\end{multline}
Section~\ref{sec:proof:compromis:M} provides an upper bound for $\| h(\cdot,0) h(0,\cdot) \|_{\Lbf^2(\neighborhood{d_1}{\nu}\times \neighborhood{d_2}{\nu})}^2$ and a lower bound for $M^\text{lin}(h,\Phi_{R^\star}; \nu)$ which allows to establish the lower bound given in Proposition~\ref{prop:compromis:M}.

\begin{prop}
\label{prop:compromis:M}
For all $S > 0$ and $\kappa \in (0,1]$, there exists $\eta > 0$ depending on $d$, $\kappa$, $\nu_\text{est}$ and $S$ and $c_M > 0$ depending only on $d$ such that for all $\nu \in [(d+4/3)\rme/S, \nu_\text{est}]$ and all $h$ such that $\|h\|_{\Lbf^2(\neighborhood{d_1}{\nu}\times \neighborhood{d_2}{\nu})} \leq \eta$, the risk satisfies
\begin{equation*}
\underset{Q^\star \in \Qbf(\nu,c_\nu,+\infty)}{\inf_{R^\star \,:\, \Phi_{R^\star} \in \Upsilon_{\kappa,S}}}
M(\Phi_{R^\star} + h ; \nu | R^\star,Q^\star)
	\geq c_M c^4_\nu \|h\|_{\Lbf^2(\neighborhood{d_1}{\nu}\times \neighborhood{d_2}{\nu})}^{4}
 		\eqsp.
\end{equation*}
\end{prop}

\begin{proof}
The proof is postponed to Section~\ref{sec:proof:compromis:M}.
\end{proof}

Using the above proposition for $\kappa = \kappa_0$ together with equation~\eqref{eq_consistence_unif} is enough to establish Proposition~\ref{prop:phihat}.

\begin{prop}
\label{prop:phihat}
For all $\kappa_0 \in (1/2,1]$, $S>0$, $c_\nu>0$ and $c_Q>0$, there exists a constant $c>0$ such that for all $\nu \in [(d+4/3)\rme/S,\nu_\text{est}]$, $n \geq 1$, $x \in (0,cn]$ and $\kappa \in [\kappa_0,1]$,
\begin{equation}
\label{eq:upperbound:fourier}
\underset{Q^\star \in \Qbf(\nu,c_\nu,c_Q)}{\inf_{R^\star \,:\, \Phi_{R^\star} \in \Upsilon_{\kappa,S} \cap \Hcal}}
\po_{R^\star,Q^\star} \left( \sup_{\kappa' \in [\kappa_0,\kappa]} \|\hat{\phi}_{\kappa',n} - \Phi_{R^\star}\|_{\Lbf^2(\neighborhood{d}{\nu})} \leq c\left(\sqrt{\frac{x}{n}} \vee \frac{x}{n} \right)^{1/4} \right) \geq 1 - 4e^{-x} \eqsp.
\end{equation}
\end{prop}

\subsection{Upper bound for the estimator of the density of the signal distribution}

Let $\kappa' \in (0,1]$.
Assume H\ref{assum:phistar} holds for the constants $\beta, c_{\beta}$, then by definition of $\hat{f}_{\kappa',n}$ together with Plancherel's theorem,
\begin{align*}
\| \hat{f}_{\kappa',n} &- f^\star \|_{\Lbf^2(\R^{d_1}\times \R^{d_2})}^2 \\
	={} & \frac1{(4\pi^2)^d} \int_{\R^{d_1}\times \R^{d_2}} \left| \mathds1_{t \in \neighborhood{d_1}{\omega_{\kappa',n}}\times\neighborhood{d_2}{\omega_{\kappa',n}}}T_{m_{\kappa',n}} \widehat{\phi}_{\kappa',n}(t) - \Phi_{R^\star}(t) \right|^2 \rmd t\eqsp, \\
	={} & \frac1{(4\pi^2)^d} \int_{\neighborhood{d_1}{\omega_{\kappa',n}}\times\neighborhood{d_2}{\omega_{\kappa',n}}} \left| T_{m_{\kappa',n}}\widehat{\phi}_{\kappa',n}(t) - \Phi_{R^\star}(t) \right|^2 \rmd t \\
		&\quad + \frac1{(4\pi^2)^d} \int_{(\R^{d_1}\times \R^{d_2}) \setminus (\neighborhood{d_1}{\omega_{\kappa',n}}\times\neighborhood{d_2}{\omega_{\kappa',n}})} \left| \Phi_{R^\star}(t) \right|^2 \rmd t\eqsp, \\
	\leq{} & \frac1{(4\pi^2)^d} \|T_{m_{\kappa',n}} \widehat{\phi}_{\kappa',n} - \Phi_{R^\star} \|_{\Lbf^2(\neighborhood{d_1}{\omega_{\kappa',n}}\times\neighborhood{d_2}{\omega_{\kappa',n}})}^2 
		+ \frac1{(4\pi^2)^d} \frac{c_{\beta}}{(1 + \omega_{\kappa',n}^2)^\beta}\eqsp.
\end{align*}
Let $\nu \in [(d+4/3)\rme /S,\nu_{\text{est}}]$ be fixed in the remaining of the proof.
For all $i \geq 0$, let $P_i$ be the $i$-th Legendre polynomial and 
\begin{equation}
\label{eq:norm:legendre}
P^\text{norm}_i = (i+1/2)^{1/2} \nu^{-1/2} P_i(X/\nu)
\end{equation} 
the normalized $i$-th Legendre polynomial on $[-\nu,\nu]$. For all positive integer $p$, define the orthonormal basis $(\Pbf^\text{norm}_i)_{i \in \N^p}$ of $\C[X_1, \dots, X_p]$ (seen as a subset of $\Lbf^2(\neighborhood{p}{\nu})$), where for all $i \in \N^p$, 
\begin{equation}
\label{eq:prod:legendre}
\Pbf^\text{norm}_i(X_1,\ldots,X_p) = (P^\text{norm}_{i_1} \otimes \dots \otimes P^\text{norm}_{i_{p}})(X_1, \dots, X_{p}) = \prod_{a=1}^{p} P^\text{norm}_{i_a}(X_a)\eqsp.
\end{equation}
Since $T_{m_{\kappa',n}}\widehat{\phi}_{\kappa',n}$ and $T_{m_{\kappa',n}} \Phi_{R^\star} $ are in $\C_{m_{\kappa',n}}[X_1,\ldots,X_d]$, there exists a sequence $(a_i)_{i\in\N^d}$ such that $a_i=0$ if $\|i\|_1>m_{\kappa',n}$ and $T_{m_{\kappa',n}}\widehat{\phi}_{\kappa',n} - T_{m_{\kappa',n}} \Phi_{R^\star} = \sum_{i \in \N^d} a_i\Pbf^\text{norm}_i(X)$, where $\Pbf^\text{norm}_i$ is defined in \eqref{eq:prod:legendre}. By properties of the Legendre polynomials, see \cite[page~11]{meister2007deconvolving}, for all $x\in\R$, $|P_i(x)| \leq (2|x|+2)^i$ so that $|P^\text{norm}_i(x)| \leq ((2i+1)/(2\nu))^{1/2}(2|x/\nu|+2)^i$. Therefore, for all $i\in\N$,
\begin{equation*}
\int_{-\omega_{\kappa',n}}^{\omega_{\kappa',n}}|P^\text{norm}_i(x)|^2 \rmd x \leq \frac12\left(2 + 2\frac{\omega_{\kappa',n}}{\nu}\right)^{2i+1}\eqsp,
\end{equation*}
and by Cauchy-Schwarz inequality,
\begin{align*}
\| T_{m_{\kappa',n}}\widehat{\phi}_{\kappa',n} -& T_{m_{\kappa',n}} \Phi_{R^\star} \|_{\Lbf^2(\neighborhood{d_1}{\omega_{\kappa',n}}\times\neighborhood{d_2}{\omega_{\kappa',n}})}^2\\
	& \leq \left(\sum_{i\in\N^d, \|i\|_1\leq m_{\kappa',n}} \prod_{a=1}^d \int_{-\omega_{\kappa',n}}^{\omega_{\kappa',n}} |P^\text{norm}_{i_a}(x)|^2 \rmd x\right)\left(\sum_{i\in\N^d}|a_i|^2\right)\eqsp,\\
	&\leq 2^{-d} (m_{\kappa',n}+1)^d \left(2 + 2\frac{\omega_{\kappa',n}}{\nu} \right)^{2m_{\kappa',n} + d} \| T_{m_{\kappa',n}}\widehat{\phi}_{\kappa',n} - T_{m_{\kappa',n}} \Phi_{R^\star} \|_{\Lbf^2(\neighborhood{d_1}{\nu}\times\neighborhood{d_2}{\nu})}^2\eqsp,\\
&\leq m_{\kappa',n}^d \left(2 + 2\frac{\omega_{\kappa',n}}{\nu} \right)^{2m_{\kappa',n} + d} \| T_{m_{\kappa',n}}\widehat{\phi}_{\kappa',n} - T_{m_{\kappa',n}} \Phi_{R^\star} \|_{\Lbf^2(\neighborhood{d_1}{\nu}\times\neighborhood{d_2}{\nu})}^2\eqsp.
\end{align*}
Since $\Upsilon_{\kappa,S} \subset \Upsilon_{\kappa',S}$ when $\kappa' \leq \kappa$, by Lemma~\ref{lem_controle_constantes} and Lemma~\ref{lem_troncature}  in the the supplementary material, when $\Phi_{R^\star} \in \Upsilon_{\kappa,S}$ and $\kappa' \leq \kappa$,
\begin{align*}
\| \Phi_{R^\star} - T_{m_{\kappa',n}} \Phi_{R^\star} &\|_{\Lbf^2(\neighborhood{d_1}{\omega_{\kappa',n}}\times\neighborhood{d_2}{\omega_{\kappa',n}})}^2\\
	&\leq (8\omega_{\kappa',n})^d (S \omega_{\kappa',n})^{2m_{\kappa',n}} m_{\kappa',n}^{-2\kappa' m_{\kappa',n}+2d}f_{\kappa'}(S\omega_{\kappa',n})^2 \\
	&\leq (8\omega_{\kappa',n})^d (S \omega_{\kappa',n})^{2m_{\kappa',n}} m_{\kappa',n}^{-2\kappa' m_{\kappa',n}+2d} \times 6 (S\omega_{\kappa',n})^{2/\kappa'} \exp(2\kappa' (S\omega_{\kappa',n})^{1/\kappa'}) \eqsp,
\end{align*}
and both $ \| \widehat{\phi}_{\kappa',n} - T_{m_{\kappa',n}}\widehat{\phi}_{\kappa',n} \|_{\Lbf^2(\neighborhood{d_1}{\nu}\times\neighborhood{d_2}{\nu})}^2$ and $ \|\Phi_{R^\star} - T_{m_{\kappa',n}} \Phi_{R^\star} \|_{\Lbf^2(\neighborhood{d_1}{\nu}\times\neighborhood{d_2}{\nu})}^2$ are upper bounded by
\begin{equation*}
(8\nu)^d (S\nu)^{2m_{\kappa',n}} m_{\kappa',n}^{-2\kappa' m_{\kappa',n}+2d} \times 6 (S\nu)^{2/\kappa'} \exp(2\kappa' (S\nu)^{1/\kappa'})
\eqsp.
\end{equation*}
Thus, Proposition~\ref{prop:phihat} shows that for all $\kappa_0 \in (1/2,1]$, $S > 0$, $\beta > 0$, $c_\nu > 0$, $c_Q > 0$ and $c_\psi > 0$, there exist $c > 0$ and $c' > 0$ such that for all $\nu \in [(d+4/3)\rme/S, \nu_\text{est}]$, $n \geq 1$, $x \in (0,cn]$ and $\kappa \in [\kappa_0,1]$,
\begin{multline*}
\underset{Q^\star \in \Qbf(\nu,c_\nu,c_Q)}{\inf_{R^\star \,:\, \Phi_{R^\star} \in \Psi(\kappa,S,\beta,c_\psi) \cap \Hcal}}
\po_{R^\star,Q^\star} \Bigg(
	\forall \kappa' \in [\kappa_0,\kappa], \ \\
	\| \hat{f}_{\kappa',n} - f^\star \|_{\Lbf^2(\R^{d_1} \times \R^{d_2})}^2
	\leq c' \max\Bigg\{(1+\omega_{\kappa',n}^2)^{-\beta},
		m_{\kappa',n}^{-2\kappa' m_{\kappa',n}+2d} \omega_{\kappa',n}^{2m_{\kappa',n} + d + 2/\kappa'} \rme^{2\kappa' (S\omega_{\kappa',n})^{1/\kappa'}}, \\
		m_{\kappa',n}^d \left(\frac{\omega_{\kappa',n}}{\nu}\right)^{2m_{\kappa',n}+d}
		\left[ 2 m_{\kappa',n}^{-2\kappa' m_{\kappa',n}+2d} \nu^{2m_{\kappa',n} + d + 2/\kappa'} \rme^{2\kappa' (S\nu)^{1/\kappa'}}+ \left(\sqrt{\frac{x}{n}} \vee \frac{x}{n}\right)^{1/2}\right]\Bigg\}
\Bigg) \geq 1 - 4e^{-x} \eqsp.
\end{multline*}
Since $\omega_{\kappa,n}$ is chosen of the form $\omega_{\kappa,n} = c_{\kappa} m_{\kappa,n}^\kappa / S$ with $c_\kappa \in (0,1]$, and since by assumption $S \nu \geq x_0 \vee u_0 \geq 1$ where $x_0$ and $u_0$ are defined in Lemma~\ref{lem_controle_constantes} in the the supplementary material, there exists $c'>0$ such that the event in the above equation may be rewritten as follows: for all $\kappa' \in [\kappa_0,\kappa]$,
\begin{align*}
\| \hat{f}_{\kappa',n} & - f^\star \|_{\Lbf^2(\R^{d_1} \times \R^{d_2})}^2\\
&\leq c' \max\Bigg\{(1 + m_{\kappa',n}^{2\kappa'})^{-\beta},
	m_{\kappa',n}^{3d+2} c_{\kappa'}^{2m_{\kappa',n}} \rme^{2 \kappa' m_{\kappa',n}}, \\
	&\hspace{2cm} (2\kappa' m_{\kappa',n})^{2\kappa' m_{\kappa',n}} (2\kappa')^{-2\kappa' m_{\kappa',n}} m_{\kappa',n}^{2d} c_{\kappa'}^{2m_{\kappa',n}} \\
	&\hspace{2cm} \quad \times \left[ 2 m_{\kappa',n}^{-2\kappa' m_{\kappa',n}+2d} \nu^{2m_{\kappa',n} + d + 2/\kappa'} \rme^{2\kappa' (S\nu)^{1/\kappa'}}+ \left(\sqrt{\frac{x}{n}} \vee \frac{x}{n} \right)^{1/2} \right]\Bigg\}\eqsp, \\
&\leq c' \max\Bigg\{2^{-\beta} m_{\kappa',n}^{-2 \kappa' \beta},
	\rme^{\left[(3d+2) + 2 \log c_\kappa' + 2\right] m_{\kappa',n}}, 2(c_{\kappa'}\nu)^{2m_{\kappa',n}}\nu^{d + 2/\kappa'} \rme^{2\kappa' (S\nu)^{1/\kappa'}}\\
	&\hspace{4cm}+ \rme^{\left[-2\log(2\kappa'_0) + 2d + 2 \log c_\kappa'\right] m_{\kappa',n}} (2\kappa' m_{\kappa',n})^{2\kappa' m_{\kappa',n}} \left(\sqrt{\frac{x}{n}} \vee \frac{x}{n} \right)^{1/2} \Bigg\}\eqsp,
\end{align*}
since $c_{\kappa' }\leq 1$, $\kappa' \leq 1$ and $m_{\kappa',n} \geq 1$. Then, choosing $c_{\kappa' }\leq \exp(-(3d+5)/2 + \log(2\kappa_0))\wedge 1/\nu_{\text{est}}$ and
\begin{equation*}
m_{\kappa',n} \leq \frac1{2\kappa'} \frac{\alpha \log n}{\log( \alpha \log n)}\eqsp,
\end{equation*}
for some $\alpha > 0$ yields, for all $\kappa_0 \in (1/2,1]$, $S > 0$, $\beta > 0$, $c_\nu > 0$, $c_Q > 0$ and $c_\psi > 0$, there exist $c > 0$ and $c' > 0$ such that for all $\nu \in [(d+4/3)\rme/S, \nu_\text{est}]$, $n \geq 1$, $x \in (0,cn]$ and $\kappa \in [\kappa_0,1]$,
\begin{multline*}
\underset{Q^\star \in \Qbf(\nu,c_\nu,c_Q)}{\inf_{R^\star \,:\, \Phi_{R^\star} \in \Psi(\kappa,S,\beta,c_\psi)}}
\po_{R^\star,Q^\star} \Bigg(
	\forall \kappa' \in [\kappa_0,\kappa], \ 
	\| \hat{f}_{\kappa',n} - f^\star \|_{\Lbf^2(\R^{d_1} \times \R^{d_2})}^2 \\
	\leq c' \max\left\{
			2^{-\beta} m_{\kappa',n}^{-2 \kappa' \beta},
			\rme^{- m_{\kappa',n}} \left[ 1 \vee \frac{x^{1/4} n^\alpha}{n^{1/4}} \vee \frac{x^{1/2} n^\alpha}{n^{1/2}} \right] \right\}
\Bigg) \geq 1 - 4e^{-x} \eqsp.
\end{multline*}
Now, assume $\alpha \leq 1/4$ and $(c_m \log n) / \log \log n \leq m_{\kappa,n} \leq (C_m \log n) / \log \log n$ for all $\kappa$ and $n$ for some constants $c_m > 0$ and $C_m > 0$ and take $x = \log n$. It follows that there exists $n_0$ such that for all $n \geq n_0$,
\begin{equation}
\label{eq_controle_final_sec_C}
\sup_{\kappa \in [\kappa_0,1]}
\underset{Q^\star \in \Qbf(\nu,c_\nu,c_Q)}{\inf_{R^\star \,:\, \Phi_{R^\star} \in \Psi(\kappa,S,\beta,c_\psi)}}
\po_{R^\star,Q^\star} \left(
	\sup_{\kappa' \in [\kappa_0,\kappa]} \left\{
	m_{\kappa',n}^{2 \kappa' \beta} \| \hat{f}_{\kappa',n} - f^\star \|_{\Lbf^2(\R^{d_1} \times \R^{d_2})}^2 \right\}
		\leq c' 2^{-\beta} \right)
	\geq 1 - \frac{4}{n} \eqsp.
\end{equation}
Finally, note that $m_{\kappa',n}^{2 \kappa' \beta} \| \hat{f}_{\kappa',n} - f^\star \|_{\Lbf^2(\R^{d_1} \times \R^{d_2})}^2 \leq (C_M \frac{\log n}{\log \log n})^{2\beta} \text{diam}(\Upsilon_{\kappa_0,S})^2$ by construction, so that Theorem~\ref{th:fhat} follows.

\section{Proof of Proposition~\ref{prop:compromis:M}}
\label{sec:proof:compromis:M}

By \eqref{eq:lowerbound:M}, Proposition~\ref{prop:compromis:M} may be proved by balancing a lower bound for $M^\text{lin}(h,\phi;\nu)$ and an upper bound for $\| h(\cdot, 0) h(0, \cdot) \|_{\Lbf^2(\neighborhood{d_1}{\nu}\times \neighborhood{d_2}{\nu})} ^2$.

The lower bound on $M^\text{lin}(h,\phi;\nu)$ is first obtained on polynomials with known degree $m$:
\begin{lem}
\label{lem_minoration_Mlin}
Let $d = d_1+d_2$ and $g$ be the function defined in \eqref{eq:def:g}. There exist a constant $\sfc>0$ such that for all $\kappa, S, \nu > 0$, $m \in \N^*$, $\phi \in \Upsilon_{\kappa,S}$ and $h \in \Gcal_{\kappa,S}$,
\begin{equation*}
M^\text{lin}(T_m h,T_m \phi;\nu)\geq \sfc (4\sqrt{2})^{-2d}
		(4\rme)^{-6m}
		m^{-5d-3}
		(\nu \vee \nu^{-3})^{-2m}
		g(\kappa,S)^{-6m}
		d^{-6m}
		\|T_m h\|_{\Lbf^2(\neighborhood{d_1}{\nu}\times \neighborhood{d_2}{\nu})}^2\eqsp,
\end{equation*}
where $M^\text{lin}$, $\Upsilon_{\kappa,S}$, $\Gcal_{\kappa,S}$ and $T_m \phi$ are defined in \eqref{eq:def:mlin}, \eqref{eq:def:funcset}, \eqref{eq:def:funcsetdiff} and \eqref{eq:def:tronc}. The function $g$ is defined in \eqref{eq:def:g} in the the supplementary material..
\end{lem}

\begin{proof}
The proof is postponed to Section~\ref{sec:minoration_Mlin} in the the supplementary material.
\end{proof}

Then, we extend this lower bound on all functions $h$ and $\phi$ by controlling the difference between $h$ and $\phi$ and their truncations to degree $m$:
\begin{lem}
\label{lem_partie_quadratique}
Let $d = d_1+d_2$. There exist $\sfc>0$ and $\tilde\sfc>0$ such that for all $\kappa, \nu, S > 0$, $\phi \in \Upsilon_{\kappa,S}$, $h \in \Gcal_{\kappa,S}$ and $m \geq d/\kappa$, 
\begin{multline*}
M^\text{lin}(h,\phi;\nu)\geq \com{(\sfc/2)\alpha(m,\nu,\kappa,S)} \|h\|_{\Lbf^2(\neighborhood{d_1}{\nu}\times \neighborhood{d_2}{\nu})}^2\\
 - (\sfc\com{\alpha(m,\nu,\kappa,S)} + \tilde\sfc C_\Upsilon^4(\kappa,S,\nu)) 2^{2d}(2\nu)^d (S\nu)^{2m}m^{-2\kappa m+2d} f_\kappa(S\nu)^2\eqsp,
\end{multline*}
where
\begin{equation}
\label{eq:def:alpha}
\alpha(m,\nu,\kappa,S) = (2\sqrt{2})^{-2d}(4\rme)^{-6m}m^{-5d-3}(\nu \vee \nu^{-3})^{-2m}g(\kappa,S)^{-6m}d^{-6m}
\end{equation}
and where $M^\text{lin}(h,\phi;\nu)$, $\Upsilon_{\kappa,S}$, $\Gcal_{\kappa,S}$, $f_\kappa$ and $g$,  are defined in \eqref{eq:def:mlin}, \eqref{eq:def:funcset}  \eqref{eq:def:funcsetdiff}, and \eqref{eq:def:fkappa}, \eqref{eq:def:g}   in the supplementary material.
\end{lem}

\begin{proof}
The proof is postponed to Section~\ref{sec_proofs_minoration_M} in the supplementary material.
\end{proof}

Finally, a careful choice of $m$ allows to show that $M^\text{lin}(h,\phi;\nu)$ is lower bounded by $\|h\|^{2+o(1)}$ when $\|h\|$ is small enough:

\begin{prop}
\label{prop:compromis:Mlin}
Assume that $S \geq 1$ is such that $S\nu \geq x_0 \vee u_0$ where $x_0$ and $u_0$ are defined in Lemma~\ref{lem_controle_constantes}  in the supplementary material.
Then, for all $h$ such that 
\begin{equation*}
\|h\|_{\Lbf^2(\neighborhood{d_1}{\nu}\times \neighborhood{d_2}{\nu})} < \exp\left(- \left(1 \vee \frac{\delta \rme^{\gamma}}{\vartheta}\right)^{2/\kappa}\right)\eqsp,
\end{equation*}
where $\gamma = (7d+3)/2$, $\vartheta = ((8\rme d)^{3} (\nu \vee \nu^{-3}))^{-1} (S\nu)^{-1} (\rme^{d+2} S)^{-3/\kappa}\exp(-6\kappa (\rme^{d+2} S)^{1/\kappa})$ and $\delta^2 = 288 \cdot 7^4(\tilde{\sfc}/\sfc) (64 \nu)^d \kappa^{-4d-6} \exp((2d+9)\kappa (S\nu)^{1/\kappa})$,
\begin{equation}
\label{eq:min:mlin:final}
M^\text{lin}(h,\phi;\nu) 
 	\geq \frac{\sfc}{8} 8^{-d} \|h\|^2
 		 \left( \frac{\kappa \log \log (1/\|h\|)}{4 \log (1/\|h\|)} \right)^{5d+3}
 		 \|h\|^{ \displaystyle \frac{-8 \log (b S\nu)}{\kappa \log \log \left(1/\|h\|\right)}}\eqsp,
\end{equation}
where $b=((8\rme d)^{3} (\nu \vee \nu^{-3}))^{-1} (S\nu)^{-1} (\rme^{d+2} S)^{-3/\kappa}\exp(-6\kappa (\rme^{d+2} S)^{1/\kappa})$ and where $\tilde{\sfc}$ and $\sfc$ are numerical constants defined in Lemma~\ref{lem_partie_quadratique}.
\end{prop}
Note that $\nu \geq (d+4/3)\rme/S$ entails $S\nu \geq x_0 \vee u_0$ where $x_0$ and $u_0$ are defined in Lemma~\ref{lem_controle_constantes}. 

\begin{proof}
The proof is postponed to Section~\ref{sec_proofs_minoration_M} in the supplementary material.
\end{proof}

The upper bound on $\| h(\cdot, 0) h(0, \cdot) \|_{\Lbf^2(\neighborhood{d_1}{\nu}\times \neighborhood{d_2}{\nu})} ^2$ is likewise first obtained on polynomials with known degrees $m$ then extended to any function $h$ by controlling the difference between $h$ and its truncation:
\begin{lem}
\label{lem_partie_quartique}
Let $d = d_1+d_2$. There exists a numerical constant $c_5>0$ such that for all $\kappa > 0$, $S < \infty$, $\nu > 0$, $m \geq d/\kappa$ and $h \in \Gcal_{\kappa,S}$,
\begin{multline*}
\| h(\cdot, 0) h(0, \cdot) \|_{\Lbf^2(\neighborhood{d_1}{\nu} \times \neighborhood{d_2}{\nu})}^2
	\leq 16\left\{(2c_5 m/\nu)^{2d_1} + (2c_5 m/\nu)^{2d_2}\right\} \| h \|_{\Lbf^2(\neighborhood{d_1}{\nu}\times \neighborhood{d_2}{\nu})}^4 \\
	+ \left\{(4 (2c_5 m/\nu)^{d_1} + 2)^2 + (4 (2c_5 m/\nu)^{d_2} + 2)^2\right\}2^{4d}(2\nu)^{2d} (S\nu)^{4m} m^{-4\kappa m+4d} f_\kappa(S\nu)^4 \eqsp,
\end{multline*}
where $\Gcal_{\kappa,S}$ and $f_\kappa$ are defined in \eqref{eq:def:funcsetdiff} and \eqref{eq:def:fkappa}.
\end{lem}

\begin{proof}
The proof is postponed to Section~\ref{sec_proofs_minoration_M} in the supplementary material.
\end{proof}

Finally, a careful choice of $m$ shows that this term is upper bounded by $\|h\|^{4-o(1)}$ when $\|h\|$ is small enough:

\begin{prop}
\label{prop:compromis:h2}
Assume that $S \geq 1$ is such that $S\nu \geq x_0 \vee u_0$ where $x_0$ and $u_0$ are defined in Lemma~\ref{lem_controle_constantes} in the the supplementary material.
Then, for all $h$ such that 
\begin{equation*}
\|h\|_{\Lbf^2(\neighborhood{d_1}{\nu}\times \neighborhood{d_2}{\nu})} < \exp\left(- \left(1 \vee \frac{\tilde \delta\rme^{\tilde \gamma}}{\tilde \vartheta}\right)^{2/\kappa}\right)
\end{equation*}
where $\tilde \gamma = d$, $\tilde \vartheta = (S\nu)^{-1}$ and $\tilde \delta = 3\sqrt{6} \cdot 2^d (2\nu)^{d/2} (S \nu)^{1/\kappa} \exp(\kappa (S \nu)^{1/\kappa})$, and with $c_5$ the numerical constant of Lemma~\ref{lem_partie_quartique},
\begin{equation}
\label{eq:min:h2:final}
\| h(\cdot, 0) h(0, \cdot) \|^2
	\leq 64 \left(\frac{2c_5}{\nu} \vee \frac{2\kappa}{3}\right)^{2(d_1\vee d_2)} \left(\frac{4}{\kappa} \frac{\log (1/\|h\|)}{\log \log (1/\|h\|)}\right)^{2(d_1\vee d_2)} \| h \|^4 \eqsp.
\end{equation}
\end{prop}

\begin{proof}
The proof is postponed to Section~\ref{sec_proofs_minoration_M} in the supplementary material.
\end{proof}

By Proposition~\ref{prop:compromis:Mlin}, Proposition~\ref{prop:compromis:h2} and \eqref{eq:lowerbound:M}, 
\begin{multline*}
 M(\phi^\star; \nu | R^\star, Q^\star)\geq c^4_\nu\frac{\sfc}{16} \|h\|^2 8^{-d}
 		 \left( \frac{\kappa \log \log (1/\|h\|)}{4 \log (1/\|h\|)} \right)^{5d+3}
 		 \|h\|^{ \displaystyle \frac{-8 \log (b S\nu)}{\kappa \log \log (1/\|h\|)}} \\
- 64 c^4_\nu \left(\frac{2c_5}{\nu} \vee \frac{2\kappa}{3}\right) ^{2(d_1\vee d_2)} \left(\frac{4}{\kappa} \frac{\log (1/\|h\|)}{\log \log (1/\|h\|)}\right)^{2(d_1\vee d_2)} \| h \|^4\eqsp.
\end{multline*}
Therefore, assuming 
\begin{multline}
\label{eq:final:comp}
\frac{\sfc}{16} \|h\|^2 8^{-d}
 		 \left( \frac{\kappa \log \log (1/\|h\|)}{4 \log (1/\|h\|)} \right)^{5d+3}
 		 \|h\|^{ \displaystyle \frac{-8 \log (b S\nu)}{\kappa \log \log (1/\|h\|)}} \\
 	\geq 128 \left(\frac{2c_5}{\nu} \vee \frac{2\kappa}{3}\right)^{2(d_1\vee d_2)} \left(\frac{4}{\kappa} \frac{\log (1/\|h\|)}{\log \log (1/\|h\|)}\right)^{2(d_1\vee d_2)} \| h \|^4
\end{multline}
yields
\begin{equation}
\label{eq_minoration_quasi_finale_M}
M(\phi^\star; \nu | R^\star, Q^\star)
	\geq 64 \left(\frac{2c_5}{\nu} \vee \frac{2\kappa}{3}\right)^{2(d_1\vee d_2)} \left(\frac{4}{\kappa} \frac{\log (1/\|h\|)}{\log \log (1/\|h\|)}\right)^{2(d_1\vee d_2)} \| h \|^4\eqsp.
\end{equation}
Note that \eqref{eq:final:comp} is implied by
\begin{equation*}
\left( \frac{\kappa \log \log (1/\|h\|)}{4 \log (1/\|h\|)} \right)^{3(d+1)} \left(\frac1{\|h\|}\right)^{\displaystyle 2 - \frac{-8 \log (b S\nu)}{\kappa \log \log (1/\|h\|)}}
	\geq 2048 \cdot 8^d \sfc^{-1} \left(\frac{2c_5}{\nu} \vee \frac{2\kappa}{3}\right)^{2(d_1\vee d_2)}\eqsp.
\end{equation*}
Assume $\|h\| \leq \exp( -(b S\nu)^{-8/\kappa} ) \wedge \rme^{-\rme}$, then this equation is implied by
\begin{equation*}
\frac{1 / \|h\|}{\log (1/\|h\|)^{3(d+1)}}
	\geq 2048 \cdot 8^d \sfc^{-1} \left(\frac{2c_5}{\nu} \vee \frac{2\kappa}{3}\right)^{2(d_1\vee d_2)} \left(\frac{4}{\kappa}\right)^{3(d+1)}\eqsp.
\end{equation*}
Since $\log x \leq x$ for all $x \geq 1$, $\log x^{1/(2\alpha)} \leq x^{1/(2\alpha)}$ for all $x \geq 1$ and $\alpha > 0$, so that $(\log x)^\alpha \leq (2\alpha)^\alpha \sqrt{x}$ for all $x \geq 1$ and $\alpha > 0$. Thus, this equation is implied by
\begin{equation*}
\frac{\sqrt{1 / \|h\|}}{(6(d+1))^{3(d+1)}}
	\geq 2048 \cdot 8^d \sfc^{-1} \left(\frac{2c_5}{\nu} \vee \frac{2\kappa}{3}\right)^{2(d_1\vee d_2)} \left(\frac{4}{\kappa}\right)^{3(d+1)}\eqsp,
\end{equation*}
that is
\begin{equation*}
\|h\| \leq
	2^{-22} \cdot 64^{-d} \sfc^2 \left(\frac{\nu}{2c_5} \wedge \frac{3}{2\kappa}\right)^{4(d_1\vee d_2)} \left(\frac{\kappa}{24(d+1)}\right)^{6(d+1)}\eqsp.
\end{equation*}
In this case, $\log(1/\|h\|)/\log \log(1/\|h\|)$ in equation~\eqref{eq_minoration_quasi_finale_M} is lower bounded by a constant, which concludes the proof of Proposition~\ref{prop:compromis:M}.

\section{Proof of Theorem~\ref{theoident1}}
\label{app:proof:ident}

The proof follows the same lines as that of Theorem 1 in \cite{gassiat:lecorff:lehericy:2019}. The following statement, which may be established by arguing variable by variable, is used repeatedly. If a multivariate function is analytic on the whole multivariate complex space and is the null function in an open set of the multivariate real space or in an open set of the multivariate purely imaginary space, then it is the null function on the whole multivariate complex space. 
Assume $\po_{R,\noisedist}=\po_{\tilde{R},\noisedistbis}$ and let $\phi_{i}$ (resp. $\widetilde \phi_{i}$) be the characteristic function of $\noisedist^{(i)}$ (resp. $\noisedistbis^{(i)}$) for $i\in\{1,2\}$. Since the distribution of $Y^{(1)}$ and $Y^{(2)}$ are the same under $\po_{\transk,\noisedist}$ and $\po_{\transkbis,\noisedistbis}$, for any $t\in \R^{d_1}$,
\begin{equation}
\label{chara1bis}
\phi_1\left(t\right)\Phi_{R}\left(t,0\right)
	= \widetilde{\phi}_1\left(t\right) \Phi_{\tilde{R}}\left(t,0\right)
\end{equation}
and for any $t\in \R^{d_2}$,
\begin{equation}
\label{chara2bis}
\phi_2\left(t\right)\Phi_{R}\left(0,t\right)
	= \widetilde{\phi}_2\left(t\right) \Phi_{\tilde{R}}\left(0,t\right)\eqsp.
\end{equation}
Since the distribution of $\bfY$ is the same under $\po_{\transk,\noisedist}$ and $\po_{\transkbis,\noisedistbis}$, for any $(t_1,t_2)\in \R^{d_1}\times \R^{d_2}$,
\begin{equation}
\label{chara3bis}
\phi_1\left(t_1\right) \phi_2\left(t_2\right) \Phi_{R}\left(t_1,t_2\right)
	=\widetilde{\phi}_1\left(t_1\right) \widetilde{\phi}_2\left(t_2\right) \Phi_{\tilde{R}}\left(t_1,t_2\right)\eqsp.
\end{equation}
There exists a neighborhood $V$ of $0$ in $\R^{d_1}\times \R^{d_2}$ such that for all $t=(t_1,t_2)\in V$, $\phi_1\left(t_1\right)\neq 0$, $\phi_2\left(t_2\right)\neq 0$, $\widetilde{\phi}_1\left(t_1\right)\neq 0$, $\widetilde{\phi}_2\left(t_2\right)\neq 0$, so that~\eqref{chara1bis}, \eqref{chara2bis} and~\eqref{chara3bis} imply that for any $ (t_1,t_2)\in V^2$,
\begin{equation}
\label{chara5bis}
\Phi_{R}\left(t_1,t_2\right)\Phi_{\tilde{R}}\left(t_1,0\right)\Phi_{\tilde{R}}\left(0,t_2\right)
	=\Phi_{\tilde{R}}\left(t_1,t_2\right)\Phi_{R}\left(t_1,0\right)\Phi_{R}\left(0,t_2\right)\eqsp.
\end{equation}
Since $(z_1,z_2) \mapsto \Phi_{R}\left(z_1,z_2\right)\Phi_{\tilde{R}}\left(z_1,0\right)\Phi_{\tilde{R}}\left(0,z_2\right)-\Phi_{\tilde{R}}\left(z_1,z_2\right)\Phi_{R}\left(z_1,0\right)\Phi_{R}\left(0,z_2\right)$ is a multivariate analytic function of $d_1+d_2$ variables which is zero in a purely real neighborhood of $0$, then it is the null function on the whole multivariate complex space so that for any $z_1\in\C^{d_1}$ and $z_2\in\C^{d_2}$,
\begin{equation}
\label{analyticbis}
\Phi_{R}\left(z_1,z_2\right)\Phi_{\tilde{R}}\left(z_1,0\right)\Phi_{\tilde{R}}\left(0,z_2\right)
	=\Phi_{\tilde{R}}\left(z_1,z_2\right)\Phi_{R}\left(z_1,0\right)\Phi_{R}\left(0,z_2\right)
\eqsp.
\end{equation}
Fix $(u_2,\ldots,u_{d_1})\in \C^{d_1-1}$ and let $\Zcal$ be the set of zeros of $u \mapsto \Phi_{R}(u,u_2,\ldots,u_{d_1},0)$ and $\widetilde{\Zcal}$ be the set of zeros of $u\mapsto\Phi_{\tilde{R}}(u,u_2,\ldots,u_{d_1},0)$. Let $u_1\in \Zcal$. Write $z_1=(u_1,u_2,\ldots,u_{d_1})$ so that by \eqref{analyticbis}, for any $z_2\in\C^{d_2}$,
\begin{equation}
\label{zerosbis}
\Phi_{R}\left(z_1,z_2\right)\Phi_{\tilde{R}}\left(z_1,0\right)\Phi_{\tilde{R}}\left(0,z_2\right)=0\eqsp.
\end{equation}
Using H\ref{assum:}, $z_2 \mapsto \Phi_{R}\left(z_1,z_2\right)$ is not the null function. 
Thus, there exists $z_2^{\star}$ in $\C^{d_2}$ such that $\Phi_{R}\left(z_1,z_2^{\star}\right)\neq 0$ and by continuity, there exists an open neighborhood of $z_2^{\star}$ such that for all $z_2$ in this open set, $\Phi_{R}\left(z_1,z_2\right)\neq 0$. Since $z\mapsto \Phi_{\tilde{R}}\left(0,z\right)$ is not the null function and is analytic on $\C^{d_2}$, it can not be null all over this open set, so that there exists $z_2$ such that simultaneously $\Phi_{R}\left(z_1,z_2\right)\neq 0$ and $\Phi_{\tilde{R}}\left(0,z_2\right)\neq 0$. 
Then~\eqref{zerosbis} leads to $\Phi_{\tilde{R}}\left(z_1,0\right)=0$, so that $\Zcal \subset \widetilde{\Zcal}$.
A symmetric argument yields $\widetilde{\Zcal}\subset \Zcal$ so that $\Zcal = \widetilde{\Zcal}$. Moreover, the analytic functions $u\mapsto \Phi_{R}(u,u_2,\ldots,u_d,0)$ and $u\mapsto \Phi_{\tilde{R}}(u,u_2,\ldots,u_d,0)$ have exponential growth order less than $2$, so that using Hadamard's factorization Theorem, see \cite[Chapter~5, Theorem~5.1]{Stein:complex}, there exists a polynomial function $s$ with degree at most $1$ (and coefficients depending on $(u_2,\ldots,u_d)$) such that for all $u\in\C$,
\begin{equation*}
\Phi_{R}(u,u_2,\ldots,u_d,0) = \rme^{s(u)}\Phi_{\tilde{R}}(u,u_2,\ldots,u_d,0)\eqsp.
\end{equation*}
Arguing similarly for all variables, there exists a function $S$ on $\C^{d_1}$, which is, for any $i=1,\ldots,d_1$, polynomial with degree at most $1$ in $u_i$, and such that for all $(u_1,\ldots,u_{d_1})\in\C^{d_1}$,
\begin{equation}
\label{eq:phi:exp}
\Phi_{R}(u_1,u_2,\ldots,u_{d_1},0) = \rme^{S(u_1,u_2,\ldots,u_{d_1})}\Phi_{\tilde{R}}(u_1,u_2,\ldots,u_{d_1},0)\eqsp.
\end{equation}
In other words, there exists complex functions $a_{i}$, $b_{i}$ on $\C^{{d_1}-1}$ such that, if we denote $u^{(-i)}$ the $({d_1}-1)$-dimensional complex vectors with the same coordinates as $u$ except that $u_i$ is not included in the coordinates, then
\begin{equation*}
S(u_{1},u_{2},\ldots,u_{{d_1}})=a_{i}(u^{(-i)})u_{i}+b_{i}(u^{(-i)}),\;i=1,\ldots,{d_1} \eqsp.
\end{equation*}
But, for $i\neq j$, the fact that $a_{i}(u^{(-i)})u_{i}+b_{i}(u^{(-i)})=a_{j}(u^{(-j)})u_{j}+b_{i}(u^{(-j)})$ implies that $a_{i}(u^{(-i)})$ and $b_{i}(u^{(-i)})$ are polynomial functions with degree at most $1$ in $u_j$ (this may be seen for instance by taking complex derivatives), and by induction we get that $S$ is a polynomial function which is, for any $i=1,\ldots,{d_1}$, polynomial with degree at most $1$ in $u_i$.

Since $\Phi_{R}(0,\ldots,0)=\Phi_{\tilde{R}}(0,\ldots,0)=1$, the constant term of the polynomial $S$ is $0$. We are now going to prove that the polynomial $S$ has total degree at most $1$. Note that the fact that $S$ has degree at most $1$ in each variable is not enough to deduce that $S$ is linear: for instance, $u_1 u_2$ has degree at most $1$ in each variable but has total degree $2$.

Assume that $\tilde{R}^{(1)}$ is not supported by $0$. Then there exist $a=(a_{1},\ldots,a_{d_{1}})\in\R^{d_{1}}$, $\alpha>0$ and $\delta >0$ such that 
\begin{equation*}
0 \notin \prod_{j=1}^{d_{1}} [a_{j}-\alpha,a_{j}+\alpha] \quad\mathrm{and}\quad
\tilde{R}^{(1)}\left(\prod_{j=1}^{d_{1}} [a_{j}-\alpha,a_{j}+\alpha]
\right)\geq \delta\eqsp,
\end{equation*}
which gives, for all $u\in\R^{d_{1}}$,
\begin{equation*}
\Phi_{\tilde{R}}(-iu,0) \geq \delta \rme^{ \sum_{j=1}^{d_{1}}\inf_{x_{j} \in [a_{j}-\alpha,a_{j}+\alpha]}u_{j}x_{j}}\eqsp,
\end{equation*}
so that using~\eqref{eq:phi:exp}, for all $u\in\R^{d_{1}}$,
\begin{equation*}
\Phi_{R}(-iu,0) \geq \delta \rme^{S(u)} \rme^{ \sum_{j=1}^{d_1}\inf_{x_{j} \in [a_{j}-\alpha,a_{j}+\alpha]}u_{j}x_{j}}\eqsp.
\end{equation*}
If $S$ has total degree at least $2$, then there exist $i\neq j$ and polynomial functions with degree at most one in each variable $c_{1}$ on $\C^{d_1-2}$ and $c_{2}$, $c_{3}$ on $\C^{d_1-1}$ such that, if we denote $u^{(-i,-j)}$ the $(d_1-2)$-dimensional complex vectors with the same coordinates as $u$ except that $u_i$ and $u_{j}$ are not included in the coordinates, then
$S(u)=c_{1}(u^{(-i;-j)})u_{i}u_{j} + c_{2}(u^{(-i)}) + c_{3}(u^{(-j)}) $. Without loss of generality say that $i=1$ and $j=2$. Then it is possible to find $u\in \R^{d_{1}}$ and $\tilde{\delta}>0$ such that for all $t\geq 0$,
$S(-itu_{1},-itu_{2},-iu_{3},\ldots,-iu_{d_{1}})\geq \tilde{\delta} t(u_{1}^2+u_{2}^{2})$ leading to
\begin{equation*}
\forall t\geq 0,\; \Phi_{R}(-itu_{1},-itu_{2},-iu_{3},\ldots,-iu_{d_{1}},0) \geq \delta\rme^{\tilde{\delta} t(u_{1}^2+u_{2}^{2})} \rme^{ \sum_{j=1}^{d_{1}}\inf_{x_{j} \in [a_{j}-\alpha,a_{j}+\alpha]}u_{j}x_{j}}\eqsp,
\end{equation*}
contradicting the assumption that $R^{(1)} \in \M_{\rho}$ for some $\rho <2$. Thus, $S$ has total degree at most $1$ and there exists $m\in \C^{d_{1}}$ such that for all $z\in\C^{d_{1}}$, 
\begin{equation}
\label{idem}
\Phi_{R}(z,0)=\rme^{im_1^\top z}\Phi_{\tilde{R}}(z,0)\eqsp.
\end{equation}
On the other hand, if $\tilde{R}^{(1)}$ is supported by $0$ then~\eqref{eq:phi:exp} leads to $\Phi_{R}(-iu,0)=\rme^{S(-iu)}$ for all $u\in\R^{d_{1}}$ and the same argument leads to~\eqref{idem}.

As for all $z\in\R^{d_{1}}$, $\Phi_{R}(-z,0)=\overline{\Phi_{R}(z,0)}$ and $\Phi_{\tilde{R}}(-z,0)=\overline{\Phi_{\tilde{R}}(z,0)}$, $m_1\in\R^{d}$. 
Arguing similarly for the function $\Phi_{R}(0, z_2)$, there exists $m_2 \in\R^{d_2}$ such that for all $z\in\C^{d_2}$, 
\begin{equation}
\label{idem2}
\Phi_{R}(0,z)=\rme^{i m_2^\top z}\Phi_{\tilde{R}}(0,z)\eqsp.
\end{equation}
Combining~\eqref{idem} and~\eqref{idem2} with~\eqref{analyticbis} yields, for all $ (t_{1},t_{2})\in \R^{d_1}\times\R^{d_2}$,
\begin{equation}
\label{ident1m}
\Phi_{R}(t_{1},t_{2})=e^{im_1^{\top}t_{1}+im_2^\top t_{2}}\Phi_{\tilde{R}}(t_{1},t_{2})\eqsp.
\end{equation} 
Then, using~\eqref{chara1bis}, for all $t\in \R^{d_1}$ such that $\Phi_{R}(t,0)\neq 0$, $\phi_1(t)=\rme^{-im_1^\top t}\widetilde{\phi}_1(t)$. Since the set of zeros of $t\mapsto \Phi_{R}(t,0)$ has empty interior, for each $t$ such that $\Phi_{R}(t,0)= 0$ it is possible to find a sequence $(t_{n})_{n\geq 1}$ such that $t_{n}$ tends to $t$ and for all $n$, $\Phi_{R_K}(t_{n},0)\neq 0$. But $\phi_1$ and $\widetilde{\phi}_1$ are continuous functions, so that for all $t\in\R^{d_1}$, 
\begin{equation}
\label{ident2m}
\phi_1(t)=\rme^{-im_{1}^\top t}\tilde{\phi}_1(t)\eqsp.
\end{equation} 
Similarly using~\eqref{chara2bis}, we get that for all $t\in\R^{d_2}$, 
\begin{equation}
\label{ident3m}
\phi_2(t)=\rme^{-im_{2}^\top t}\tilde{\phi}_2(t)\eqsp.
\end{equation} 
The proof is concluded by noting that~\eqref{ident1m}, \eqref{ident2m} and~\eqref{ident3m} imply that $R=\tilde{R}$ and $Q=\tilde{Q}$ up to translation.

\section{Proof of Corollary \ref{ident:ICAnoisy}}
\label{sec:proofCor2}
The noisy ICA model may be written as 
\begin{equation*}
\begin{pmatrix}
Y_I\\Y_J
\end{pmatrix}
=
\begin{pmatrix}
A_IS\\A_JS
\end{pmatrix}
+
\begin{pmatrix}
\varepsilon_I\\\varepsilon_J
\end{pmatrix}\eqsp.
\end{equation*}
Then, the ICA model fits the setting of Theorem~\ref{theoident1} with $Y^{(1)} = Y_I$, $Y^{(2)} = Y_J$, $X^{(1)} = A_IS$ and $X^{(2)} = A_JS$. Write $d_1 = |I|$ and $d_2 = d-|I|$ and denote by $R$ the joint distribution of $(A_IS,A_JS)$. Note first that if for all $1\leq j\leq q$ the distribution of all $S_j$ is in $\M^1_\rho$ then the distribution of $A_IS$ is in $\M^{d_1}_\rho$ as for all $\lambda\in \R^{d_1}$,
\begin{equation*}
\mathbb{E}\left[\exp\left(\lambda^\top A_IS\right)\right] = \prod_{j=1}^{q} \Psi_j\left((\lambda^\top A_I)_j\right)\eqsp,
\end{equation*}
where for all $1\leq j\leq q$ and all $z\in\mathbb{C}$, $\Psi_j(z) = \mathbb{E}\left[\exp\left(zS_j\right)\right]$.
Then, by assumption, and the Cauchy-Schwarz inequality, there exist $A_j\in\R$ and $B_j\in\R$ such that 
\begin{equation*}
\Psi_j\left((\lambda^\top A_I)_j\right) \leq A_j\rme^{B_j|\langle \lambda ; A_I(j)\rangle|^\rho}\leq A_j\rme^{B_j\|A_I(j)\|^\rho \|\lambda\|^\rho}\eqsp,
\end{equation*}
where $A_I(j)$ is the $j$-th column of $A_I$. Therefore, $A_IS$ is in $\M^{d_1}_\rho$ with constants given by $A = \prod_{j=1}^q A_I(j)$ and $B = \sum_{j=1}^q B_j\|A_I(j)\||^\rho$. Similarly, $A_JS$ is in $\M^{d_2}_\rho$. Then, for any $(z_0,z)\in \C^{d_1}\times \C^{d_2}$,
\begin{align*}
 \Phi_{R}(z_{0},z) = \mathbb{E}\left[\exp\left((iz_0^\top A_I +iz^\top A_J)S\right)\right]
	&= \prod_{j=1}^{q}\mathbb{E}\left[\exp\left((iz_0^\top A_I +iz^\top A_J)_jS_j\right)\right]\eqsp,\\
	&= \prod_{j=1}^{q} \Psi_j\left((iz_0^\top A_I +iz^\top A_J)_j\right)\eqsp.
\end{align*}
For all $1\leq j\leq q$, the function $z \mapsto \Psi_j((z_0^\top A_I +z^\top A_J)_j)$ is analytic therefore $z\mapsto \Phi_{R}(z_{0},z)$ is the null function if and only if there exists $1\leq j\leq q$ such that the function $z \mapsto \Psi_j((iz_0^\top A_I +iz^\top A_J)_j)$ is null (as $\Phi_{R}(z_{0},\cdot)$ is a finite product of analytic functions).
As all columns of $A_J$ are nonzero, for all $1\leq j\leq q$, $z \mapsto \Psi_j((iz_0^\top A_I +iz^\top A_J)_j)$ is not the null function. Similarly, for any $(z,z_0)\in \C^{d_1}\times \C^{d_2}$,
\begin{equation*}
 \Phi_{R}(z,z_0) = \prod_{j=1}^{q} \Psi_j\left((iz^\top A_I +iz_0^\top A_J)_j\right)
\end{equation*}
and the proof that $ \Phi_{R}(z,z_0) $ is not the null function follows the same steps.

\section{Links between \texorpdfstring{$\Mcal^d_\rho$}{M(rho)} and \texorpdfstring{$\Upsilon_{1/\rho,S}$}{Upsilon(1/rho,S)}: proof of Lemma~\ref{lem_lien_Mrho_Upsilon}}
\label{sec_preuve_lem_lien_Mrho_Upsilon}

{\em First implication.} First, note that for all $n \geq 0$, $n! \geq (n/\rme)^n$, so that by the concavity of $x\mapsto \log x$, for all $j \in \N^d$,
\begin{equation}
\label{eq_maj_inverse_factorielles}
\left(\prod_{a=1}^d j_a!\right)^{-1} \leq \rme^{\|j\|_1} \exp\left( - \sum_{a=1}^d j_a \log j_a \right) \leq \rme^{\|j\|_1} \exp(- \|j\|_1 \log (\|j\|_1/d)) \leq \left( \frac{\rme d}{\|j\|_1} \right)^{\|j\|_1}\eqsp.
\end{equation}
Let $\rho \geq 1$ and $\mu \in \M^d_\rho$. Write $\varphi_\mu : \lambda \in \R^d \mapsto \int \exp\left(i\lambda^\top x\right)\mu(\rmd x)$. 
Then $\varphi_\mu(0) = 1$ and for all $j \in \N^d \setminus \{0\}$, if $X$ has distribution $\mu$, by the inequality of arithmetic and geometric means and by convexity of $x \longmapsto x^{\|j\|_1}$ on $\R_+$,
\begin{align*}
| \partial^j \varphi_\mu(0) |
	= \left| \E \left[ \prod_{a=1}^d X_a^{j_a} \right] \right|
	&\leq \E\left[ \left(\sum_{a=1}^d\frac{j_a}{\|j\|_1}|X_a|\right)^{\|j\|_1} \right] \\
	&\leq \E\left[\sum_{a=1}^d\frac{j_a}{\|j\|_1} |X_a|^{\|j\|_1} \right]
	\leq \max_{1 \leq a \leq d} \E\left[ |X_a|^{\|j\|_1} \right]\eqsp.
\end{align*}
Since $\mu \in \Mcal_\rho^d$ by assumption, there exists $A$ and $B$ such that for all $\lambda \in \R^d$,
\begin{equation}
\label{eq:mrho}
\E[e^{\lambda^\top X}] \leq A e^{B \|\lambda\|_2^\rho}\eqsp.
\end{equation}
Hence, by Markov's inequality, for all $a \in \{1, \dots, d\}$, $t > 0$ and $\lambda > 0$,
\begin{equation*}
\po(X_a \geq t)
	\leq \frac{\E[e^{\lambda X_a}]}{e^{\lambda t}}
	\leq A \exp(B \lambda^\rho - \lambda t)\eqsp.
\end{equation*}
Thus, if $\rho = 1$, then $|X_a| \leq B$ almost surely, and therefore for all $j \in \N^d \setminus \{0\}$,
\begin{equation*}
| \partial^j \varphi_\mu(0) | \leq B^{\|j\|_1}\eqsp,
\end{equation*}
which concludes the proof together with equation~\eqref{eq_maj_inverse_factorielles}. In the following, assume $\rho > 1$, so that
\begin{equation*}
\po(X_a \geq t) \leq A \exp(-C t^{\rho / (\rho - 1)})\eqsp,
\end{equation*}
where $C = (\rho-1) B (\rho B)^{-\rho / (\rho-1)} > 0$. Therefore, writing $\gamma = \rho / (\rho - 1) > 1$ yields
\begin{equation*}
\po(|X_a| \geq t) \leq 2A \exp(-C t^\gamma)\eqsp.
\end{equation*}
Let $j \in \N^d \setminus \{0\}$, then
\begin{align*}
\E \left[ |X_a|^{\|j\|_1}\right] = \int_{\varepsilon \geq 0} \po(|X_a|^{\|j\|_1} \geq \varepsilon) \rmd \varepsilon
	&= \|j\|_1 \int_{t \geq 0} \po(|X_a| \geq t)t^{\|j\|_1-1} \rmd t\eqsp,\\
	&\leq 2A\|j\|_1 \int_{t \geq 0} t^{\|j\|_1-1} e^{-Ct^\gamma}\rmd t\eqsp, \\
	&\leq 2A\|j\|_1 \left( 1 + \int_{t \geq 1} t^{\|j\|_1-1} e^{-Ct^\gamma} \right)\rmd t\eqsp. 
\end{align*}
For all $x\in \R$, note that $J_x = \int_{t \geq 1} t^x e^{-Ct^\gamma}\rmd t = (\gamma C)^{-1}(\rme^{-C} + (x - \gamma + 1)J_{x - \gamma})$. Since for $x \leq 0$, $J_x \leq \int_{t \geq 1} e^{-Ct}\rmd t \leq \rme^{-C}/C$ as $\gamma > 1$
\begin{equation*}
J_x \leq \frac{\rme^{-C}}{\gamma C} \left( 1 + \frac{x}{\gamma C} + \dots + \left(\frac{x}{\gamma C}\right)^{\lceil x/\gamma \rceil - 1} \right) + \left(\frac{x}{\gamma C}\right)^{\left\lceil x/\gamma \right\rceil} \frac{\rme^{-C}}{C}\eqsp.
\end{equation*}
Thus, if $x \geq \gamma C/2$, $J_x \leq 2 (e^{-C} / C) (4x / (\gamma C) )^{\lceil x/\gamma \rceil}$, and if $x \leq \gamma C/2$, $J_x \leq 2e^{-C}/C$, so that
\begin{equation*}
J_x \leq 2 \frac{\rme^{-C}}{\gamma C} \left( 1 + \left(\frac{4x}{\gamma C}\right)^{\left\lceil x/\gamma \right\rceil}\right)
\end{equation*}
and as a consequence
\begin{align*}
\E \left[ |X_a|^{\|j\|_1} \right]
	\leq 2A\|j\|_1( 1 + J_{\|j\|_1-1} )
	&\leq 2A\|j\|_1 \left( 1 + 2 \frac{\rme^{-C}}{C} + 2 \frac{\rme^{-C}}{C} \left(\frac{4(\|j\|_1-1)}{\gamma C}\right)^{\left\lceil (\|j\|_1-1)/\gamma \right\rceil} \right)\eqsp, \\
	&\leq 2A\|j\|_1 \left( 1 + 2 \frac{\rme^{-C}}{C} + 2 \frac{\rme^{-C}}{C}\left(\frac{4\|j\|_1}{\gamma C}\right)^{(\|j\|_1-1)/\gamma + 1} \right)\eqsp. 
\end{align*}
Hence, since $\|j\|_1 \geq 1$, there exists constants $c, c' > 0$ which only depends on $\rho$, $A$ and $B$ such that
\begin{align*}
| \partial^j \varphi_\mu(0) |
	&\leq c \|j\|_1^{2-1/\gamma} \left( 1 + \left(\frac{4\|j\|_1}{\gamma C}\right)^{\|j\|_1/\gamma} \right)\eqsp,\\
	&\leq 2 c \rme^{\|j\|_1 (2-1/\gamma)} \left( \left(\frac{4}{\gamma C} \vee 1 \right) \|j\|_1\right)^{\|j\|_1/\gamma}\eqsp,\\
	&\leq \left( c' \|j\|_1\right)^{\|j\|_1/\gamma}\eqsp.
\end{align*}
Bringing the above inequality together with equation~\eqref{eq_maj_inverse_factorielles} implies for all $j \in \N^d \setminus \{0\}$,
\begin{equation*}
\left| \frac{\partial^j \varphi_\mu(0)}{\prod_{a=1}^d j_a!} \right|
	\leq (2 \rme d (c')^{1/\gamma})^{\|j\|_1} \|j\|_1^{\|j\|_1 (1/\gamma - 1)}
	\leq S^{\|j\|_1} \|j\|_1^{-\|j\|_1 / \rho}\eqsp,
\end{equation*}
where $S = \rme d (c')^{1/\gamma}$, which concludes the proof.

{\em Second implication.} Let $S, \kappa > 0$ and let $\mu$ be a probability measure on $\R^d$ such that $\phi : \lambda \in \R^d \mapsto \int \exp(i \lambda^\top x) \mu(\rmd x) \in \Upsilon_{\kappa,S}$. Then $\phi$ can be extended on $\C^d$ and is equal to its Taylor expansion. In particular, for all $\lambda \in \R^d$,
\begin{equation*}
\phi(-i\lambda) \leq 1 + \sum_{j \in \N^d \setminus \{0\}} S^{\|j\|_1} \|j\|_1^{-\kappa \|j\|_1} \prod_{a=1}^d \lambda_a^{j_a} \leq 1 + \sum_{m \geq 1} m^d (S \|\lambda\|)^m m^{-\kappa m}\eqsp.
\end{equation*}
By Lemma~\ref{lem_controle_somme_ignoble}, there exists a constant $x_0 > 0$ depending only on $\kappa$ and $d$ such that for all $\lambda \in \R^d$,
\begin{align*}
\phi(-i\lambda)
	&\leq 1 + 6 (S \|\lambda\| \vee x_0)^{\frac{d+1}{\kappa}} \exp\left(\kappa (S \|\lambda\| \vee x_0)^{1/\kappa} \right)\eqsp,
\end{align*}
which implies that there exists a constant $c$ depending only on $\kappa$ and $d$ such that for all $\lambda \in \R^d$,
\begin{align*}
\int \exp(\lambda^\top x) \mu(\rmd x)
	&\leq c \left( 1 + (S \|\lambda\|)^{\frac{d+1}{\kappa}} \right) \exp\left(\kappa (S \|\lambda\|)^{1/\kappa} \right)\eqsp.
\end{align*}

\section{Proof of Theorem \ref{theo:adaptrate}}
\label{sec:theoadapt}

By equation~\eqref{eq_controle_final_sec_C} from the proof of Theorem~\ref{th:fhat}, taking $m_{\kappa,n}$ as in equation~\eqref{eq:mkappa},
for all $\kappa_0 \in (1/2,1]$, $S > 0$, $\beta > 0$, $c_\nu > 0$, $c_Q > 0$ and $c_\psi > 0$, there exist $c' > 0$ and $n_0$ such that for all $\nu \in [(d+4/3)\rme/S, \nu_\text{est}]$ and $n \geq n_0$,
\begin{equation}
\label{eq_tout_controle}
\sup_{\kappa \in [\kappa_0,1]}
\underset{Q^\star \in \Qbf(\nu,c_\nu,c_Q)}{\inf_{R^\star \,:\, \Phi_{R^\star} \in \Psi(\kappa,S,\beta,c_\psi)}}
\po_{R^\star,Q^\star} \! \left(
	\sup_{\kappa' \in [\kappa_0,\kappa]} \left\{
	\left(\frac{\log n}{\log \log n}\right)^{\kappa' \beta} \!\!\! \| \hat{f}_{\kappa',n} - f^\star \|_{\Lbf^2(\R^{d_1} \times \R^{d_2})} \right\}
		\leq c' \right) \!
	\geq 1 - \frac{4}{n} \eqsp.
\end{equation}
Write $\sigma_n(\kappa') = c' \left(\log n/\log \log n\right)^{-\kappa' \beta}$, we will show that
\begin{equation}
\label{eq_controle_proba_adaptatif}
\sup_{\kappa \in [\kappa_0,1]}
\underset{Q^\star \in \Qbf(\nu,c_\nu,c_Q)}{\inf_{R^\star \,:\, \Phi_{R^\star} \in \Psi(\kappa,S,\beta,c_\psi)}}
\po_{R^\star,Q^\star} \left(
	\| \hat{f}_{\hat{\kappa}_{n},n} - f^\star \|_{\Lbf^2(\R^{d_1} \times \R^{d_2})}
	 \leq 5 \sigma_{n}(\kappa) \right)
	\geq 1 - \frac{4}{n} \eqsp,
\end{equation}
and since $\| \hat{f}_{\hat{\kappa}_{n},n} - f^\star \|_{\Lbf^2(\R^{d_1} \times \R^{d_2})}^2 \leq \text{diam}(\Upsilon_{\kappa_0,S})^2$ by construction, Theorem~\ref{theo:adaptrate} follows.

Fix $\kappa$, $R^\star$ and $Q^*$ and assume we are in the event of probability at least $1-4/n$ of equation~\eqref{eq_tout_controle} where $\| \hat{f}_{\kappa',n} - f^\star \|_{\Lbf^2(\R^{d_1} \times \R^{d_2})}^2 \leq \sigma_n(\kappa')$ for all $\kappa' \in [\kappa_0,\kappa]$.
By the triangular inequality, for all $\kappa' \in [\kappa_0, \kappa]$,
\begin{align*}
\| \hat{f}_{\hat{\kappa}_n,n} - f^\star \|
	&\leq \| \hat{f}_{\kappa',n} - f^\star \| + \| \hat{f}_{\hat{\kappa}_n,n} - \hat{f}_{\kappa',n} \|\eqsp, \\
	&\leq \sigma_n(\kappa') +
		\begin{cases}
		A_n(\hat{\kappa}_n) + \sigma_n(\kappa') &\text{if } \hat{\kappa}_n \geq \kappa', \\
		A_n(\kappa') + \sigma_n(\hat{\kappa}_n) &\text{otherwise},
		\end{cases} \\
	&\leq \sigma_n(\kappa')
			+ A_n(\hat{\kappa}_n) + \sigma_n(\kappa')
			+ A_n(\kappa') + \sigma_n(\hat{\kappa}_n)\leq 2A_n(\kappa') + 3\sigma_n(\kappa')
\end{align*}
by definition of $\hat{\kappa}_n$ and since $A_n \geq 0$ and $\sigma_n \geq 0$. Recall that
\begin{equation*}
A_n(\kappa') = 0 \vee \sup_{\kappa'' \in [\kappa_0,\kappa']} \{ \| \hat{f}_{\kappa'',n} - \hat{f}_{\kappa' \vee \kappa'',n} \|_{\Lbf^2(\R^{d_1} \times \R^{d_2})} - \sigma_n(\kappa'') \} \eqsp,
\end{equation*}
so that as $\kappa' \leq \kappa$,
\begin{align*}
A_n(\kappa')
	&\leq 0 \vee \sup_{\kappa'' \in [\kappa_0, \kappa]} \{ \| \hat{f}_{\kappa'',n} - f^\star \| + \| \hat{f}_{\kappa',n} - f^\star \|_{\Lbf^2(\R^{d_1} \times \R^{d_2})} - \sigma_n(\kappa'') \}\eqsp, \\
	&= \| \hat{f}_{\kappa',n} - f^\star \| + \left(0 \vee \sup_{\kappa'' \in [\kappa_0,\kappa]} \{ \| \hat{f}_{\kappa'',n} - f^\star \|_{\Lbf^2(\R^{d_1} \times \R^{d_2})} - \sigma_n(\kappa'') \} \right)\eqsp, \\
	&\leq \| \hat{f}_{\kappa',n} - f^\star \| \leq \sigma_n(\kappa')\eqsp.
\end{align*}
Equation~\eqref{eq_controle_proba_adaptatif} follows by taking $\kappa' = \kappa$.

\section{Proof of Lemma~\ref{lem_controle_deviations_M}}
\label{sec_proof_lem_controle_deviations_M}

To simplify the notations, write $\Phi^\star : (t_1,t_2) \mapsto \Phi_{R^\star}(t_1,t_2) \Phi^\star_{Q^{\star,(1)}}(t_1) \Phi^\star_{Q^{\star,(2)}}(t_2)$ the characteristic function of $\bfY$ under the parameters $(R^\star,Q^\star)$.
By definition of $M$ and $M_n$ and since for any complex numbers $a$ and $b$, $||a|^2 - |b|^2| \leq |a-b|( |a|+ |b| )$, for any $\kappa>0$, $S>0$, $\phi \in \Upsilon_{\kappa,S}$, $R^\star$ and $Q^\star$,
\begin{align*}
|M_n(\phi) - M(\phi;\nu_{\text{est}} | R^\star,Q^\star)|
	&\leq \int_{\neighborhood{d_1}{\nu_{\text{est}}} \times \neighborhood{d_2}{\nu_{\text{est}}}}
		\Big|
		(\phi(t_1,t_2) \tilde{\phi}_n(t_1,0) \tilde{\phi}_n(0,t_2) - \tilde{\phi}_n(t_1,t_2) \phi(t_1,0) \phi(0,t_2) ) \\
		&\hspace{2cm} - ( \phi(t_1,t_2) \Phi^\star(t_1,0) \Phi^\star(0,t_2) - \Phi^\star(t_1,t_2) \phi(t_1,0) \phi(0,t_2) )
		\Big| \\
		&\quad \times \left(\Big|
		(\phi(t_1,t_2) \tilde{\phi}_n(t_1,0) \tilde{\phi}_n(0,t_2) - \tilde{\phi}_n(t_1,t_2) \phi(t_1,0) \phi(0,t_2) )\Big| \right.\\
		&\hspace{.8cm} \left.+ \Big|( \phi(t_1,t_2) \Phi^\star(t_1,0) \Phi^\star(0,t_2) - \Phi^\star(t_1,t_2) \phi(t_1,0) \phi(0,t_2) )
		\Big|\right) \rmd t_1 \rmd t_2 \\
	&\leq 2 \|\phi\|_{\infty, \neighborhood{d}{\nu_{\text{est}}}} (1 + \|\phi\|_{\infty,\neighborhood{d}{\nu_{\text{est}}}}) \\
		&\qquad \int_{\neighborhood{d_1}{\nu_{\text{est}}} \times \neighborhood{d_2}{\nu_{\text{est}}}}
		\Big|
		\phi(t_1,t_2) (\tilde{\phi}_n(t_1,0) \tilde{\phi}_n(0,t_2) - \Phi^\star(t_1,0) \Phi^\star(0,t_2)) \\
		&\hspace{2cm} - (\tilde{\phi}_n(t_1,t_2)-\Phi^\star(t_1,t_2)) \phi(t_1,0) \phi(0,t_2) 
		\Big| \rmd t_1 \rmd t_2 \\
	&\leq 4 \|\phi\|_{\infty, \neighborhood{d}{\nu_{\text{est}}}}^2 (1 + \|\phi\|_{\infty,\neighborhood{d}{\nu_{\text{est}}}})^2
		\nu_{\text{est}}^{d} \|\tilde{\phi}_n - \Phi^\star\|_{\infty, \neighborhood{d}{\nu_{\text{est}}}} \\
	&\leq 16 \|\phi\|_{\infty, \neighborhood{d}{\nu_{\text{est}}}}^4	\nu_{\text{est}}^d \|\tilde{\phi}_n - \Phi^\star\|_{\infty, \neighborhood{d}{\nu_{\text{est}}}}
\end{align*}
since $\|\Phi^\star\|_\infty \leq 1$ and $\|\tilde{\phi}_n \|_\infty \leq 1$ by definition. Thus, by (\ref{eq_maj_CUpsilon}) in Lemma~\ref{lem_controle_constantes}, for any $n \geq 1$, $\kappa > 0$, $S > 0$, $x_0 \geq 1 \vee (\frac{d+4/3}{\kappa})^\kappa$ and any probability measures $R^\star$ and $Q^\star$ on $\R^d$,
\begin{multline*}
\sup_{\phi \in \Upsilon_{\kappa,S}}
|M_n(\phi) - M(\phi;\nu_{\text{est}} | R^\star,Q^\star)| \\
	\leq 38\,416 \, \nu_{\text{est}}^d (S\nu_{\text{est}} \vee x_0)^{4\frac{d+1}{\kappa} } \exp\left( 4\kappa (S\nu_{\text{est}} \vee x_0)^{1/\kappa} \right)
		\|\tilde{\phi}_n - \Phi^\star\|_{\infty, \neighborhood{d}{\nu_{\text{est}}}} \eqsp.
\end{multline*}
Let $N^\text{Re}(\epsilon,\nu_{\text{est}} | R^\star,Q^\star)$ (resp. $N^\text{Im}(\epsilon,\nu_{\text{est}} | R^\star,Q^\star)$) be the number of brackets of size $\epsilon$ required to cover $\{y \in \R^d \mapsto \text{Re}( e^{it^\top y} ), t \in \neighborhood{d}{\nu_{\text{est}}}\}$ (resp. with $\text{Im}$ instead of $\text{Re}$), where the size of the bracket $[u,v]$ is $\E_{R^\star,Q^\star}[(v-u)^2(Y)]^{1/2}$. Since all these functions take values in $[-1,1]$ and for all $y, t, t' \in \R^d$, $|e^{i t^\top y} - e^{i t'^\top y}| \leq |(t-t')^\top y| \leq \sqrt{d} \|t-t'\|_\infty \|y\|$, it is possible to obtain a bracket of size $\epsilon$ for each of these two sets from a bracket of size $\epsilon / (\sqrt{d \E_{R^\star,Q^\star}[\|\bfY\|^2]})$ of $(\neighborhood{d}{\nu_{\text{est}}}, \|\cdot\|_\infty)$, which means that
\begin{align*}
N^\text{Re}(\epsilon,\nu_{\text{est}} | R^\star,Q^\star) \vee N^\text{Im}(\epsilon,\nu_{\text{est}} | R^\star,Q^\star) &\leq \left(\frac{4\nu_{\text{est}} \sqrt{d \E_{R^\star,Q^\star}[\|\bfY\|^2]}}{\epsilon} \vee 1\right)^d \eqsp.
\end{align*}
Thus, by \cite{massart:2003}, Theorem~6.8 and Corollary~6.9, there exists a numerical constant $C$ such that for all $x > 0$, $R^\star$ and $Q^\star$,
\begin{equation*}
\po_{R^\star,Q^\star} \left(
	\|\tilde{\phi}_n - \Phi^\star\|_{\infty, \neighborhood{d}{\nu_{\text{est}}}} \geq C\left[\frac{E(R^\star,Q^\star)}{n} + \sqrt{\frac{x}{n}} + 2\frac{x}{n}\right]
\right) \leq 4e^{-x}
\end{equation*}
(the factor 4 is due to the bilateral control on both the real and imaginary part of $\tilde{\phi}_n - \Phi^\star$) where
\begin{align*}
E(R^\star,Q^\star) &= \sqrt{n} \int_0^1 \sqrt{n \wedge d \log \frac{\sqrt{1 \vee 16 \nu_{\text{est}}^2 d \E_{R^\star,Q^\star}[\|\bfY\|^2]}}{u} } \rmd u + \frac{3}2 d \log (1 \vee 16 \nu_{\text{est}}^2 d \E_{R^\star,Q^\star}[\|\bfY\|^2]) \\
	&= \sqrt{n} A \int_0^{1/A} \sqrt{n \wedge d \log \frac1{v}} \rmd u + 3 d \log A \qquad \text{where } A = \sqrt{1 \vee 16 \nu_{\text{est}}^2 d \E_{R^\star,Q^\star}[\|\bfY\|^2]} \\
	&= \sqrt{n} A \left( \sqrt{n} \rme^{-n} + d \int_{A}^{\rme^n} \frac{\sqrt{\log x}}{x^2} \rmd x \right) + 3 d \log A \\
	&\leq d \sqrt{n} A \left( \rme^{-1} + \int_1^{+\infty} \frac{\sqrt{x}}{x^2} \rmd x \right) + 3 d A \\
	&\leq 6 d \sqrt{n} A \\
	&= 6 d \sqrt{n} \sqrt{1 \vee 16 \nu_{\text{est}}^2 d \E_{R^\star,Q^\star}[\|\bfY\|^2]}\eqsp.
\end{align*}
Hence, for all $n \geq 1$, $x > 0$, $R^\star$ and $Q^\star$, with probability at least $1-4e^{-x}$ under $\po_{R^\star,Q^\star}$,
\begin{align*}
\|\tilde{\phi}_n - \Phi^\star\|_{\infty, \neighborhood{d}{\nu_{\text{est}}}}
	&\leq C^\star\left[6 d \sqrt{\frac{1 \vee 16 \nu_{\text{est}}^2 d \E_{R^\star,Q^\star}[\|\bfY\|^2]}{n}} + \sqrt{\frac{x}{n}} + 2\frac{x}{n}\right] \eqsp,
\end{align*}
and finally there exists a numerical constant $c_M$ such that for all $n \geq 1$ and $x > 0$, with probability at least $1-4e^{-x}$,
\begin{multline*}
\sup_{\phi \in \Upsilon_{\kappa,S}} |M_n(\phi) - M(\phi;\nu_{\text{est}})| \\
	\leq c_M \nu_{\text{est}}^d (S\nu_{\text{est}} \vee x_0)^{4\frac{d+1}{\kappa} } \exp\left( 4\kappa (S\nu_{\text{est}} \vee x_0)^{1/\kappa} \right)
		\left[d \sqrt{\frac{1 \vee \nu_{\text{est}}^2 d \E[\|\bfY\|^2]}{n}} \vee \sqrt{\frac{x}{n}} \vee \frac{x}{n}\right]
\end{multline*}
where $x_0 = 1 \vee (\frac{d+4/3}{\kappa})^\kappa$.

\section{Technical results}
\label{sec_technical}

\begin{lem}
\label{lem_controle_somme_ignoble}
For all $d \geq 0$ and all $x > 0$, let $\psi_{x,d}$ be the function defined on $\R_+^*$ by $\psi_{x,d}: u \mapsto u^d x^u u^{-\kappa u}$. Let $x_0 = \left((d+4/3)/\kappa\right)^\kappa$, then for all $x > 0$,
\begin{equation*}
\sum_{m \geq 1} \psi_{x,d}(m) \leq 6 (x \vee x_0)^{\frac{d+1}{\kappa}} \exp( \kappa (x \vee x_0)^{1/\kappa} )\eqsp.
\end{equation*}
\end{lem}
\begin{proof}
For all $x > 0$, there exists $u_{\star}(x)$ such that $\psi_{x,d}$ is nondecreasing on $(0,u_{\star}(x)]$ and nonincreasing on $[u_{\star}(x),+\infty)$. This real number satisfies
\begin{equation*}
u_{\star}(x) = \sup \left\{ u > 0 : u\log(x) -\kappa u\log u - \kappa u + d \geq 0 \right\}\eqsp.
\end{equation*}
Hence, for all $x > 0$, $u_{\star}(x) \geq \rme^{-1} x^{1/\kappa}$ and for all $x \geq (d/\kappa)^{\kappa}$, $u_{\star}(x)\leq x^{1/\kappa}$, so that
\begin{equation*}
\psi_{x,d}(u_{\star}(x_0)) = u_{\star}(x)^d \left( \frac{x^{1/\kappa}}{u_{\star}(x)} \right)^{\kappa u_{\star}(x)}\leq x^{d/\kappa} \exp(\kappa x^{1/\kappa}) \eqsp.
\end{equation*}
Thus,
\begin{align*}
\sum_{m \geq 1} \psi_{x,d}(m)
	&\leq \!\!\!\!\!\sum_{m = 1}^{\lfloor u_{\star}(x) \rfloor-1} \!\!\!\!\!\psi_{x,d}(m)
			+ \psi_{x,d}(\lfloor u_{\star}(x) \rfloor) + \psi_{x,d}(\lceil u_{\star}(x) \rceil) + \!\!\!\!\!\!\!\! \sum_{m \geq \lceil u_{\star}(x) \rceil + 1} \!\!\!\!\!\!\!\psi_{x,d}(m)\eqsp, \\
	&\leq \int_1^{\lfloor u_{\star}(x) \rfloor}\hspace{-.8cm} \psi_{x,d}(u) \rmd u
			+ 2\psi_{x,d}(u_{\star}(x))
			+ \int_{\lceil u_{\star}(x) \rceil}^{\infty} \hspace{-.6cm} \psi_{x,d}(u) \rmd u\eqsp,
\end{align*}
so that for all $x \geq (d/\kappa)^\kappa$,
\begin{align*}
\sum_{m \geq 1} \psi_{x,d}(m)
	&\leq 2x^{d/\kappa} \exp\left(\kappa x^{1/\kappa}\right) + \int_{u \geq 1} \psi_{x,d}(u) \rmd u\eqsp, \\
	&\leq 2x^{d/\kappa} \exp\left(\kappa x^{1/\kappa}\right)
			+ \int_{u \geq 1} u^d \left(\frac{x}{u^\kappa}\right)^u \rmd u\eqsp.
\end{align*}
Let $u = x^{1/\kappa} v$, then
\begin{align*}
\int_{u \geq 1} u^d \left(\frac{x}{u^\kappa}\right)^u \rmd u
	&\leq x^{\frac{d+1}{\kappa}} \int_{v \geq x^{-1/\kappa}} \hspace{-.8cm} v^{d - \kappa x^{1/\kappa} v} \rmd v\eqsp, \\
	&\leq x^{\frac{d+1}{\kappa}} \left( \int_{0 \leq v \leq 1} \exp(-\kappa x^{1/\kappa} v \log v) \rmd v
		+ \int_{v \geq 1} v^{d - \kappa x^{1/\kappa} v} \rmd v \right)\eqsp, \\
	&\leq x^{\frac{d+1}{\kappa}} \left( \exp\left(\kappa x^{1/\kappa} \frac1{\rme}\right)
		+ \int_{v \geq 1} v^{d - \kappa x^{1/\kappa}} \rmd v \right)\eqsp, \\
	&\leq x^{\frac{d+1}{\kappa}} \left( \exp(\kappa x^{1/\kappa})
		+ \frac1{\kappa x^{1/\kappa} - (d+1))} \right)\eqsp, \\
	&\leq x^{\frac{d+1}{\kappa}} \left( \exp(\kappa x^{1/\kappa})
		+ 3 \right)\eqsp,
\end{align*}
when $\kappa x^{1/\kappa} \geq d+4/3$, so that for all $x \geq \left(\frac{d+4/3}{\kappa}\right)^\kappa$,
\begin{align*}
\sum_{m \geq 1} \psi_{x,d}(m)
	\leq 3 x^{\frac{d+1}{\kappa}} (1+\exp(\kappa x^{1/\kappa}))
	\leq 6 x^{\frac{d+1}{\kappa}} \exp(\kappa x^{1/\kappa})\eqsp.
\end{align*}
The case $x \leq ((d+4/3)/\kappa)^\kappa$ follows from the fact that $x \mapsto \psi_{x,d}(m)$ is nondecreasing for all positive integer $m$ and all $d \geq 0$.
\end{proof}

The results established in this section involve the following quantities.
\begin{align}
C_\Upsilon(\kappa,S,\nu) &= \sup_{\phi \in \Upsilon_{\kappa, S}} \| \phi \|_{\Lbf^\infty(\neighborhood{d_1}{\nu}\times \neighborhood{d_2}{\nu})}\eqsp,\label{rk_dvt_Taylor_Upsilon}\\
f_\kappa:\eqsp &u\eqsp \mapsto \sum_{m\geq 1} (m + d/\kappa)^{-\kappa m} u^{m}\eqsp,\label{eq:def:fkappa}\\
g: \eqsp& (\kappa,S)\eqsp\mapsto \sup_{x \geq 1} \left\{(\max(S,1)\rme^{d+2}2^\kappa)^x x^{- \kappa x + 1}\right\}\label{eq:def:g}\eqsp,
\end{align}
where $d = d_1 + d_2$.

\begin{lem}
\label{lem_controle_constantes}
Let $\kappa>0$ and $u_0 = (4/(3\kappa))^\kappa$, then for all $u > 0$,
\begin{equation}
\label{eq_maj_fkappa}
f_\kappa(u)\leq 6 (u \vee u_0)^{1/\kappa}\exp(\kappa (u \vee u_0)^{1/\kappa})\eqsp,
\end{equation}
where $f_{\kappa}$ is defined in \eqref{eq:def:fkappa}.
Let $\kappa, S > 0$ and $x_0 = 1 \vee ((d+4/3)/\kappa)^\kappa$, then for all $\nu > 0$,
\begin{equation}
\label{eq_maj_CUpsilon}
C_\Upsilon(\kappa,S,\nu)\leq 7 (S\nu \vee x_0)^{\frac{d+1}{\kappa}}\exp(\kappa (S\nu \vee x_0)^{1/\kappa})\eqsp,
\end{equation}
where $C_\Upsilon$ is defined in \eqref{rk_dvt_Taylor_Upsilon}.
For all $\kappa,S>0$,
\begin{equation}
\label{eq_maj_g}
g(\kappa,S)\leq 2 \rme^{(d+2)/\kappa} (S \vee 1)^{1/\kappa}\exp\left(2 \kappa \rme^{(d+2)/\kappa} (S \vee 1)^{1/\kappa} \right)\eqsp,
\end{equation}
where $g$ is defined in \eqref{eq:def:g}.
\end{lem}

\begin{proof}
The inequality \eqref{eq_maj_fkappa} follows from Lemma~\ref{lem_controle_somme_ignoble} and the fact that $f_\kappa(u)\leq \sum_{m \geq 1} m^{-\kappa m} u^{m}$ for all $u > 0$. The inequality \eqref{eq_maj_CUpsilon} follows exactly the same proof as the second implication of Lemma~\ref{lem_lien_Mrho_Upsilon}. To prove \eqref{eq_maj_g}, write, for all $\kappa,S > 0$, $\beta(\kappa,S) = (S \vee 1) \rme^{d+2} 2^\kappa$ and consider the function $\psi : x \mapsto \beta(\kappa,S)^{x} x^{-\kappa x+1} = \psi_{\beta(\kappa,S), 1}(x)$ with the notation where $\psi_{x,d}$ is defined in Lemma~\ref{lem_controle_somme_ignoble}. By definition, $g(\kappa,S) = \sup_{x \geq 1} \psi(x)$.
In the proof of Lemma~\ref{lem_controle_somme_ignoble}, it is shown that this function is upper bounded on $\R_+^*$ by $\beta(\kappa,S)^{1/\kappa} \exp(\kappa \beta(\kappa,S)^{1/\kappa})$ as soon as $\beta(\kappa,S) \geq (1/\kappa)^\kappa$, which is always true since
$
\beta(\kappa,S)^{1/\kappa}
	\geq 2 e^{2/\kappa}
	\geq 2 \times 2/\kappa
	\geq 1 / \kappa.
$
\end{proof}

\begin{lem}
\label{lem_troncature}
Let $\phi\in$ and $d = d_1+d_2$. For all $\kappa > 0$, there exists a function $f_\kappa : \R_+ \rightarrow \R_+$ such that for all $S < \infty$, $\nu > 0$, $\phi \in \Upsilon_{\kappa,S}$ and $m \geq d/\kappa$,
\begin{equation*}
\| \phi - T_m \phi \|_{\Lbf^\infty(\neighborhood{d_1}{\nu}\times\neighborhood{d_2}{\nu})} \leq 2^d(S\nu)^{m} m^{-\kappa m + d} f_\kappa(S\nu)\eqsp,
\end{equation*}
where $\Upsilon_{\kappa,S}$, $f_\kappa$ and $ T_m \phi$ are defined in \eqref{eq:def:funcset}, \eqref{eq:def:fkappa} and \eqref{eq:def:tronc}.
\end{lem}
\begin{proof}
Let $\kappa > 0$, $S < \infty$ and $\phi \in \Upsilon_{\kappa,S}$. By definition of $\Upsilon_{\kappa,S}$, for all $x \in \neighborhood{d_1}{\nu}\times \neighborhood{d_2}{\nu}$, 
\begin{equation*}
\phi(x) = \sum_{i \in \N^d} c_i \prod_{a=1}^d x_a^{i_a}\eqsp, \quad \mathrm{where} \quad |c_i| \leq S^{\|i\|_1} \|i\|_1^{- \kappa \|i\|_1}\eqsp.
\end{equation*}
Then, for all $x \in \neighborhood{d_1}{\nu}\times\neighborhood{d_2}{\nu}$, 
\begin{equation*}
(\phi - T_m \phi)(x) = \sum_{i\in\mathsf{I}_d^{m,+}} c_i \prod_{a=1}^d x_a^{i_a}\eqsp,
\end{equation*}
where, for all $m \in \N$, 
\begin{equation}
\label{eq:tuples:fixedsum}
\mathsf{I}_d^m = \left\{ i \in \N^d : \sum_a i_a = m\right\}\eqsp,\quad \mathsf{I}_d^{m,+} = \left\{ i \in \N^d : \sum_a i_a > m\right\}\eqsp.
\end{equation}
This yields
\begin{equation*}
\| \phi - T_m \phi \|_{\Lbf^\infty(\neighborhood{d_1}{\nu}\times\neighborhood{d_2}{\nu})} \leq \sum_{i \in \mathsf{I}_d^{m,+} }\|i\|_1^{- \kappa \|i\|_1} (S\nu)^{\|i\|_1}
\end{equation*}
and
\begin{align*}
\| \phi - T_m \phi \|_{\Lbf^\infty(\neighborhood{d_1}{\nu}\times\neighborhood{d_2}{\nu})}
	\leq \sum_{m' > m} (m')^{-\kappa m'} (S\nu)^{m'} |\mathsf{I}_d^{m'}|
	&\leq \sum_{m' > m} (m')^{-\kappa m'} (S\nu)^{m'} (1 + m')^d\eqsp,\\
	&\leq 2^d \sum_{m' > m} (m')^{-\kappa m'} (S\nu)^{m'} (m')^d\eqsp.
\end{align*}
Therefore, for $m \geq d/\kappa$, as for all $m'> m$, $(m')^{-\kappa m+d} \leq (m+1)^{-\kappa m+d}$,
\begin{align*}
\| \phi - T_m \phi \|_{\Lbf^\infty(\neighborhood{d_1}{\nu}\times\neighborhood{d_2}{\nu})}
	&\leq 2^d (S\nu)^{m} (m+1)^{-\kappa m+d} \sum_{m' >m} (m')^{-\kappa (m'-m)} (S\nu)^{(m'-m)} \eqsp,\\
	&\leq (S\nu)^{m} m^{-\kappa m+d} 2^d \sum_{m' > 0} (m' + d/\kappa)^{-\kappa m'} (S\nu)^{m'} \eqsp,
\end{align*}
which concludes the proof by definition of $f_\kappa$, see \eqref{eq:def:fkappa}.
\end{proof}

\section{Proof of Lemma~\ref{lem_minoration_Mlin}}
\label{sec:minoration_Mlin}

Let $\kappa, \nu > 0$, $S < \infty$, $m \in \N^*$, $\phi \in \Upsilon_{\kappa,S}$ and $h \in \Gcal_{\kappa,S}$.
\begin{align} 
M^\text{lin}(T_m h,T_m \phi;\nu) &= \int_{\mathsf{B}^d_\nu} | (T_m h)(t_1,t_2) (T_m \phi)(t_1,0) (T_m \phi)(0,t_2)\nonumber \\
		&\hspace{2.5cm} - (T_m \phi)(t_1,t_2) (T_m h)(t_1,0) (T_m \phi)(0,t_2)\nonumber \\
		&\hspace{2.5cm} - (T_m \phi)(t_1,t_2) (T_m \phi)(t_1,0) (T_m h)(0,t_2) |^2 \rmd t_1 \rmd t_2\eqsp, \nonumber\\
\label{eq_reecriture_Mlin_avec_Acal}
	&= \| \Acal(\phi,m) h \|_{\Lbf^2(\neighborhood{2d}{\nu})}^2\eqsp,
\end{align}
where $\Acal(\phi,m)$ is a linear operator onto $\Lbf^2(\neighborhood{2d}{\nu})$. Write 
\begin{equation*}
P_1 = (T_m \phi)(\cdot,0) (T_m \phi)(0,\cdot)\eqsp,\quad P_2 = - (T_m \phi)(\cdot,0) (T_m \phi)\eqsp, \quad P_3 = - (T_m \phi)(0,\cdot) (T_m \phi)\eqsp,
\end{equation*}
so that
\begin{equation*}
\Acal(\phi,m) h = (T_m h) P_1 + (T_m h)(0,\cdot) P_2 + (T_m h)(\cdot,0) P_3\eqsp.
\end{equation*}
Let $H$ be the vector of coordinates of $h$ in the canonical basis of $\C[X_1, \dots, X_{2d}]$, for all $(x,y) \in \mathsf{B}^d_\nu$,
\begin{equation*}
h(x,y) = \sum_{i, j \in \N^d} H_{(i,j)} \prod_{a,b=1}^d x_a^{i_a} y_b^{j_b}\eqsp.
\end{equation*}
Then $h = H^\top \Mfrak$ where $\Mfrak$ is the vector such that for all $i \in \N^{2d}$, 
\begin{equation}
\label{eq:def:monomial}
\Mfrak_i = \prod_{a=1}^{2d} X_a^{i_a}\eqsp.
\end{equation}
Let $A$ be the matrix such that the coordinates of $\Acal(\phi,m) h$ in the canonical basis are, for all $i\in\N^{2d}$, 
\begin{equation}
\label{eq:def:coordtmandA}
\mathcal{A}_i = \sum_{j \in \N^{2d}} A_{i,j} H_j\eqsp.
\end{equation}
Likewise, let $J_m$ be the matrix of the operator $T_m$ in the canonical basis: for all $(i,j)\in \N^{2d}\times \N^{2d}$, 
\begin{equation}
\label{eq:def:coordtm}
(J_m)_{i,j} = \one_{\|i\|_1 \leq m}\one_{i=j}\eqsp.
\end{equation}
 Let $f$, $(P_{1,i})_{i \in \N^{2d}}$, $(P_{2,i})_{i \in \N^{2d}}$ and $(P_{3,i})_{i \in \N^{2d}}$ the vector of coordinates of $\phi$, $P_1$, $P_2$ and $P_3$ in the canonical basis. Then, for all $(i,j) \in \N^{2d}\times \N^{2d}$, 
\begin{equation*}
A_{i,j} = (P_{1,i-j} + P_{2,i-j} + P_{3,i-j}) \one_{\|j\|_1 \leq m}\eqsp,
\end{equation*}
with the convention $P_{1,i} = P_{2,i} = P_{3,i} = 0$ if there exists $a \in \{1, \dots, 2d\}$ such that $i_a < 0$. For all $(i,j) \in (\N^d)^2$,
\begin{equation*}
\begin{cases}
\displaystyle P_{1,(i,j)} = f_{(i,0)} f_{(0,j)} \one_{\|i\|_1 \leq m}
					\one_{\|j\|_1 \leq m}\eqsp, \\
\displaystyle P_{2,(i,j)} = - \sum_{u \in \N^d : u \leq i} f_{(u,j)} f_{(i-u,0)} 
					\one_{\|u\|_1 \leq m} \one_{\|i-u\|_1 \leq m} \one_{\|j\|_1 \leq m}\eqsp, \\
\displaystyle P_{3,(i,j)} = - \sum_{v \in \N^d : v \leq j} f_{(i,v)} f_{(0,j-v)} 
					\one_{\|i\|_1 \leq m} \one_{\|v\|_1 \leq m} \one_{\|j-v\|_1 \leq m}\eqsp.
\end{cases}
\end{equation*}

\begin{lem}
\label{lem_prop_A}
For all $(i,j) \in \N^d\times \N^d$, 
\begin{enumerate}[i)]
\item $A_{i,j} = 0$ if there exists $a \in \{1, \dots, 2d\}$ such that $j_a \geq i_a$ ($A$ is lower triangular)~;
\item $A_{i,j} = 0$ if $\|j\|_1 > m$, so that $A J_m = A$~;
\item $A_{i,j} = 0$ if $\|i\|_1 > \|j\|_1 + 2 m$, so that $A J_{m'} = J_{m' + 2m} A J_{m'}$ for all $m' \in \N$~;
\item $A_{i,i} = -\phi(0)^2 = -1$~;
\item the coefficient $A_{i,j}$ is upper bounded as follows:
\begin{equation*}
|A_{i,j}| \leq S^{\|i - j\|_1} ( \|i-j\|_1/2)^{- \kappa \|i-j\|_1} \left\{1 + (1+\|i_1-j_1\|_1)^{d_1+1} + (1+\|i_2-j_2\|_1)^{d_2+1}\right\}\eqsp,
\end{equation*}
where $i = (i_1,i_2)\in\N^{d_1}\times \N^{d_2}$ and $j = (j_1,j_2)\in\N^{d_1}\times \N^{d_2}$.
\end{enumerate}
\end{lem}

\begin{proof}
Items i) to iv) are direct consequences of the definitions. Let $(i,j) \in (\N^d)^2$, by definition of $\Upsilon_{\kappa,S}$, $|f_k| \leq S^{\|k\|_1} \|k\|_1^{- \kappa \|k\|_1}$ for all $k \in \N^{2d}$. Then, by concavity of $x \mapsto x \log x$,
\begin{align*}
|P_{1, (i,j)}|
	&\leq S^{\|i\|_1 + \|j\|_1} \exp\left(-2\kappa \left[ \frac12 \|i\|_1 \log \|i\|_1 + \frac12 \|j\|_1 \log \|j\|_1 \right] \right)\eqsp,\\
	&\leq S^{\|i\|_1 + \|j\|_1} \exp\left(-\kappa (\|i\|_1 + \|j\|_1) \log \left(\frac{\|i\|_1 + \|j\|_1}2 \right) \right)\eqsp, \\
	&\leq S^{\|i\|_1 + \|j\|_1} \left(\frac{\|i\|_1 + \|j\|_1}2 \right)^{-\kappa (\|i\|_1 + \|j\|_1)}\eqsp.
\end{align*}
Similarly, using definition \eqref{eq:tuples:fixedsum},
\begin{align*}
|P_{2, (i,j)}|
	&\leq S^{\|i\|_1 + \|j\|_1} \sum_{u \in \N^d : u \leq i} (\|u\|_1 + \|j\|_1)^{-\kappa (\|u\|_1 + \|j\|_1)} \|i-u\|_1^{-\kappa \|i-u\|_1}\eqsp, \\
	&\leq S^{\|i\|_1 + \|j\|_1} \sum_{k = 0}^{\|i\|_1} |\mathsf{I}_d^{k}| (k + \|j\|_1)^{-\kappa (k + \|j\|_1)} (\|i\|_1 - k)^{-\kappa (\|i\|_1 - k)}\eqsp, \\
	&\leq S^{\|i\|_1 + \|j\|_1} \sum_{k = 0}^{\|i\|_1} |\{0, \dots, \|i\|_1\}^d|
		\exp\Big(-2\kappa \Big[ \frac12 (k + \|j\|_1) \log (k + \|j\|_1) \\
			&\hspace{7.5cm}+ \frac12 (\|i\|_1 - k) \log (\|i\|_1 - k) \Big] \Big)\eqsp, \\
	&\leq S^{\|i\|_1 + \|j\|_1} (\|i\|_1 + 1)^{d+1} \exp\left(-\kappa (\|i\|_1 + \|j\|_1) \log \left(\frac{\|i\|_1 + \|j\|_1}2 \right) \right)\eqsp,\\
	&\leq S^{\|i\|_1 + \|j\|_1} (\|i\|_1 + 1)^{d+1} \left(\frac{\|i\|_1 + \|j\|_1}2 \right)^{-\kappa (\|i\|_1 + \|j\|_1)}\eqsp.
\end{align*}
and
\begin{equation*}
|P_{3, (i,j)}| \leq S^{\|i\|_1 + \|j\|_1} (\|j\|_1 + 1)^{d_2+1} \left(\frac{\|i\|_1 + \|j\|_1}2 \right)^{-\kappa (\|i\|_1 + \|j\|_1)}\eqsp,
\end{equation*}
which concludes the proof.
\end{proof}
For all $i \geq 0$, let $P_i$ be the $i$-th Legendre polynomial and $P^\text{norm}_i$ its normalized version defined as in \eqref{eq:norm:legendre}. Let $\Bbf$ be the change-of-basis matrix from the canonical basis formed by the monomials $(\Mfrak_i)_{i\in\N^d}$, where $d = d_1+d_2$ and $\Mfrak_i$ is defined in \eqref{eq:def:monomial}, to the basis generated by the normalized Legendre polynomials: for all $i \in \N^d$,
\begin{equation}
\label{eq_chgt_base_Legendre}
\Pbf^\text{norm}_i(X_1, \dots, X_d) = \sum_{j \in \N^d} \Bbf_{i,j} \prod_{a=1}^d X_a^{j_a}\eqsp,
\end{equation}
where $\Pbf^\text{norm}_i$ is defined in \eqref{eq:prod:legendre}. Then, $\Bbf_{i,j} = 0$ if there exists $1\leq a \leq d$ such that $j_a > i_a$ or such that $i_a - j_a$ is an odd integer. Otherwise, for all $k \in \N^d$ such that $k_a \leq i_a/2$ for all $a \in \{1, \dots, d\}$,
\begin{align}
\Bbf_{i,i-2k}
	&= \nu^{-d/2} \left(\prod_{a=1}^d (i_a + 1/2)\right)^{1/2} 2^{-\|i\|_1} \nu^{-\|i-2k\|_1} (-1)^{\|k\|_1} \prod_{a=1}^d \binom{i_a}{k_a} \binom{2i_a-2k_a}{i_a}\eqsp, \nonumber\\
	&= \nu^{-d/2} \left(\prod_{a=1}^d (i_a + 1/2)\right)^{1/2} 2^{-\|i\|_1} \nu^{-\|i-2k\|_1} (-1)^{\|k\|_1} \prod_{a=1}^d \binom{i_a-k_a}{k_a} \binom{2i_a-2k_a}{i_a-k_a}\eqsp.\label{eq:def:changeofbasismatrix}
\end{align}

\begin{lem}
Let $h \in \Lbf^2(\neighborhood{d_1}{\nu}\times\neighborhood{d_2}{\nu})$ and $X$ be the vector of coordinates of $T_m h$ in the Legendre polynomials basis. Then,
\begin{equation*}
\| T_m h \|_{\Lbf^2(\neighborhood{d_1}{\nu}\times\neighborhood{d_2}{\nu})}^2 = \| X \|^2
\end{equation*}
and
\begin{equation*}
\|\Acal(\phi,m) h\|_{\Lbf^2(\neighborhood{d_1}{\nu}\times\neighborhood{d_2}{\nu})}^2 = \| X^\top \Bbf A^\top \Bbf^{-1} \|^2 = \| X^\top J_m \Bbf J_m A^\top J_{3m} \Bbf^{-1} J_{3m} \|^2\eqsp,
\end{equation*}
where $A$, $J_m$ and $\Bbf$ are defined in \eqref{eq:def:coordtmandA}, \eqref{eq:def:coordtm} and \eqref{eq_chgt_base_Legendre}.
\end{lem}

\begin{proof}
Let $h \in \Lbf^2(\neighborhood{d_1}{\nu}\times\neighborhood{d_2}{\nu})$ and $\Lfrak$ be the vector of Legendre polynomials. By definition of $J_m$, as $\Lfrak = \Bbf \Mfrak$, by \eqref{eq_chgt_base_Legendre},
\begin{equation*}
T_m h = X^\top \Lfrak = X^\top \Bbf \Mfrak = (J_m H)^\top \Mfrak\eqsp.
\end{equation*}
Then, $H^\top J_m = X^\top \Bbf = X^\top \Bbf J_m$ and
\begin{align*}
\Acal(\phi, m) h = (A H)^\top \Mfrak = (A J_m H)^\top \Mfrak = H^\top J_m A^\top (\Bbf^{-1} \Lfrak)
	&= X^\top \Bbf A^\top \Bbf^{-1} \Lfrak \\
	&= X^\top J_m \Bbf J_m A^\top J_{3m} \Bbf^{-1} J_{3m} \Lfrak
\end{align*}
by Lemma~\ref{lem_prop_A} and the fact that $J_m \Bbf^{-1} = J_m \Bbf^{-1} J_m$ since $\Bbf^{-1}$ is lower triangular, so that $X^\top = X^\top J_m$. The proof is concluded by noting that Legendre plolynomials form an orthonormal basis of $\Lbf^2(\neighborhood{d_1}{\nu}\times\neighborhood{d_2}{\nu})$, the operator $\Lfrak^\top: \Lbf^2(\neighborhood{d_1}{\nu}\times\neighborhood{d_2}{\nu}) \rightarrow \Lbf^2(\neighborhood{d_1}{\nu}\times\neighborhood{d_2}{\nu})$ is then norm preserving.
\end{proof}
A lower bound for $\|\Acal(\phi,m) h\|_{\Lbf^2(\neighborhood{d_1}{\nu}\times\neighborhood{d_2}{\nu})}$ may then be obtained by lower bounding the smallest singular values of $J_m \Bbf J_m$, $J_m A^\top J_{3m}$ and $J_{3m} \Bbf^{-1} J_{3m}$ as
\begin{align}
\inf_{h \in \Gcal_{\kappa,S}} \frac{\|\Acal(\phi,m) h\|_{\Lbf^2(\neighborhood{d_1}{\nu}\times\neighborhood{d_2}{\nu})}}{\|T_m h\|_{\Lbf^2(\neighborhood{d_1}{\nu}\times\neighborhood{d_2}{\nu})}}
	&\geq \inf_{X \in \text{Im}(J_m)} \frac{\| X^\top J_m \Bbf J_m A^\top J_{3m} \Bbf^{-1} J_{3m} \|}{\|X^\top\|}\eqsp,\nonumber\\
	&\geq \sigma_{\rk(J_m)}(J_m \Bbf J_m)
			\sigma_{\rk(J_m)}(J_{3m} A J_m)
			\sigma_{\rk(J_{3m})}(J_{3m} \Bbf^{-1} J_{3m})\eqsp,\nonumber\\
\label{eq_lien_minorationAcal_valeurssingulieres}
	&= \sigma_1(J_m \Bbf^{-1})^{-1}
			\sigma_{\rk(J_m)}(A J_m)
			\sigma_1(J_{3m} \Bbf)^{-1}\eqsp.
\end{align}
The following lemmas allow to control the three terms of equation~\eqref{eq_lien_minorationAcal_valeurssingulieres}.

\begin{lem}
Let $d = d_1+d_2$. For all $m \in \N^*$ and all $\nu>0$, 
\begin{equation*}
\sigma_1(J_m \Bbf) \leq \nu^{-d/2} m^d 4^m (\nu^{-1} \vee 1)^m\eqsp,
\end{equation*}
where $J_m$ and $\Bbf$ are defined in \eqref{eq:def:coordtm} and \eqref{eq_chgt_base_Legendre}.
\end{lem}

\begin{proof}
For all $k,i \in \N$ such that $k \leq i$,
\begin{equation*}
\binom{i}{k} \leq \binom{i}{i/2} \sim 2^{i} / \sqrt{\pi i/2} \quad\mathrm{and}\quad \binom{2i-2k}{i} \leq \binom{2i}{i} \sim 4^{i} / \sqrt{\pi i}\eqsp.
\end{equation*}
Thus, by \eqref{eq:def:changeofbasismatrix}, for all $\nu > 0$,
\begin{equation*}
|\Bbf_{i,i-2k}|\leq \nu^{-d/2} \prod_{a=1}^d (4/\nu)^{i_a} \nu^{2k_a}\leq \nu^{-d/2} 4^{\|i\|_1} \left( \nu^{-1} \vee 1 \right)^{\|i\|_1 - 2\|k\|_1} \eqsp.
\end{equation*}
Then,
\begin{multline*}
\sigma_1(J_m \Bbf J_m)
	\leq |\{(i,j) \in \N^d\times \N^d : (J_m \Bbf J_m)_{i,j} \neq 0 \}|^{1/2} \|J_m \Bbf J_m\|_\infty
	\\ 
\leq \rk(J_m) \|J_m \Bbf J_m\|_\infty \leq m^d \|J_m \Bbf J_m\|_\infty\eqsp,
\end{multline*}
which yields $\sigma_1(J_m \Bbf J_m) \leq \nu^{-d/2} m^d 4^m (\nu^{-1} \vee 1)^m$.
\end{proof}

\begin{lem}
Let $d = d_1+d_2$. For all $m \in \N^*$ and all $\nu > 0$, 
\begin{equation*}
\sigma_1(J_m \Bbf^{-1}) \leq \sqrt{2}2^d m^{(d+1)/2} \nu^{d/2} (\nu \vee 1)^m \eqsp,
\end{equation*}
where $J_m$ and $\Bbf$ are defined in \eqref{eq:def:coordtm} and \eqref{eq_chgt_base_Legendre}.
\end{lem}

\begin{proof}
Write $\Lfrak$ the vector of Legendre polynomials. By definition of $\Mfrak$ and $\Bbf$, see \eqref{eq:def:monomial} and \eqref{eq_chgt_base_Legendre}, $\Lfrak = \Bbf \Mfrak$ and for all $i \in \N^d$,
\begin{equation*}
\left\| \Mfrak_i \right\|_{\Lbf^2(\neighborhood{d_1}{\nu}\times \neighborhood{d_2}{\nu})}^2 = \left\| \sum_{j\in\N^d} (\Bbf^{-1})_{i,j} \Lfrak_j \right\|_{\Lbf^2(\neighborhood{d_1}{\nu}\times \neighborhood{d_2}{\nu})}^2 = \sum_{j\in\N^d} (\Bbf^{-1})_{i,j}^2\eqsp,
\end{equation*}
as Legendre polynomials form an orthonormal basis of $\Lbf^2(\neighborhood{d_1}{\nu}\times \neighborhood{d_2}{\nu})$. Then, using that 
\begin{equation*}
\| \Mfrak_i \|_{\Lbf^2(\neighborhood{d_1}{\nu}\times \neighborhood{d_2}{\nu})}^2 = \prod_{a=1}^d \frac{2 \nu^{2i_a+1}}{2i_a+1}\eqsp,
\end{equation*}
yields
\begin{align*}
\| J_m \Bbf^{-1} J_m \|_F^2
	&\leq \begin{cases}
		2^d (m+1)^d \nu^d & \text{if } \nu \leq 1\eqsp, \\
		2^d (m+1)^{d+1} \nu^{2m+d} & \text{if } \nu > 1\eqsp.
	\end{cases}
\end{align*}
The proof is concluded by $\sigma_1(J_m \Bbf^{-1}) = \sigma_1(J_m \Bbf^{-1} J_m) \leq \| J_m \Bbf^{-1} J_m \|_F$.
\end{proof}

\begin{lem}
Let $d = d_1+d_2$. Then,
\begin{equation*}
\sigma_{\rk(J_m)}(A J_m) \geq \com{4^{-1}\,(2\sqrt{2})^{-d} m^{-d-1} (d\rme)^{-3m}} g(\kappa,S,d_1,d_2)^{-3m}\eqsp,
\end{equation*}
where $g$ is defined in \eqref{eq:def:g}.
\end{lem}

\begin{proof}
By Lemma~\ref{lem_prop_A}, $A = -J_m + N$ where $N$ is a $(3m+1)$-nilpotent strict lower triangular matrix. Let $D = - \sum_{k=0}^{3m} N^k$ with the convention $N^0 = J_m$, then $D A = J_m$, and
\begin{equation*}
\sigma_{\rk(J_m)}(A) \geq \sigma_1(D)^{-1}\eqsp.
\end{equation*}
Therefore,
\begin{equation*}
\sigma_{\rk(J_m)}(A)^{-1}\leq \sigma_1\left(- \sum_{k=0}^{3m-1} N^k\right) \leq \sigma_1\left(J_{3m} \sum_{k=0}^{3m-1} N^k J_m \right)\eqsp \leq (\rk(J_m) \rk(J_{3m}))^{1/2} \sum_{k=0}^{3m} \| N^k \|_\infty\eqsp.
\end{equation*}
Therefore,
\begin{equation*}
\sigma_{\rk(J_m)}(A)^{-1}\leq ((3m+1)(m+1))^{d/2} \left(1 + \sum_{k=1}^{3m} \sup_{i, j \in \N^d} \left|
		\underset{i = a^{(0)} \leq a^{(1)} \leq \dots \leq a^{(k)} = j}{\sum_{a^{(0)}, a^{(1)}, \dots, a^{(k)} \in \N^d \text{ distincts}}} \prod_{u=1}^{k} N_{a^{(u-1)}, a^{(u)}}
		\right| \right)\eqsp.
\end{equation*}
Let $k \in \N^*$ and $i = a^{(0)} \leq a^{(1)} \leq \dots \leq a^{(k)} = j$ distinct in $\N^d$.
By Lemma~\ref{lem_prop_A}, writing for all $u\geq 0$, $a^{(u)} = (a_1^{(u)},a_2^{(u)})\in\N^{d_1}\times\N^{d_2}$,
\begin{multline*}
|N_{a^{(u-1)}, a^{(u)}}|
	\leq S^{\|a^{(u)} - a^{(u-1)}\|_1} ( \| a^{(u)} - a^{(u-1)}\|_1/2)^{- \kappa \| a^{(u)} - a^{(u-1)}\|_1} \\
\times \left\{1 + (1+\| a_1^{(u)} - a_1^{(u-1)}\|_1)^{d_1+1} + (1+\| a_2^{(u)} - a_2^{(u-1)}\|_1)^{d_2+1}\right\}
\end{multline*}
so that
\begin{align*}
\Bigg| \prod_{u=1}^{k} & N_{a^{(u-1)}, a^{(u)}} \Bigg| \\
	&\leq \com{} S^{\sum_{u=1}^k \|a^{(u)} - a^{(u-1)}\|_1} \exp\left(
			-\kappa k \sum_{u=1}^k \frac1{k} \|a^{(u)} - a^{(u-1)}\|_1 \log (\|a^{(u)} - a^{(u-1)}\|_1/2)
		\right)\\
&\hspace{2.5cm}\times\com{\prod_{u=1}^{k}(\|a_1^{(u)} - a_1^{(u-1)}\|_1 + 1)^{d_1+1}(\|a_2^{(u)} - a_2^{(u-1)}\|_1 + 1)^{d_2+1}}\eqsp,\\
	&\leq \com{} S^{\|j-i\|_1} \com{\exp\left((d_1+1)\|j_1-i_1\|_1 + (d_2+1)\|j_2-i_2\|_1\right)}\exp\left(-\kappa \|j-i\|_1\log \frac{\|j-i\|_1}{\com2k}\right)\eqsp, \\
	&\leq \com{} S^{\|j-i\|_1}(\rme^{d_1+1})^{\|j_1-i_1\|_1}(\rme^{d_2+1})^{\|j_2-i_2\|_1} \left( \frac{\|j-i\|_1}{\com2k} \right)^{- \kappa \|j-i\|_1}\eqsp,
\end{align*}
using that $x \mapsto x \log x$ is convex, $\log(1+x)\leq x$ for $x\geq 0$ and using
\begin{equation*}
\sum_{u=1}^k \|a^{(u)} - a^{(u-1)}\|_1 = \sum_{u=1}^k (\|a^{(u)}\|_1 - \|a^{(u-1)}\|_1) = \|j-i\|_1\eqsp.
\end{equation*}
It remains to count
\begin{equation*}
\mathsf{s}^k_{i,j} = \underset{i = a^{(0)} \leq a^{(1)} \leq \dots \leq a^{(k)} = j}{\sum_{a^{(0)}, a^{(1)}, \dots, a^{(k)} \in \N^{2d} \text{ distincts}}} 1\eqsp.
\end{equation*}
This sum counts the number of paths connecting $i$ and $j$, going away from 0 in $\N^d$ and made of $k$ steps with non-zero length. 
Thus, it is upper bounded by the number of paths of length $\|j-i\|_1$ going away from zero in $\N^d$ and made of $k$ steps with nonzero length. Such a path is entirely described by the direction of each step ($d$ possibilities each) and the length of each step (or equivalently the distance travelled after each of the first $k-1$ steps, which is equivalent to choosing $k-1$ distinct integers in $\{1, \dots, \|j-i\|_1 - 1\}$). Therefore 
\begin{equation*}
\mathsf{s}^k_{i,j}\leq d^k \binom{\|j-i\|_1 - 1}{k-1} \leq d^k \|j-i\|_1^k / k! \leq d^k \|j-i\|_1^k (\rme/k)^k\eqsp,
\end{equation*}
and
\begin{align*}
\sigma_{\rk(J_m)}(A J_m)^{-1}
	&\leq \com{((3m+1)(m+1))^{d/2}} \left(1 + \sum_{k=1}^{3m} (d \rme)^k \sup_{\ell \geq k} (S\com{\rme^{d_1+1}\rme^{d_2+1}2^{\kappa}})^\ell \left( \frac{\ell}{k} \right)^{- \kappa \ell + k} \right)\eqsp, \\
	&\leq \com{((3m+1)(m+1))^{d/2}} \left(1 + \sum_{k=1}^{3m} \left(d \rme \sup_{x \geq 1} (S\com{\rme^{d_1+1}\rme^{d_2+1}2^{\kappa}})^x x^{- \kappa x + 1}\right)^k \right)\eqsp, \\
	&\leq \com{((3m+1)(m+1))^{d/2} (3m+1)} \max \left(1, d \rme \sup_{x \geq 1} (S\com{\rme^{d_1+1}\rme^{d_2+1}2^{\kappa}})^x x^{- \kappa x + 1} \right)^{3m}\eqsp,\\
&\leq 4(2\sqrt{2})^dm^{d+1}\max \left(1, d \rme \sup_{x \geq 1} (S\com{\rme^{d_1+1}\rme^{d_2+1}2^{\kappa}})^x x^{- \kappa x + 1} \right)^{3m}\eqsp,
\end{align*}
which concludes the proof by \eqref{eq:def:g}.
\end{proof}
The proof of Lemma~\ref{lem_minoration_Mlin} may then be completed. By equations~\eqref{eq_reecriture_Mlin_avec_Acal} and~\eqref{eq_lien_minorationAcal_valeurssingulieres} and the three above lemmas, there exists a numerical constant $\sfc > 0$ such that
\begin{align*}
M^\text{lin}(T_m h,&T_m \phi;\nu)\\
	&\geq \sigma_1(J_m \Bbf^{-1})^{-2}
			\sigma_{\rk(J_m)}(A J_m)^2
			\sigma_1(J_{3m} \Bbf)^{-2} \|T_m h\|_{\Lbf^2(\neighborhood{d_1}{\nu}\times \neighborhood{d_2}{\nu})}^2\eqsp, \\
	&\geq \sfc (4^{-d} m^{-d-1} \nu^{-d} (\nu \vee 1)^{-2m})\times((2\sqrt{2})^{-2d} m^{-2d-2} (d\rme)^{-6m} g(\kappa,S,d_1,d_2)^{-6m})\\
		&\qquad \times (\nu^d (3m)^{-2d} 4^{-6m} \left(\nu^{-1} \vee 1\right)^{-6m}) \|T_m h\|_{\Lbf^2(\neighborhood{d_1}{\nu}\times \neighborhood{d_2}{\nu})}^2\eqsp, \\
	&\geq \sfc (4\sqrt{2})^{-2d}
		(4\rme)^{-6m}
		m^{-5d-3}
		(\nu \vee \nu^{-3})^{-2m}
		g(\kappa,S,d_1,d_2)^{-6m}
		d^{-6m}
		\|T_m h\|_{\Lbf^2(\neighborhood{d_1}{\nu}\times \neighborhood{d_2}{\nu})}^2\eqsp.
\end{align*}

\section{Proofs of Section~\ref{sec:proof:compromis:M}}
\label{sec_proofs_minoration_M}

\subsection{Proof of Lemma~\ref{lem_partie_quadratique}}

Let $\kappa, \nu, S > 0$, $m \geq d/\kappa$, $\phi \in \Upsilon_{\kappa,S}$ and $h \in \Gcal_{\kappa,S}$ and write $V = h - T_m h$ and $U = \phi - T_m \phi$.
Using the inequality $|a+b|^2 \geq |a|^2/2 - |b|^2$ for all $(a,b)\in\mathbb{C}^2$, 
\begin{equation*}
M^\text{lin}(h,\phi;\nu)\geq \frac12 M^\text{lin}(T_m h,\phi;\nu) - 9 (2\nu)^d \|V\|_\infty^2 \|\phi\|_\infty^4 \eqsp.
\end{equation*}
By Lemma~\ref{lem_troncature}, $\|V\|_\infty\leq 2^d(S\nu)^{m}m^{-\kappa m+d}f_\kappa(S\nu)$ so that 
\begin{equation*}
M^\text{lin}(h,\phi;\nu)\geq \frac12 M^\text{lin}(T_m h,\phi;\nu)- 9 (2\nu)^d 2^{2d}(S\nu)^{2m}m^{-2\kappa m+2d} f_\kappa(S\nu)^2 C_\Upsilon^4(\kappa,S,\nu)\eqsp.
\end{equation*}
Similarly,
\begin{align*}
M^\text{lin}(h,\phi;\nu)
	&\geq \frac12 M^\text{lin}(h,T_m \phi;\nu)
			- (2\nu)^d \|U\|_\infty^2 \| h \|_\infty^2 ( 6 \|T_m \phi\|_\infty + 3 \|U\|_\infty )^2\eqsp, \\
	&\geq \frac12 M^\text{lin}(h,T_m \phi;\nu)
			- (2\nu)^d 2^{2d}(S\nu)^{2m}m^{-2\kappa m+2d} f_\kappa(S\nu)^2 (2C_\Upsilon(\kappa,S,\nu))^2\\
	&\qquad \qquad		\times \com{( 6C_\Upsilon(\kappa,S,\nu) + 3\times2C_\Upsilon(\kappa,S,\nu) )^2}\eqsp.
\end{align*}
Therefore, there exists a constant $\tilde\sfc$ such that
\begin{equation*}
M^\text{lin}(h,\phi;\nu)
	\geq \frac1{4} M^\text{lin}(T_m h,T_m \phi;\nu)
			- \tilde\sfc C_\Upsilon^4(\kappa,S,\nu) 2^{2d}(2\nu)^d (S\nu)^{2m}m^{-2\kappa m+2d} f_\kappa(S\nu)^2\eqsp.
\end{equation*}
Then, by Lemma~\ref{lem_minoration_Mlin}, there exists $\sfc > 0$ such that
\begin{multline*}
M^\text{lin}(h,\phi;\nu)
	\geq \sfc (4\sqrt{2})^{-2d}
		(4\rme)^{-6m}
		m^{-5d-3}
		(\nu \vee \nu^{-3})^{-2m}
		g(\kappa,S,d_1,d_2)^{-6m}
		d^{-6m}
		\|T_m h\|_{\Lbf^2(\neighborhood{d_1}{\nu}\times \neighborhood{d_2}{\nu})}^2\\
			- \tilde\sfc C_\Upsilon^4(\kappa,S,\nu) 2^{2d}(2\nu)^d (S\nu)^{2m}m^{-2\kappa m+2d} f_\kappa(S\nu)^2\eqsp.
\end{multline*}
Finally, by Lemma~\ref{lem_troncature} and the inequality 
\begin{equation*}
\|T_m h\|_{\Lbf^2(\neighborhood{d_1}{\nu}\times \neighborhood{d_2}{\nu})}^2
	\geq \|h\|_{\Lbf^2(\neighborhood{d_1}{\nu}\times \neighborhood{d_2}{\nu})}^2 / 2
		- \|h-T_m h\|_{\Lbf^2(\neighborhood{d_1}{\nu}\times \neighborhood{d_2}{\nu})}^2\eqsp,
\end{equation*}
\begin{multline*}
M^\text{lin}(h,\phi;\nu)\geq \com{(\sfc/2)\alpha(d_1,d_2,m,\nu,\kappa,S)} \|h\|_{\Lbf^2(\neighborhood{d_1}{\nu}\times \neighborhood{d_2}{\nu})}^2\\
 - (\sfc\com{\alpha(d_1,d_2,m,\nu,\kappa,S)} + \tilde\sfc C_\Upsilon^4(\kappa,S,\nu)) 2^{2d}(2\nu)^d (S\nu)^{2m}m^{-2\kappa m+2d} f_\kappa(S\nu)^2\eqsp,
\end{multline*}
where
\begin{equation*}
\alpha(d_1,d_2,m,\nu,\kappa,S) = (4\sqrt{2})^{-2d}
		(4\rme)^{-6m}
		m^{-5d-3}
		(\nu \vee \nu^{-3})^{-2m}
		g(\kappa,S,d_1,d_2)^{-6m}
		d^{-6m}\eqsp.
\end{equation*}
This concludes the proof.

\subsection{Proof of Proposition~\ref{prop:compromis:Mlin}}

In this proof the subscript $\Lbf^2(\neighborhood{d_1}{\nu}\times \neighborhood{d_2}{\nu})$ is dropped from the notation $\|h\|$ for better clarity. As $S\nu \geq x_0 \vee u_0$ where $x_0$ and $u_0$ are defined in Lemma~\ref{lem_controle_constantes}, then by Lemma~\ref{lem_controle_constantes} and Lemma~\ref{lem_partie_quadratique}, for all $m \geq d/\kappa$,
\begin{align*}
M^\text{lin}(h,\phi;\nu)
	&\geq \frac{\sfc}2 \alpha(m,\nu,\kappa,S) \left(\|h\|^2-72 (8\nu)^d (S\nu)^{2m + 2/\kappa} \exp(2\kappa (S\nu)^{1/\kappa})) m^{-2\kappa m+2d}\right)\\
	&\qquad - \tilde\sfc \, 7^4 \cdot 36 (8\nu)^d (S\nu)^{2m + (4d+6)/\kappa} \exp(6\kappa (S\nu)^{1/\kappa}) m^{-2\kappa m+2d} \eqsp,
\end{align*}
where $\alpha$ is defined in \eqref{eq:def:alpha}. The proposition will follow from a careful choice of $m$ depending on $\|h\|$.

Since $x^{1/\kappa} \leq \kappa^{-1}\exp(\kappa x^{1/\kappa}/2)$ for all $x > 0$,
\begin{align*}
M^\text{lin}(h,\phi;\nu)
 	&\geq \frac{\sfc}2\alpha(m,\nu,\kappa,S)\left(\|h\|^2 - 72 (8\nu)^d (S\nu)^{2m} \kappa^{-2} \exp(3\kappa (S\nu)^{1/\kappa})) m^{-2\kappa m+2d} \right) \\
 	&\qquad - \tilde\sfc \, 7^4 \cdot 36 (8\nu)^d (S\nu)^{2m} \kappa^{-4d-6} \exp((2d+9)\kappa (S\nu)^{1/\kappa}) m^{-2\kappa m+2d}\eqsp. 
\end{align*}
Assume that
\begin{equation}
\label{eq_contrainte_1}
(8\nu)^d (S\nu)^{2m} \exp\left(3\kappa (S\nu)^{1/\kappa})\right) m^{-2\kappa m+2d} \leq \frac{\kappa^2}{144}\|h\|^2\eqsp.
\end{equation}
Then,
\begin{multline*}
M^\text{lin}(h,\phi;\nu) \geq \frac{\sfc}{4}\alpha(m,\nu,\kappa,S)\|h\|^2\\
 		- \tilde\sfc \, 7^4 \cdot 36 (8\nu)^d (S\nu)^{2m} \kappa^{-4d-6} \exp\left((2d+9)\kappa (S\nu)^{1/\kappa}\right) m^{-2\kappa m+2d}\eqsp.
\end{multline*}
The constraint \eqref{eq_contrainte_1} can be written 
\begin{equation}
\label{eq:const:1}
m^{\kappa m - a_1} b_1^m \geq c_1\|h\|^{-1}\eqsp,
\end{equation}
where $a_1 = d$, $b_1 = (S\nu)^{-1}$ and $c_1^2 = 144 \kappa^{-2} (8 \nu)^d \exp(3\kappa (S\nu)^{1/\kappa})$. On the other hand, if $S \geq 1$,
\begin{align*}
\alpha(m,\nu,\kappa,S)
	&= 8^{-d} m^{-5d-3} (4\rme d)^{-6m}(\nu \vee \nu^{-3})^{-2m} g(\kappa,S)^{-6m}\eqsp, \\
	&\geq 8^{-d} m^{-5d-3} (4\rme d)^{-6m}(\nu \vee \nu^{-3})^{-2m}
			2^{-6m} (\rme^{d+2} S)^{-6m/\kappa} \exp( -12\kappa m(\rme^{d+2} S)^{1/\kappa})\eqsp, \\
	&= 8^{-d} (\rme^{d+2} S)^{-6m
/\kappa} m^{-5d-3} (8\rme d)^{-6m}(\nu \vee \nu^{-3})^{-2m} \exp( -12\kappa m(\rme^{d+2} S)^{1/\kappa})\eqsp.
\end{align*}
Therefore,
\begin{multline*}
M^\text{lin}(h,\phi;\nu) \geq \frac{\sfc}{4} 8^{-d} (\rme^{d+2} S)^{-6m/\kappa} m^{-5d-3} (8\rme d)^{-6m}(\nu \vee \nu^{-3})^{-2m} \exp( -12\kappa m(\rme^{d+2} S)^{1/\kappa}) \|h\|^2\\
 - \tilde\sfc \, 7^4 \cdot 36 (8\nu)^d (S\nu)^{2m} \kappa^{-4d-6} \exp((2d+9)\kappa (S\nu)^{1/\kappa}) m^{-2\kappa m+2d}\eqsp.
\end{multline*}
Assume that
\begin{multline}
\label{eq_contrainte_2_temp}
\frac{\sfc}{8}\|h\|^2 8^{-d} (\rme^{d+2} S)^{-6m/\kappa} m^{-5d-3} (8\rme d)^{-6m}(\nu \vee \nu^{-3})^{-2m} \exp\left( -12\kappa m(\rme^{d+2} S)^{1/\kappa} \right) \\
	\geq \tilde\sfc \, 7^4 \cdot 36 (8\nu)^d (S\nu)^{2m} \kappa^{-4d-6} \exp\left((2d+9)\kappa (S\nu)^{1/\kappa}\right) m^{-2\kappa m+2d}\eqsp.
\end{multline}
Then,
\begin{multline}
\label{eq_minoration_Mlin_brute}
M^\text{lin}(h,\phi;\nu) \geq \frac{\sfc}{8} 8^{-d} (\rme^{d+2} S)^{-6m/\kappa} m^{-5d-3} (8\rme d)^{-6m}(\nu\vee \nu^{-3})^{-2m}\\
\times \exp\left( -12\kappa m(\rme^{d+2} S)^{1/\kappa} \right)\|h\|^2\eqsp.
\end{multline}
Note that \eqref{eq_contrainte_2_temp} is equivalent to 
\begin{multline*}
m^{2\kappa m - 7d -3} ((8\rme d)^3 (\nu \vee \nu^{-3}) S\nu)^{-2m} (\rme^{d+2} S)^{-6m/\kappa}\exp(-12\kappa m(\rme^{d+2} S)^{1/\kappa})\\
 \geq\frac{288 \cdot 7^4 \tilde{\sfc}}{\sfc} (64 \nu)^d \kappa^{-4d-6}\exp((2d+9)\kappa (S\nu)^{1/\kappa}) \|h\|^{-2}\eqsp.
\end{multline*}
which can be written 
\begin{equation}
\label{eq:const:alltogether}
m^{\kappa m - a_2} b_2^m \geq c_2\|h\|^{-1}\eqsp,
\end{equation}
where $a_2 = (7d+3)/2$, $b_2 = ((8\rme d)^{3} (\nu \vee \nu^{-3}))^{-1} (S\nu)^{-1} (\rme^{d+2} S)^{-3/\kappa}\exp(-6\kappa (\rme^{d+2} S)^{1/\kappa})$ and $c_2^2 = 288 \cdot 7^4(\tilde{\sfc}/\sfc) (64 \nu)^d \kappa^{-4d-6} \exp((2d+9)\kappa (S\nu)^{1/\kappa})$. Note that $a_2>a_1$, $c_2>c_1$ since $\kappa \leq 1$ and $b_2<b_1$.
Thus, \eqref{eq:const:1} and \eqref{eq:const:alltogether} hold when
\begin{equation}
\label{eq:constraint:final}
\kappa m \log (\kappa m) - (\log(b_2^{-1})m + a_2 \log m + \log c_2) \geq \log(1/\|h\|)\eqsp.
\end{equation}
Equation~\eqref{eq:constraint:final} is satisfied when
\begin{equation*}
\kappa m (\log (\kappa m) - \left(\log(b_2^{-1})/\kappa + a_2 /\kappa+ \log (c_2)/\kappa\right) \geq \log(1/\|h\|)\eqsp,
\end{equation*}
which can be written
\begin{equation}
\label{eq:constraint:final2}
\kappa m \log \left(\left(\frac{b_2}{c_2\rme^{a_2}}\right)^{1/\kappa}\kappa m\right) \geq \log(1/\|h\|)\eqsp.
\end{equation}
Note that for all $A > 1$, the solution $x$ of the equation $x \log x = A$ satisfies $x \leq 3A / (2\log A)$, so that choosing
\begin{equation}
\label{eq_contrainte_2_alt}
m = \left\lfloor \frac2{\kappa} \frac{\log (1/\|h\|)}{\log \left\{\left(1 \wedge \frac{b_2}{c_2\rme^{a_2}}\right)^{1/\kappa} \log (1/\|h\|)\right\}} \right\rfloor\eqsp,
\end{equation}
ensures that \eqref{eq:constraint:final2} holds as soon as 
\begin{equation*}
\left(1 \wedge \frac{b_2}{c_2\rme^{a_2}}\right)^{1/\kappa} \log (1/\|h\|) > 1 \quad \mathrm{and} \quad \frac{\log(1/\|h\|)}{\log\log(1/\|h\|)} \geq 2\kappa\eqsp,
\end{equation*}
which is always true when $\|h\| < \rme^{-1}$ since $\kappa \leq 1$. If the condition on $\|h\|$ is strenghtened into 
\begin{equation*}
\left(1 \wedge \frac{b_2}{c_2\rme^{a_2}}\right)^{2/\kappa} \log (1/\|h\|) > 1\eqsp,
\end{equation*}
then the choice \eqref{eq_contrainte_2_alt} implies
\begin{equation*}
m \leq \frac{4}{\kappa} \frac{\log (1/\|h\|)}{\log \log (1/\|h\|)}\eqsp.
\end{equation*}
Since with $b_2$ defined above, \eqref{eq_minoration_Mlin_brute} can be written
\begin{align*}
M^\text{lin}(h,\phi;\nu)
 	&\geq \frac{\sfc}{8} \|h\|^2 8^{-d}
 		 m^{-5d-3} (b_2 S\nu)^{2m}\eqsp,
\end{align*}
and $b_2 S \nu \leq 1$, this yields
\begin{equation*}
M^\text{lin}(h,\phi;\nu) \geq \frac{\sfc}{8} \|h\|^2 8^{-d}
 		 \left( \frac{\kappa \log \log (1/\|h\|)}{4 \log (1/\|h\|)} \right)^{5d+3}
 		 \|h\|^{ \displaystyle \frac{-8 \log (b_2 S\nu)}{\kappa \log \log (1/\|h\|)}}\eqsp.
\end{equation*}

\subsection{Proof of Lemma~\ref{lem_partie_quartique}}

Let $\nu > 0$ and $d= d_1+d_2$. For all $h \in \C_m[X_1, \dots, X_d]$, there exists a unique matrix $H = (H_{i,j})_{i\in \N^{d_1},j\in \N^{d_2}}$ such that for all $(x,y) \in \neighborhood{d_1}{\nu} \times \neighborhood{d_2}{\nu}$, $h(x,y) = \sum_{i\in \N^{d_1},j\in \N^{d_2}} H_{i,j} \Pbf^\text{norm}_i(x) \Pbf^\text{norm}_j(y)$, with $\Pbf^\text{norm}_i$ defined in equation~\eqref{eq:prod:legendre}. Since $H_{i,j} = 0$ if $\|i\|_1 + \|j\|_1 > m$ (as $\deg(h) \leq m$), by Cauchy-Schwarz inequality,
\begin{align*}
\| h(\cdot, 0) \|_{\Lbf^2(\neighborhood{d_1}{\nu})}^2
	&= \sum_{i \in \N^{d_1}} \left| \sum_{j \in \N^{d_2} : \|j\|_1 \leq m} H_{i,j} \Pbf^\text{norm}_j(0) \right|^2 \\
	&\leq \sum_{i \in \N^{d_1}} \sum_{j \in \N^{d_2}} |H_{i,j}|^2 \sum_{j' \in \N^{d_2} : \|j'\|_1 \leq m} \Pbf^\text{norm}_{j'}(0)^2\eqsp,\\
	&= \sum_{i \in \N^{d_1}} \sum_{j \in \N^{d_2}} |H_{i,j}|^2 \sum_{j' \in \N^{d_2} : \|j'\|_1 \leq m} \prod_{a=1}^{d_2} \left( 4^{-j'_a}\sqrt{\frac{j'_a + 1/2}{\nu}} \binom{2j'_a}{j'_a} \right)^2.
\end{align*}
By Stirling's formula, for all $j_a \in \N$,
\begin{equation*}
4^{-j'_a}\sqrt{j'_a + 1/2} \binom{2j'_a}{j'_a} \leq \sqrt{2/\pi}\eqsp.
\end{equation*}
Then, there exists a numerical constant $c>0$ such that
\begin{multline*}
\| h(\cdot, 0) \|_{\Lbf^2(\neighborhood{d_1}{\nu})}^2 \leq \sum_{i \in \N^{d_1}} \sum_{j \in \N^{d_2}} |H_{i,j}|^2 \!\!\!\!\!\sum_{j' \in \N^{d_2} : \|j'\|_1 \leq m}\!\!\!\!\! (c/\nu)^{d_2}\\
\leq (c/\nu)^{d_2} \| H \|_F^2 |\{ 0, \dots, m \}^d| \leq (2c/\nu)^{d_2} m^{d_2} \| h \|_{\Lbf^2(\neighborhood{d_1}{\nu}\times \neighborhood{d_2}{\nu})}^2\eqsp.
\end{multline*}
Assume now that $h \in \Gcal_{\kappa,S}$ and $m \geq d/\kappa$. By Lemma~\ref{lem_troncature},
\begin{align*}
\| h(\cdot, 0) & \|_{\Lbf^2(\neighborhood{d_1}{\nu})}^2
	= \| T_m h(\cdot, 0) + (h - T_m h)(\cdot,0) \|_{\Lbf^2(\neighborhood{d_1}{\nu})}^2\eqsp, \\
	&\leq 2\| T_m h(\cdot, 0) \|_{\Lbf^2(\neighborhood{d_1}{\nu})}^2
		+ 2 2^{2d}(2\nu)^d (S\nu)^{2m}m^{-2\kappa m+\com{2d}} f_\kappa(S\nu)^2\eqsp, \\
	&\leq 2 (2c/\nu)^{d_2} m^{d_2} \| T_m h \|_{\Lbf^2(\neighborhood{d_1}{\nu}\times \neighborhood{d_2}{\nu})}^2
		+ 2 2^{2d}(2\nu)^d (S\nu)^{2m}m^{-2\kappa m+\com{2d}} f_\kappa(S\nu)^2\eqsp, \\
	&\leq 2 (2c/\nu)^{d_2} m^{d_2} \| h - (h-T_m h) \|_{\Lbf^2(\neighborhood{d_1}{\nu}\times \neighborhood{d_2}{\nu})}^2 
		+ 2 2^{2d}(2\nu)^d (S\nu)^{2m}m^{-2\kappa m+\com{2d}} f_\kappa(S\nu)^2\eqsp, \\
	&\leq \com{4 (2cm/\nu)^{d_2}} \| h \|_{\Lbf^2(\neighborhood{d_1}{\nu}\times \neighborhood{d_2}{\nu})}^2 + \com{(4 (2cm/\nu)^{d_2} + 2)2^{2d}(2\nu)^d} (S\nu)^{2m} m^{-2\kappa m+\com{2d}} f_\kappa(S\nu)^2\eqsp.
\end{align*}
Following the same steps for $\| h(0, \cdot) \|_{\Lbf^2(\neighborhood{d_1}{\nu}\times \neighborhood{d_2}{\nu})}^2$ yields
\begin{multline*}
\| h(\cdot, 0) h(0, \cdot) \|_{\Lbf^2(\neighborhood{2d}{\nu})}^2\leq 16\left\{(2cm/\nu)^{2d_1} + (2cm/\nu)^{2d_2}\right\} \| h \|_{\Lbf^2(\neighborhood{d_1}{\nu}\times \neighborhood{d_2}{\nu})}^4 \\ + \com{\left\{(4 (2cm/\nu)^{d_1} + 2)^2 + (4 (2cm/\nu)^{d_2} + 2)^2\right\}2^{4d}(2\nu)^{2d}} (S\nu)^{4m} m^{-4\kappa m+\com{4d}} f_\kappa(S\nu)^4 \eqsp,
\end{multline*}
which concludes the proof.

\subsection{Proof of Proposition~\ref{prop:compromis:h2}}

In this proof the subscript $\Lbf^2(\neighborhood{d_1}{\nu}\times \neighborhood{d_2}{\nu})$ is dropped from the notation $\|h\|$ for better clarity. By Lemma~\ref{lem_partie_quartique}, there exists a numerical constant $c_5 > 0$ such that for all $c_5' \geq c_5$, for all $\kappa > 0$, $S < \infty$, $\nu > 0$, $m \geq d/\kappa$ and $h \in \Gcal_{\kappa,S}$,
\begin{multline*}
\| h(\cdot, 0) h(0, \cdot) \|_{\Lbf^2(\neighborhood{d}{\nu})}^2
	\leq 16\left\{(2c_5' m/\nu)^{2d_1} + (2c_5' m/\nu)^{2d_2}\right\} \| h \|_{\Lbf^2(\neighborhood{d_1}{\nu}\times \neighborhood{d_2}{\nu})}^4 \\
	+ \left\{(4 (2c_5' m/\nu)^{d_1} + 2)^2 + (4 (2c_5' m/\nu)^{d_2} + 2)^2\right\}2^{4d}(2\nu)^{2d} (S\nu)^{4m} m^{-4\kappa m+4d} f_\kappa(S\nu)^4 \eqsp.
\end{multline*}
Then, by Lemma~\ref{lem_controle_constantes},
\begin{align*}
\| h(\cdot, 0) h(0, \cdot) \|^2
	&\leq 16\left\{(2c_5' m/\nu)^{2d_1} + (2c_5' m/\nu)^{2d_2}\right\} \| h \|^4 \\
	&\ + \left\{(4 (2c_5' m/\nu)^{d_1} + 2)^2 + (4 (2c_5' m/\nu)^{d_2} + 2)^2\right\}2^{4d}(2\nu)^{2d} (S\nu)^{4m} m^{-4\kappa m+\com{4d}} \\
	&\hspace{7cm}\times 6^4 (S\nu)^{4/\kappa}\exp(4\kappa (S\nu)^{1/\kappa}) \\
	&\leq 32 (1 \vee 2c_5' m/\nu)^{2(d_1 \vee d_2)} \| h \|^4 \\
	&\ + 72 \cdot 6^4 (1 \vee 2c_5' m/\nu)^{2(d_1 \vee d_2)} 64^d\nu^{2d} (S\nu)^{4m} m^{-4\kappa m+\com{4d}} (S\nu)^{4/\kappa}\exp(4\kappa (S\nu)^{1/\kappa})\eqsp.
\end{align*}
Assume that
\begin{equation}
\label{eq:const:h4}
32 \| h \|^4 \geq 72 \cdot 6^4 \cdot 64^d\nu^{2d} (S\nu)^{4m} m^{-4\kappa m+\com{4d}} (S\nu)^{4/\kappa}\exp(4\kappa (S\nu)^{1/\kappa}) \eqsp,
\end{equation}
then
\begin{equation}
\label{eq_majoration_quartique_brute}
\| h(\cdot, 0) h(0, \cdot) \|^2
	\leq 64 (1 \vee 2c_5' m/\nu)^{2(d_1 \vee d_2)} \| h \|^4 \eqsp.
\end{equation}
Assumption \eqref{eq:const:h4} can be written
\begin{equation}
\label{eq:const:alltogether:h4}
m^{\kappa m - a_3} b_3^m \geq c_3\|h\|^{-1}\eqsp,
\end{equation}
where $a_3 = d$, $b_3 = (S\nu)^{-1}$ and $c_3 = 3 \sqrt{6} \cdot 2^d (2\nu)^{d/2} (S \nu)^{1/\kappa} \exp(\kappa (S\nu)^{1/\kappa})$. Following the same steps as for the first term yields that choosing
\begin{equation}
\label{eq_contrainte_3_alt}
m = \left\lfloor \frac2{\kappa} \frac{\log (1/\|h\|)}{\log \left\{\left(1 \wedge \frac{b_3}{c_3\rme^{a_3}}\right)^{1/\kappa} \log (1/\|h\|)\right\}} \right\rfloor\eqsp,
\end{equation}
ensures that \eqref{eq:const:alltogether:h4} holds as soon as 
\begin{equation*}
\left(1 \wedge \frac{b_3}{c_3\rme^{a_3}}\right)^{1/\kappa} \log (1/\|h\|) > 1 \quad \mathrm{and} \quad \and \frac{\log(1/\|h\|)}{\log\log(1/\|h\|)} \geq 2\kappa\eqsp,
\end{equation*}
which is always true when $\|h\| < \rme^{-1}$ since $\kappa \leq 1$. If the condition on $\|h\|$ is strenghtened into 
\begin{equation}
\label{eq_condition_h_compromis_quartique}
\left(1 \wedge \frac{b_3}{c_3\rme^{a_3}}\right)^{2/\kappa} \log (1/\|h\|) > 1\eqsp,
\end{equation}
then the choice \eqref{eq_contrainte_3_alt} implies
\begin{equation*}
m \leq \frac{4}{\kappa} \frac{\log (1/\|h\|)}{\log \log (1/\|h\|)}\eqsp.
\end{equation*}
Together with \eqref{eq_majoration_quartique_brute}, this implies that if this $m$ is greater than $\nu/(2c_5')$,
\begin{equation*}
\| h(\cdot, 0) h(0, \cdot) \|^2
	\leq 64 \left(\frac{2c_5'}{\nu}\right)^{2(d_1\vee d_2)} \left(\frac{4}{\kappa} \frac{\log (1/\|h\|)}{\log \log (1/\|h\|)}\right)^{2(d_1\vee d_2)} \| h \|^4 \eqsp.
\end{equation*}
The condition $m \geq \nu/(2c_5')$ with $m$ as in \eqref{eq_contrainte_3_alt} is ensured by
\begin{align*}
\frac{3}{2\kappa} \frac{\log(1/\|h\|)}{\log \log (1/\|h\|)} \geq \frac{\nu}{2c_5'}\eqsp,
\end{align*}
which is in turn ensured by
\begin{align*}
\frac{\log(1/\|h\|)}{\log \log (1/\|h\|)} \geq 1 \vee \frac{\kappa \nu}{3c_5'}\eqsp.
\end{align*}
Now take $c_5' = c_5 \vee (\kappa \nu / 3)$.
Since for all $A \geq 1$, the solution $x$ of equation $x / \log x = A$ satisfies $x \leq 2A\log A$, this is ensured by
\begin{align*}
\log(1/\|h\|) \geq 2 \left(1 \vee \left(\frac{\kappa \nu}{3c_5}\wedge 1\right) \right) \log \left(1 \vee \left(\frac{\kappa \nu}{3c_5}\wedge 1\right)\right) = 0\eqsp,
\end{align*}
which holds as soon as $\|h\| \leq 1$, and this condition is already implied by~\eqref{eq_condition_h_compromis_quartique}.

\section{Proofs of Section~\ref{sec:lower}}
\label{sec:proof:lower}

\subsection{Proof of Lemma~\ref{lem_fns_dans_Upsilon}}
Let $H$ be defined by $h_\kappa(x) = H(x/x_0)/x_0$. Then, $\zeta \leq c([x \mapsto h_\kappa(x) (1+(x/x_0)^2)^\tau] * u_b)$ is equivalent to $x_0 \zeta(x_0 x) \leq c([z \mapsto H(z) (1+z^2)^\tau] * u_{bx_0})(x)$. In this proof, we show that there exists $A$ and $B$ such that for all $\lambda \in \R$ and $b \geq 1$,
\begin{equation*}
\int \rme^{\lambda x} ([z \mapsto H(z) (1+z^2)^\tau] * u_b)(x) \rmd x
	\leq A \rme^{B |\lambda|^{1/\kappa}}\eqsp,
\end{equation*}
in other words $[z \mapsto H(z) (1+z^2)^\tau] * u_b \in \Mcal^1_{1/\kappa}$, which entails $(x \mapsto x_0 \zeta(x_0 x)) \in \Mcal^1_{1/\kappa}$ and thus $\Fcal[x \mapsto x_0 \zeta(x_0 x)] = \Fcal[\zeta](\cdot/x_0) \in \Upsilon_{\kappa,T'}$ for some $T'$ by Lemma~\ref{lem_lien_Mrho_Upsilon}. This ensures $\Fcal[\zeta] \in \Upsilon_{\kappa,T' x_0}$ for all $b \geq 1/x_0$, which yields the result by choosing $x_0$ small enough. Let $c_H = c_h / x_0$ be the normalizing constant of $H$, then for all $b \geq 1$ and $\lambda \in \R$,
\begin{align*}
\int \rme^{\lambda x} ([z \mapsto H(z) (1+z^2)^\tau] & * u_b)(x) \rmd x \\
	&\leq \int \rme^{\lambda x} \sup_{y \in [x-1/b,x+1/b]} H(y) (1+y^2)^\tau \rmd x \\
	&\leq 2^\tau \frac{2 c_H}{b} + 2 c_H \int_{x \geq 0} \rme^{|\lambda| (x+1/b)} (1+x^2)^\tau \rme^{-([1+x^2]/2)^{1/(2(1-\kappa))}} \rmd x \\
	&\leq 2^{1+\tau} c_H + 2 c_H \rme^{|\lambda|/b} \int_{x \geq 0} (1+x^2)^\tau \rme^{|\lambda| x - (x/\sqrt{2})^{1/(1-\kappa)}} \rmd x \\
	&\leq 2^{1+\tau} c_H + 2 c_H \rme^{|\lambda|} X_\lambda (1+X_\lambda^2)^\tau \rme^{|\lambda| X_\lambda} \\
		&\qquad + 2 c_H \rme^{|\lambda|} \int_{x \geq X_\lambda} \rme^{|\lambda| x + \tau \log(1+x^2) - (x/\sqrt{2})^{1/(1-\kappa)}} \rmd x
\end{align*}
for all $X_\lambda > 0$. Let $X_\lambda$ be such that $|\lambda| x + \tau \log(1+x^2) - (x/\sqrt{2})^{1/(1-\kappa)} \leq -(1/2) (x/\sqrt{2})^{1/(1-\kappa)}$ for all $x \geq X_\lambda$. Taking $X_\lambda = c_X |\lambda|^{-1 + 1/\kappa}$ works for $\lambda$ large enough for an appropriate constant $c_X$. Then for $\lambda$ large enough,
\begin{align*}
\int e^{\lambda x} &([z \mapsto H(z) (1+z^2)^\tau] * u_b)(x) \rmd x \leq 2^{1+\tau} c_H 
		+ 2 c_H \rme^{|\lambda|} c_X |\lambda|^{-1 + 1/\kappa} (1+c_X^2 |\lambda|^{-2 + 2/\kappa})^\tau e^{c_X |\lambda|^{1/\kappa}} \\
		&\quad + 2 c_H \rme^{|\lambda|} \int_{x \geq 0} \rme^{- 2^{-1} (\frac{X_\lambda + x}{\sqrt{2}})^{1/(1-\kappa)}} \rmd x\eqsp, \\
	&\leq c\cdot \rme^{\text{cst}' |\lambda|^{1/\kappa}}
		 + 2 c_H \rme^{|\lambda|} \exp(- 2^{-1 - 1/(2(1-\kappa))} |\lambda|^{1/\kappa}) \int_{x \geq 0} \rme^{- 2^{-1} (x/\sqrt{2})^{1/(1-\kappa)}} \rmd x\eqsp, \\
	&\leq A \cdot \rme^{B |\lambda|^{1/\kappa}}\eqsp,
\end{align*}
by convexity of $x \mapsto x^{1/(1-\kappa)}$ for some constants $A$ and $B$ depending only on $\kappa$. Small values of $\lambda$ are dealt with by changing $A$ if necessary.

\subsection{Proof of Corollary~\ref{cor_conj_polynomes}}

The first inequality follows from the bound on $\|P_K h_\kappa / F_\text{env}\|_\infty$: there exists a constant $c$ such that
\begin{align*}
\| P_K h_\kappa^2 \|_{\Lbf^2(\R)}^2 \leq c K^{\kappa-1} \|F_\text{env} h_\kappa\|_{\Lbf^2(\R)}^2 \eqsp,
\end{align*}
and the polynomial growth assumption on $F_\text{env}$ ensures that $\|F_\text{env} h_\kappa\|_{\Lbf^2(\R)}^2 < \infty$.

The second inequality is a consequence of Cauchy-Schwarz' inequality: for any function $\varphi$ (here $P_K h_\kappa^2$),
\begin{align*}
\| \varphi * u_b \|_{\Lbf^2(\R)}^2
	&= \int \left( \int \varphi(y) u_b(x-y) \rmd y \right)^2 \rmd x \\
	&\leq \int \left( \int \varphi(y)^2 u_b(x-y) \rmd y \right) \left( \int u_b(x-y) \rmd y \right) \rmd x \\
	&= \|\varphi\|_{\Lbf^2(\R)}^2 \eqsp.
\end{align*}
For the third inequality, let $c_0$, $c_1$, $c_2$ be the constants of Conjecture~\ref{conj_polynomes}. Write $(I_i)_{i} = ([s_i,t_i])_i$ the intervals of Conjecture~\ref{conj_polynomes}. Assume $b \geq 2 K^\kappa / c_1$, so that the support of $u_b$ has length smaller than $c_1 K^{-\kappa}$. Then for all $i$ and for all $x \in [s_i + b^{-1}, t_i - b^{-1}]$ (which are non-empty intervals by the assumption on $b$),
\begin{align*}
((P_{K} h_\kappa^2) * u_b)(x)
	&= \int_{y \in [-b^{-1}, b^{-1}]} (P_{K} h_\kappa^2)(x-y) u_b(y) \rmd y \\
	&\geq c_2 K^{(\kappa-1)/2} \left(\inf_{[-1,1]} h_\kappa \right) \int u_b(y) \rmd y \\
	&= c_2 K^{(\kappa-1)/2} \left(\inf_{[-1,1]} h_\kappa \right)
\end{align*}
so that
\begin{align*}
\| (P_{K} h_\kappa^2) * u_b \|^2
	&\geq \sum_{i} \int_{[s_i + b^{-1}, t_i - b^{-1}]} ((P_K h_\kappa^2) * u_b)^2(x) \rmd x \\
	&\geq \sum_{i} (t_i - s_i - 2b^{-1}) c_2^2 K^{\kappa-1} \left(\inf_{[-1,1]} h_\kappa \right)^2 \\
	&\geq c_0 K^{\kappa} (c_1 K^{-\kappa} - 2b^{-1}) c_2^2 K^{\kappa-1} \left(\inf_{[-1,1]} h_\kappa \right)^2 \eqsp.
\end{align*}
Taking $b \geq 4 K^{\kappa} / c_1$ gives the desired inequality.

\subsection{Proof of Lemma~\ref{lem:upsilon:lower}}
By definition, $f_0$ is the density of $\Xbf$ when for all $1 \leq j \leq d$, $s_j = \zeta_0 = h_{\kappa} * u_b$, and $f_n$ is the density of $\Xbf$ when $S_{1}$ has density $s_1=\zeta_n$ and $S_{2},\ldots,S_{d}$ have density $s_j=\zeta_0$. The derivative of $\Fcal[f_0]$ is
\begin{multline*}
\partial^i \Fcal[f_0]
	= \underset{(k_1,k_{d_1+1}, \dots, k_{d}) \in \N^{d_2+1} : \|k\|_1 = i_1}{\sum_{(j_1, \dots, j_{d_1+1}) \in \N^{d_1+1} : \|j\|_1 = i_{d_1+1}}} \!\!\!\!\!\!\!\!\! \!\!\!\!\!\! \!\!\!\!\!\! 
		a^{i_{d_1+1} - j_{d_1+1}} a^{i_1 - k_1}
		\Fcal[s_1]^{(k_1 + j_1)}
		\Fcal[s_{d_1+1}]^{(j_{d_1+1} + k_{d_1+1})}
		\\ \times \prod_{u=2}^{d_1} \Fcal[s_u]^{(i_u + j_u)}
		\prod_{u=d_1+2}^{d} \!\! \Fcal[s_u]^{(i_u + k_u)}\eqsp,
\end{multline*}
where the vector $j$ corresponds to how $\partial_{d_1+1}^{i_{d_1+1}}$ is split among the $\Fcal[s_u]$, $1\leq u\leq d_{1}+1$, and $k$ corresponds to how $\partial_{1}^{i_1}$ is split among the $\Fcal[s_u]$, $u\in \{1,d_1+1,\ldots,d\}$, so that
\begin{align*}
|\partial^i \Fcal[f_0]|
	&\leq T^{\|i\|_1} \!\!\!\!\!\!\!\!\!\!\!\!\!\!\!\!\!\!\!\!\!\!\!\! \underset{(k_1,k_{d_1+1}, \dots, k_{d}) \in \N^{d_2+1} : \|k\|_1 = i_1}{\sum_{(j_1, \dots, j_{d_1+1}) \in \N^{d_1+1} : \|j\|_1 = i_{d_1+1}}} \!\!\!\!\!\!\!\!\! \!\!\!\!\!\!
		a^{i_{d_1+1} - j_{d_1+1}} a^{i_1 - k_1}
		\frac{(k_1 + j_1)!}{\|k_1 + j_1\|_1^{\kappa\|k_1 + j_1\|_1}}
		\frac{(j_{d_1+1} + k_{d_1+1})!}{\|j_{d_1+1} + k_{d_1+1}\|_1^{\kappa\|j_{d_1+1} + k_{d_1+1}\|_1}} \\
	&\qquad\qquad\qquad\qquad \times \prod_{u=2}^{d_1} \frac{(i_u + j_u)!}{\|i_u + j_u\|_1^{\kappa\|i_u + j_u\|_1}}
		\prod_{u=d_1+2}^{d} \frac{(i_u + k_u)!}{\|i_u + k_u\|_1^{\kappa\|i_u + k_u\|_1}}\eqsp.
\end{align*}
Using $(k/\rme)^k \leq k! \leq c (k/\rme)^k \sqrt{k}$ for some numerical constant $c$ (for instance 5) and $(\prod_{a=1}^d i_a^{i_a})^{-1} \leq (\|i\|_1/d)^{-\|i\|_1}$ by convexity of $x \mapsto x\log x$,
\begin{align*}
\frac{ |\partial^i \Fcal[f_0]| }{ \prod_{a=1}^d i_a ! }
	&\leq \left(\frac{T\rme d}{\|i\|_1}\right)^{\|i\|_1} c^d \rme^{-\|i\|_1} \|i\|_1^{d/2} \!\!\!\!\!\! \underset{(k_1,k_{d_1+1}, \dots, k_{d}) \in \N^{d_2+1} : \|k\|_1 = i_1}{\sum_{(j_1, \dots, j_{d_1+1}) \in \N^{d_1+1} : \|j\|_1 = i_{d_1+1}}} \!\!\!\!\!\!\!\!\! \!\!\!\!\!\!
		a^{i_{d_1+1} - j_{d_1+1}} a^{i_1 - k_1}
		\|k_1 + j_1\|_1^{(1-\kappa)\|k_1 + j_1\|_1} \\
	&\quad \times 
		\|j_{d_1+1} + k_{d_1+1}\|_1^{(1-\kappa)\|j_{d_1+1} + k_{d_1+1}\|_1} \\
	&\quad \times \prod_{u=2}^{d_1} \|i_u + j_u\|_1^{(1-\kappa)\|i_u + j_u\|_1}
		\prod_{u=d_1+2}^{d} \|i_u + k_u\|_1^{(1-\kappa)\|i_u + k_u\|_1}\eqsp, \\
	&\leq \left(\frac{Td}{\|i\|_1}\right)^{\|i\|_1} c^d \|i\|_1^{d/2} \!\!\!\!\!\! \underset{(k_1,k_{d_1+1}, \dots, k_{d}) \in \N^{d_2+1} : \|k\|_1 = i_1}{\sum_{(j_1, \dots, j_{d_1+1}) \in \N^{d_1+1} : \|j\|_1 = i_{d_1+1}}} \!\!\!\!\!\!\!\!\! \!\!\!\!\!\!
		a^{i_{d_1+1} - j_{d_1+1}} a^{i_1 - k_1} \frac{\|i\|_1^{\|i\|_1}}{(\|i\|_1 / d)^{\kappa \|i\|_1}}\eqsp, \\
	&\leq \frac{(Td^{1+\kappa})^{\|i\|_1}}{\|i\|_1^{\kappa \|i\|_1}} \frac{c^d}{(1-a)^2} \|i\|_1^{d/2}\eqsp, \\
	&\leq \frac{(c'T)^{\|i\|_1}}{\|i\|_1^{\kappa \|i\|_1}}\eqsp,
\end{align*}
for some $c'$ for all $i \neq 0$, which concludes the proof.

\subsection{Proof of Lemma~\ref{lem:alpha}}
Following \cite{meister2007deconvolving}, since without loss of generality $b_n \geq 1$ and by integration by part the quantity $c_{u,\beta} = \sup_{t \in \R} | \Fcal[u](t) | ^{2} (1+t^2)^{\beta}$ is finite,
\begin{align*}
\int | \Fcal[\alpha_n (P_{K_n} h_{\kappa}^2) * u_{b_n}](t) |^2 (1+t^2)^{\beta} \rmd t
	&= \alpha_n^2 \int | \Fcal[P_{K_n} h_{\kappa}^2](t) |^2
			| \Fcal[u_{b_{n}}](t) |^2 (1+t^2)^{\beta} \rmd t \eqsp, \\
	&= \alpha_n^2 \int | \Fcal[P_{K_n} h_{\kappa}^2](t) |^2 \left| \Fcal[u]\left(\frac{t}{b_n}\right) \right|^2 (1+t^2)^{\beta} \rmd t \eqsp, \\
	&\leq c_{u,\beta} \alpha_n^2 b_n^{2\beta} \int | \Fcal[P_{K_n} h_{\kappa}^2](t) |^2 \rmd t \eqsp, \\
	&\leq c_{u,\beta} \alpha_n^2 b_n^{2\beta} \|P_{K_n} h_\kappa^2\|_{\Lbf^2(\R)}^2
\end{align*}
by Cauchy-Schwarz's inequality and using $\Fcal[u_b](t) = \Fcal[u](t/b)$.
Thus, H\ref{assum:phistar} holds for $\Fcal[\zeta_n]$ if
\begin{equation*}
\alpha_n^2 \|P_{K_n} h_\kappa^2\|_{\Lbf^2(\R)}^2 = O(b_n^{-2\beta})\eqsp.
\end{equation*} 

\subsection{Proof of Lemma~\ref{lem:choicealphabetaK}}
For any probability density $m_0$ on $\R$, by the Cauchy-Schwarz inequality,
\begin{equation*}
\int_{\R^{d}} \left| (f_{0}\ast Q)(x)-(f_{n} \ast Q)(x) \right| \rmd x \leq \left( \int_{\R^{d}} |((f_{0}-f_n)\ast Q)(x)|^2 \prod_{i=1}^d m_0^{-1}(x_i) \rmd x \right)^{1/2} \eqsp.
\end{equation*}
Choosing $m_0 : x \mapsto (\pi (1+x^2))^{-1}$, yields
\begin{equation*}
\int_{\R^{d}} \left| (f_{0}\ast Q)(x)-(f_{n}\ast Q)(x) \right| \rmd x \leq \pi^{d/2} \left(\int_{\R^{d}} \left| ((f_{0}-f_n)\ast Q)(x)\right|^2 \prod_{i=1}^d (1+x_i^2) \rmd x\right)^{1/2}\eqsp.
\end{equation*}
Note that for all $x\in\R^d$,
\begin{align*}
\Fcal[f_0](x)
	&= \frac{1}{\mathrm{Det}(A)}\int_{\R^d}\prod_{j=1}^{d} \zeta_0\left( (A^{-1}t)_{j}\right) \rme^{it^\top x}\rmd t
		= \int_{\R^d}\prod_{j=1}^{d} \zeta_0\left(t_{j}\right) \rme^{it^\top A^\top x}\rmd t
		= \prod_{j=1}^{d}\Fcal[\zeta_0]((A^\top x)_j)\eqsp, \\
\Fcal[f_n](x)
	&= \Fcal[\zeta_n]((A^\top x)_1)\prod_{j=2}^{d}\Fcal[\zeta_0]((A^\top x)_j)\eqsp.
\end{align*}
By Parseval's identity, for all $\eta \in \N^d$,
\begin{equation*}
\int_{\R^{d}} \left| ((f_{0}-f_n)\ast Q)(x)\right|^2 \prod_{j=1}^d x_j^{2\eta_j} \rmd x
	= \int_{\R^{d}} \left| \left( \prod_{j=1}^d \partial^{\eta_j}_{t_j} \right) ((\Fcal[f_0]-\Fcal[f_n]) \Fcal[Q])(t)\right|^2\rmd t \eqsp.
\end{equation*}
Let $A_c = \{A^\top x : x\in[-c,c]^d\}\subset [-(1+a)c,(1+a)c]^d$. Since $\Fcal[g]$ and $\Fcal[g]'$ are supported on $[-c,c]$, using the change of variables $v = A^\top t$, for all $\eta \in \{0,1\}^d$,
\begin{align*}
\int_{\R^{d}} & \left| ((f_{0}-f_n)\ast Q)(x)\right|^2 \prod_{j=1}^d x_j^{2\eta_j} \rmd x \\
	&\leq c_d \sum_{0 \leq \eta' \leq \eta}
			\int_{[-c,c]^d} \left| \left(\prod_{j=1}^d \partial^{\eta'_j}_{t_j}\right) (\Fcal[f_0]-\Fcal[f_n])(t)\right|^2 \rmd t\eqsp, \\
	&\leq c_d \sum_{0 \leq \eta' \leq \eta}
			\int_{[-c,c]^d} \left| \left(\prod_{j=1}^d \partial^{\eta'_j}_{t_j}\right) \left( t \mapsto (\Fcal[\zeta_0]-\Fcal[\zeta_n])((A^\top t)_1) \prod_{j=2}^d \Fcal[\zeta_0]((A^\top t)_j) \right)(t) \right|^2 \rmd t\eqsp, \\
	&\leq c_d' \sum_{0 \leq \eta' \leq \eta}
			\int_{A_c} |(\Fcal[\zeta_0]-\Fcal[\zeta_n])^{(\eta'_1)}(v_1)|^2 \prod_{j=2}^d |\Fcal[\zeta_0]^{(\eta'_j)}(v_j)|^2 \rmd v\eqsp,
\end{align*}
for some constants $c_d$ and $c_d'$, so that for some constant $c_d''$,
\begin{multline*}
\| (f_0*Q) - (f_n*Q) \|_{\Lbf^1(\R^d)} \\
	\leq c_d'' \left(
			\int_{-(1+a)c}^{(1+a)c} |\Fcal[\zeta_0]-\Fcal[\zeta_n]|(t)^2 \rmd t
			+ \int_{-(1+a)c}^{(1+a)c} |(\Fcal[\zeta_0]-\Fcal[\zeta_n])'|(t)^2 \rmd t
		\right)^{1/2} \eqsp.
\end{multline*}
Using that for all $t \in \R$, $\Fcal[\zeta_0](t)-\Fcal[\zeta_n](t) = \alpha_n \Fcal[P_{K_n} h_\kappa^2](t) \Fcal[u_{b_n}](t)$,
\begin{align*}
\int_{\R^{d}} | (f_{0}\ast Q)(x) - (f_{n}\ast Q)(x) | \rmd x
	&\leq c_d'' \alpha_n \Bigg(
		\int_{-c}^c |\Fcal[P_{K_n} h_\kappa^2](t)|^2 \left| \Fcal[u]\left(\frac{t}{b_n}\right) \right|^2 \rmd t \\
		&\qquad\qquad\qquad+ b_n^{-2}\int_{-c}^c |\Fcal[P_{K_n} h_\kappa^2](t)|^2 \left| \Fcal[u]'\left(\frac{t}{b_n}\right) \right|^2 \rmd t \\
		&\qquad\qquad\qquad+ \int_{-c}^c |\Fcal[P_{K_n} h_\kappa^2]'(t)|^2 \left| \Fcal[u]\left(\frac{t}{b_n}\right) \right|^2 \rmd t
		\Bigg)^{1/2}\eqsp, \\
	&\leq c_d''' \alpha_n \left(
		\int_{-c}^c |\Fcal[P_{K_n} h_\kappa^2](t)|^2 \rmd t
		+ \int_{-c}^c |\Fcal[P_{K_n} h_\kappa^2]'(t)|^2 \rmd t
		\right)^{1/2}
\end{align*}
for some constant $c_d'''$. Then,
\begin{equation*}
\Fcal[P_{K_n} h_\kappa^2](t)= \int_\R P_{K_n}(x) h_\kappa^2(x) \sum_{j \geq 0} \frac{(ixt)^j}{j!} \rmd x = \sum_{j \geq K_n} \frac{(it)^j}{j!} \int_\R P_{K_n}(x) h_\kappa^2(x) x^j \rmd x \eqsp,
\end{equation*}
since by definition $P_{K_n} h_\kappa^2$ is orthogonal to $x \mapsto x^j$ in $\Lbf^2(\R^d)$ when $j \in \N$ and $j < K_n$. By Conjecture~\ref{conj_polynomes}, there exists a nonnegative envelope function $F_\text{env}$, a constant $c$ and a parameter $\alpha_\kappa \geq 0$ such that $|F_\text{env}(x)| \leq c (1 + |x|^{\alpha_\kappa})$ and such that the family $(P_K)_{K \geq 1}$ satisfies $\sup_{K \geq 1} K^{(1-\kappa)/2} \| P_K h_\kappa / F_\text{env} \|_\infty < \infty$. Then, there exists a constant $c$ which depends on $\kappa$,
\begin{multline*}
| \Fcal[P_{K_n} h_\kappa^2](t) |
	\leq \sup_K \| P_K h_\kappa/ F_\text{env}\|_\infty \sum_{j \geq K_n} \frac{t^j}{j!} \int_\R h_\kappa(x) |F_\text{env}(x)||x|^j \rmd x \\
	\leq c \sup_K \| P_K h_\kappa/ F_\text{env}\|_\infty \sum_{j \geq K_n} \frac{t^j}{j!} \int_{\R_+} (x^j + x^{j+\alpha_\kappa}) e^{-x^{1/(1-\kappa)}} \rmd x
\end{multline*}
and
\begin{equation*}
| \Fcal[P_{K_n} h_\kappa^2]'(t) |
	\leq c \sup_K \| P_K h_\kappa/ F_\text{env}\|_\infty \sum_{j \geq K_n-1} \frac{t^j}{j!} \int_{\R_+} (x^{j+1} + x^{j+\alpha_\kappa+1}) e^{-x^{1/(1-\kappa)}} \rmd x \eqsp.
\end{equation*}
For all $j \in \R_+$, write $M_j = \int_{\R_+} x^j e^{-x^{1/(1-\kappa)}}\rmd x$. By integration by part with $u'(x)=\frac{1}{1-\kappa} x^{\frac{1}{1-\kappa} - 1} e^{-x^{1/(1-\kappa)}}$ and $v(x)=(1-\kappa) x^{j+1-\frac{1}{1-\kappa}}$ and thus $u(x)=-e^{x^{1/(1-\kappa)}}$ and $v'(x)=(1-\kappa)(j+1-1/(1-\kappa)) x^{j-\frac{1}{1-\kappa}}$, for all $j > \frac{1}{1-\kappa} - 1$,
\begin{equation*}
M_j = (1-\kappa)\left(j+1-\frac{1}{1-\kappa}\right) M_{j-\frac{1}{1-\kappa}} \eqsp.
\end{equation*}
In particular, for all $j \geq 1$,
\begin{equation*}
M_j
	\leq (1-\kappa)^{(1-\kappa)j - 1} j^{(1-\kappa)j}
		\sup_{j' \in [0, 1/(1-\kappa))} M_{j'} \eqsp.
\end{equation*}
Note that
\begin{align*}
\frac{(j+\alpha_\kappa)^{j+\alpha_\kappa}}{j^j}
	&\sim \frac{(j+\alpha_\kappa)!}{j!} \frac{e^{j+\alpha_\kappa} \sqrt{j+\alpha_\kappa}}{e^j \sqrt{j}} \\
	&= O((j+\alpha_\kappa)^{\alpha_\kappa} e^{\alpha_\kappa} \sqrt{1 + \alpha_\kappa/j}) \\
	&= O((j+\alpha_\kappa)^{\alpha_\kappa}).
\end{align*}
Therefore, there exists a constant $c$ such that for all $t \in \R$,
\begin{equation*}
| \Fcal[P_{K_n} h_\kappa^2](t) |
	\leq c\sum_{j \geq K_n} \frac{t^j}{j!} (M_j + M_{j+\alpha_\kappa})\eqsp,
\end{equation*}
and a similar upper bound for $|\Fcal[P_{K_n} h_\kappa^2]'(t) |$.
Note that for all $\alpha > 0$, there exists a constant $c$ such that for all $t \in \R$, when $K_n \geq \alpha$,
\begin{equation*}
\sum_{j \geq K_n} \frac{t^j}{j!} M_{j+\alpha}
	\leq c \sum_{j \geq K_n} \frac{(t \rme (1-\kappa)^{(1-\kappa)})^j}{j^{j}}j^{-1/2} (j+\alpha)^{(1-\kappa)\alpha} j^{(1-\kappa)j}
	\leq c \sum_{j \geq K_n} \frac{(t\rme (1-\kappa)^{(1-\kappa)})^j}{j^{\kappa j-(1-\kappa)\alpha}}
\end{equation*}
Therefore, there exists constants $c$ and $C$ such that
\begin{equation*}
\int_{\R^{d}} | (f_{0}\ast Q)(x) - (f_{n}\ast Q)(x) | \rmd x\leq c \alpha_n \left( \frac{C}{K_n} \right)^{\kappa K_n}\eqsp.
\end{equation*}
Thus, equation~\eqref{eq:f0minusfn} holds if $K_n$ is chosen, for some large enough constant $C'$, as
\begin{equation*}
K_n = \frac{C'}{\kappa} \left(\frac{\log n}{\log \log n}\right)\eqsp.
\end{equation*}

\section{Numerical illustration of Conjecture~\ref{conj_polynomes}}
\renewcommand\thefigure{\thesection.\arabic{figure}} 

\label{sec_numerical}

In this section, we propose some numerical illustrations to support Conjecture~\ref{conj_polynomes}. First note that the case $\kappa = 1$ is true as it boils down to the results established in \cite{meister2007deconvolving} on Legendre Polynomials. The case $\kappa = 1/2$ is also strongly supported by properties of Hermite functions, see~\cite[Section A.11]{boyd2018dynamics}.

The orthonormal polynomials used in Conjecture~\ref{conj_polynomes} were approximately computed using the Python package OrthoPoly\footnote{https://github.com/j-jith/orthopoly} which allows to generate orthogonal polynomials with respect to any probability density functions. The Python code used in this numerical section is available online\footnote{https://sylvainlc.github.io/project/algorithms/}.
Figure~\ref{fig:norminf} displays the functions $x\mapsto K^{(1-\kappa)/2} (P_K h_\kappa) (K^{1-\kappa} x)$ for degrees $1\leq K\leq 16$ and for $\kappa\in\{0.55,0.6,0.7,0.8,0.9,0.95\}$. We chose to limit our simulations to $K\leq 16$ as for degrees larger than 18 the simulations faced some numerical instability to compute $P_K h_\kappa$. This figure illustrates Equation~\eqref{eq_shape}, i.e. the fact that there exists a function $F_\text{shape}$ such that
\begin{equation*}
\sup_{K \geq 1} \left\| x \mapsto \frac{(P_K h_\kappa) (x) }{K^{(\kappa-1)/2} F_\text{shape}(K^{\kappa-1} x) } \right\|_\infty < \infty \eqsp.
\end{equation*}
Then, Figure~\ref{fig:intervals} illustrates the second part of the conjecture by displaying $x\mapsto K^{(1-\kappa)/2} P_K h_\kappa (K^{-\kappa} x)$ for the same values of $\kappa$ and $K$ as in Figure~\ref{fig:norminf}.

\begin{figure}
\centering
\caption{Graphical representation of $x\mapsto K^{(1-\kappa)/2} P_K h_\kappa (K^{1-\kappa} x)$ for several values of $\kappa$ and $K$.}
\includegraphics[width = 0.45\textwidth]{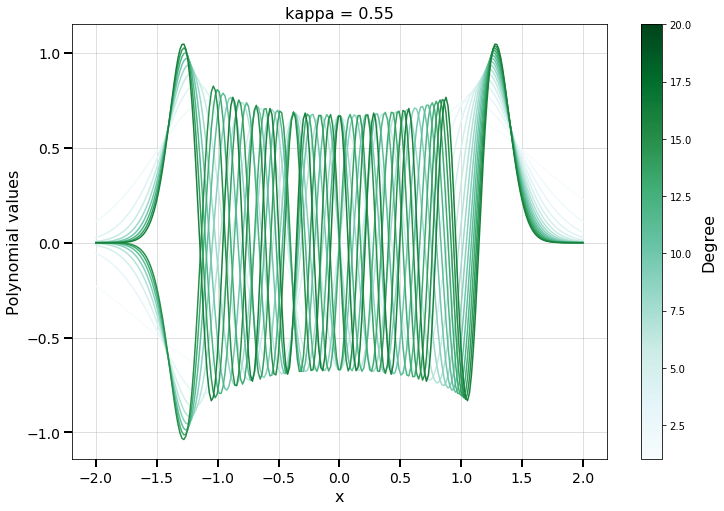}
\includegraphics[width = 0.45\textwidth]{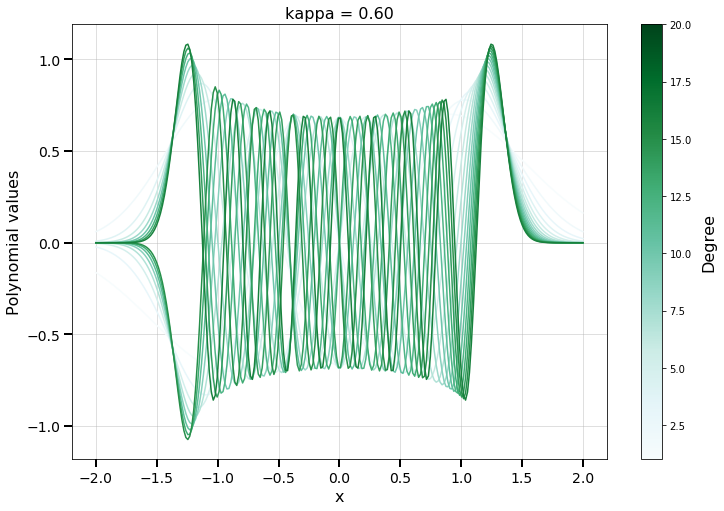}
\includegraphics[width = 0.45\textwidth]{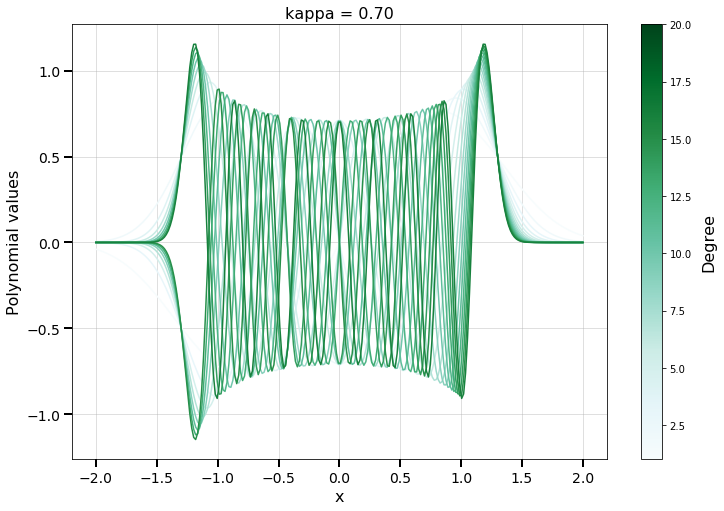}
\includegraphics[width = 0.45\textwidth]{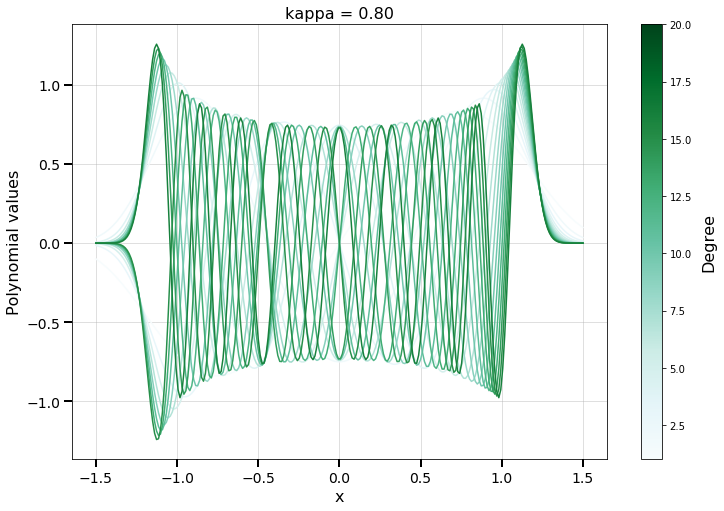}
\includegraphics[width = 0.45\textwidth]{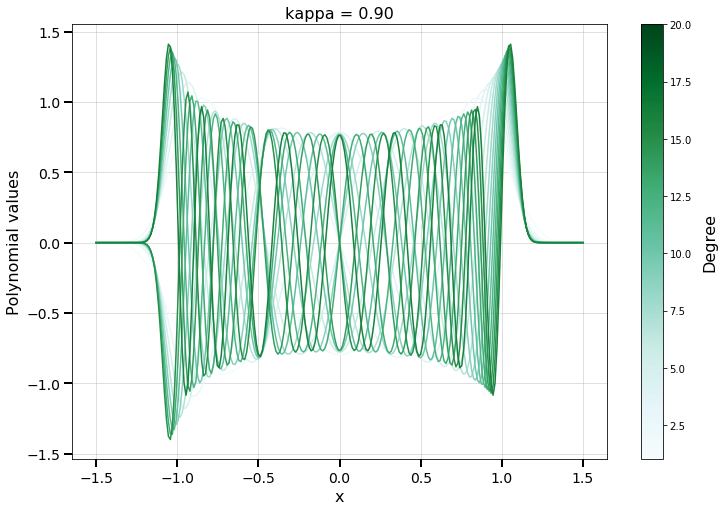}
\includegraphics[width = 0.45\textwidth]{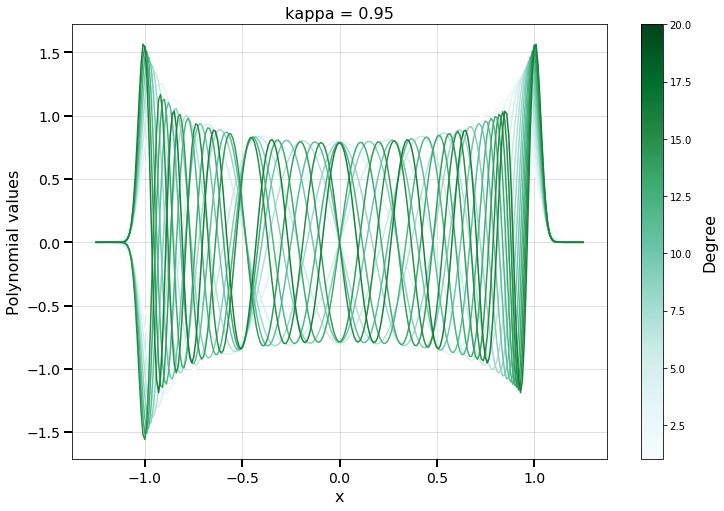}
\label{fig:norminf}
\end{figure}
\begin{figure}
\centering
\caption{Graphical representation of $x\mapsto K^{(1-\kappa)/2} P_K h_\kappa (K^{-\kappa} x)$ for several values of $\kappa$ and $K$.}
\includegraphics[width = 0.45\textwidth]{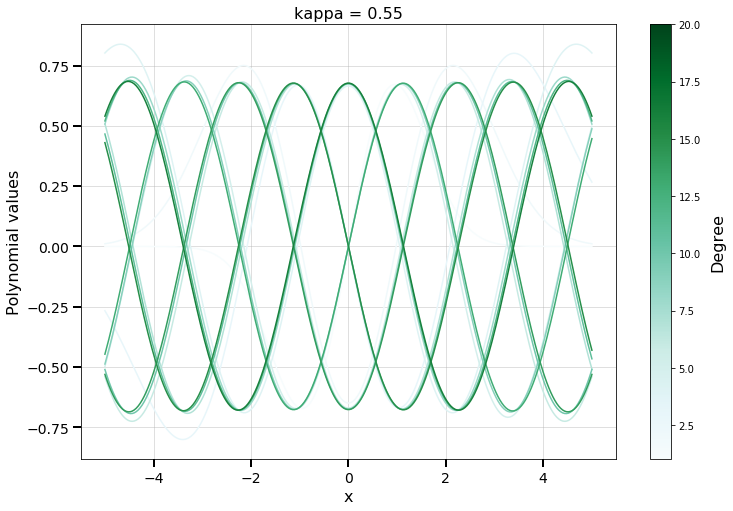}
\includegraphics[width = 0.45\textwidth]{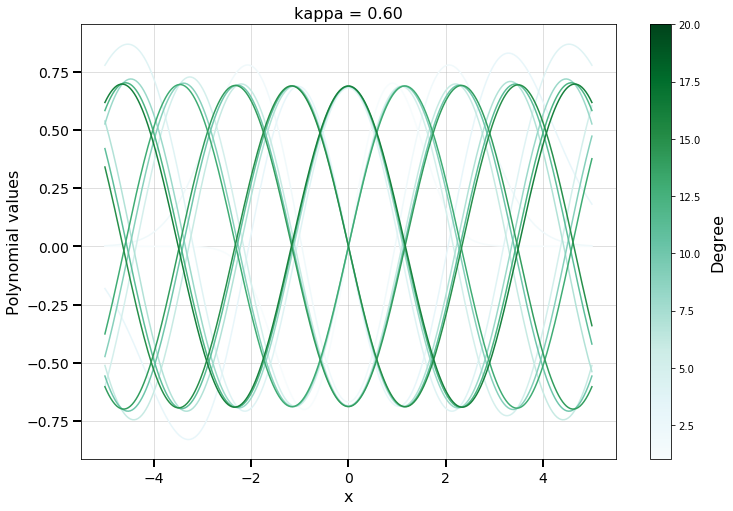}
\includegraphics[width = 0.45\textwidth]{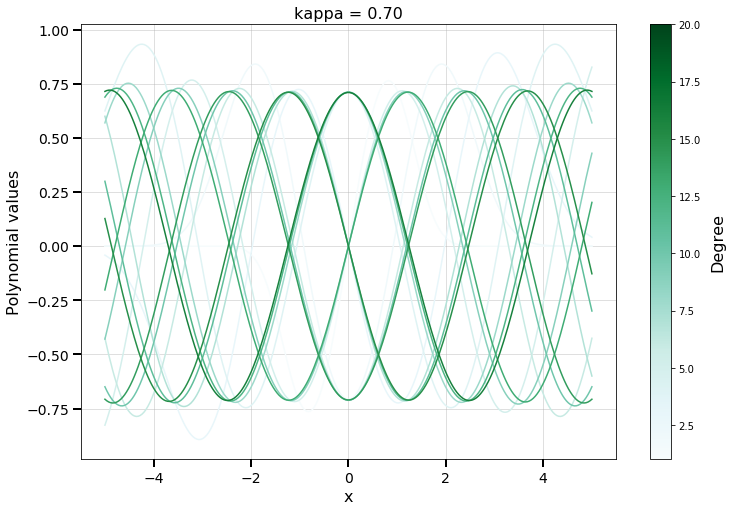}
\includegraphics[width = 0.45\textwidth]{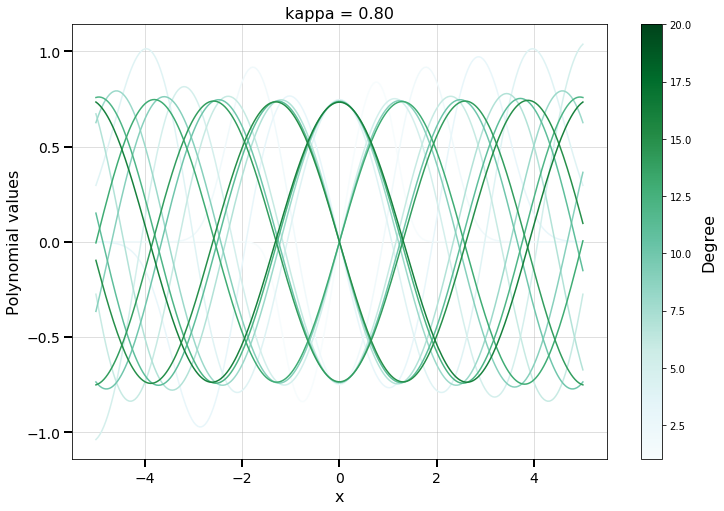}
\includegraphics[width = 0.45\textwidth]{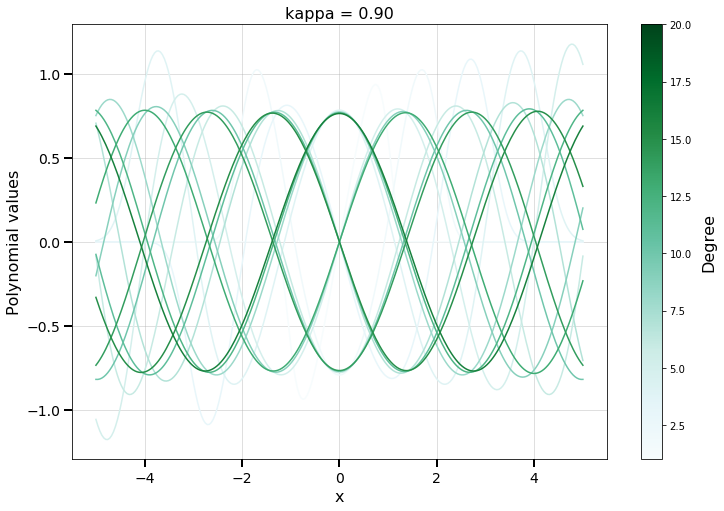}
\includegraphics[width = 0.45\textwidth]{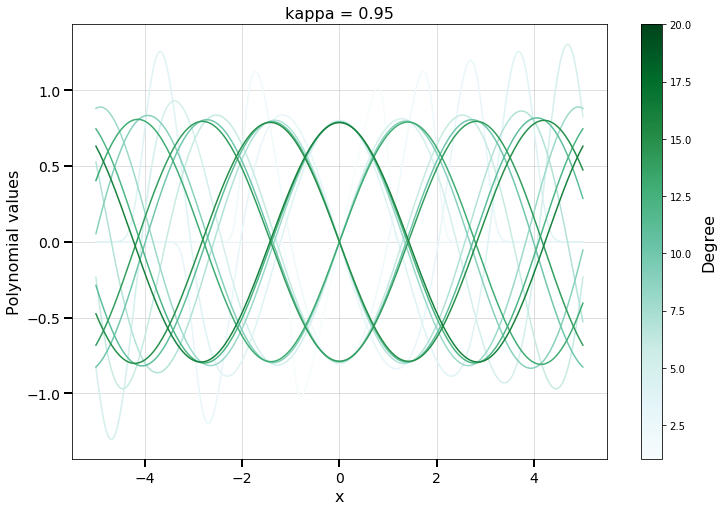}
\label{fig:intervals}
\end{figure}

\clearpage

\bibliographystyle{apalike}
\bibliography{biblio}

\end{document}

%% file: lower.tex
In this section, we provide a lower bound showing that the rate of convergence $\left(\log n/\log \log n\right)^{-2\kappa\beta}$ obtained in Theorem \ref{th:fhat} and in Theorem \ref{theo:adaptrate} is minimax optimal. 
The lower bound in \cite{meister2007deconvolving} holds for $\kappa = 1$, so in the following we only consider $\kappa \in (0,1)$. 
 In this section (and only this section), we use the notation $\Fcal[h]$ (resp. $\Fcal[Q]$) for the Fourier transform of the probability density function $h$ (resp. the probability measure $Q$). Our lower bound is stated in Theorem~\ref{th:lower}.

The proof of Theorem~\ref{th:lower} is based on Le Cam's method, also known as the {\em two-points} method, see \cite{le2012asymptotic}, one of the most widespread technique to derive lower bounds. 
The minimax risk based on $n$ observations is lower bounded by considering observations from model (\ref{eq:ica:noisy}) i.e. assuming that ${\bf Y} = A\Sbf + \bfeps$ where $\Sbf\in \R^d$ with $d = d_1 + d_2$ in which the coordinates $S_{j}$, $j=1,\ldots,d$, of $\Sbf$ are independent. Let $f_{0}$ and $f_{n}$ be the probability densities of $A\Sbf$ associated with different choices of densities for the distributions of $S_{j}$, $j=1,\ldots,d$, and $Q$ be the distribution of the noise $\varepsilon$. Then, following Le Cam's method, the minimax risk is lower bounded by
\begin{equation}
\label{LeCam}
\frac{1}{4} \|f_0 - f_n\|^2_{\Lbf^2(\R^d)}\left[1-\frac{1}{2} \left\|(f_0 \ast Q)^{\otimes n}-(f_n \ast Q)^{\otimes n} \right\|_{\Lbf^1(\R^d)^n}\right]\eqsp,
\end{equation}
where $\ast$ denotes the convolution operator. The goal is then to find two functions $f_0$ and $f_n$ such that the right most term is greater than 1/2 while the left most term is as large as possible. In this lower bound, we consider a closed set $\Hcal$ of functions from $\R^d$ to $\C^d$ such that all elements of $\Hcal$ satisfy H\ref{assum:} and which contains the probability densities of the form given by $f_0$ and $f_n$. This is the starting point of the proof of Theorem~\ref{th:lower} which also relies on a technical conjecture (Conjecture~\ref{conj_polynomes}) which is strongly supported by numerical experiments, see Section~\ref{sec_numerical}  in  the supplementary material.
\begin{theo}
\label{th:lower}
Assume that Conjecture~\ref{conj_polynomes} is true.
Then for all $\kappa \in (0,1)$, $\beta > 0$, $\cbeta > 0$, $c_Q>0$ and $\nu > 0$, there exists a constant $c>0$ such that 
\begin{equation*}
\inf_{\hat{f}}\!\!\!
\underset{R^{\star}:\Phi_{R^\star}\in \Psi(\kappa^{\star},S,\beta,\cbeta) \cap \Hcal}{\sup_{Q^\star \in \Qbf(\nu,c_\nu,c_Q)}}
	\!\!\!\!\!\!\!\!\E_{R^\star, Q^\star}\left[ \|\hat{f}_{\kappa,n}-f^\star\|_{\Lbf^2(\R^d)}^2\right]
		\geq c\left(\frac{\log n}{\log \log n}\right)^{-2\kappa\beta}\eqsp.
\end{equation*}
 The infimum is taken over all estimators $\hat{f}$, that is all measurable functions of $\bfY_1, \dots, \bfY_n$.
\end{theo}

Let $a$ be a (small) real number, in the following it is assumed that $A$ is the $(d_1+d_2)\times (d_1+d_2)$ matrix 
\begin{equation*}
A=\begin{pmatrix}
\begin{matrix}
 1&0&0&\cdots &0\\
 0&1&0&\cdots &0\\
 0&0&1&\cdots &0\\
 \vdots&&&\ddots&\\
 0&0&0&\cdots &1
\end{matrix}
& \rvline & 
\begin{matrix}
 a&a&\cdots &a\\
 0&0&\cdots &0\\
 0&0&\cdots &0\\
 \vdots&&\ddots&\\
 0&0&\cdots &0
\end{matrix}\\
\hline
\begin{matrix}
 a&a&a&\cdots &a\\
 0&0&0&\cdots &0\\
 \vdots&&&\ddots&\\
 0&0&0&\cdots &0
\end{matrix}
& \rvline & 
\begin{matrix}
 1&0&\cdots &0\\
 0&1&\cdots &0\\
 \vdots&&\ddots&\\
 0&0&\cdots &1
\end{matrix}
\end{pmatrix}\eqsp.
\end{equation*}
Assume that the coordinates of $\varepsilon$ are independent identically distributed with density (see \cite{EGR04})
\begin{equation*}
g:x \mapsto c_g \frac{1+\cos(cx)}{(\pi^2-(cx)^2)^2}
\end{equation*}
for some $c>0$, where $c_g$ is such that $g$ is a probability density, with characteristic function 
\begin{equation*}
\Fcal[g]:t \mapsto \left[\left(1-\left| \frac{t}{c}\right|\right)\cos \left(\pi \frac{t}{c}\right) + \frac{1}{\pi} \sin \left(\pi \left| \frac{t}{c}\right|\right)\right]\ind_{-c\leq t \leq c}\eqsp.
\end{equation*}
With an adequate choice of $c$, $Q\in \Qbf(\nu,c_\nu,c_Q)$.
Consider the probability density $u : x \in \R \mapsto c_u \cdot \exp(-1/(1-x^2)) \one_{[-1,1]}(x)$ with the appropriate $c_u > 0$ so that the integral of $u$ equals one. For all $b > 0$ and $x \in \R$, write $u_b(x) = bu(bx)$.
\begin{lem}
\label{lem_fns_dans_Upsilon}
Let $\kappa \in (0,1)$, $c,T > 0$ and $\tau \geq 0$. Then, there exists $x_0 > 0$ such that the following holds. Let $h_\kappa = c_h \exp(- (\sqrt{[1+(x/x_0)^2]/2})^{1/(1-\kappa)})$ where $c_h$ is such that $h_\kappa$ is a probability density. For all $b \geq 1/x_0$, any probability density $\zeta$ such that
$\zeta \leq c ([x \mapsto h_\kappa(x) (1+(x/x_0)^2)^\tau] * u_b)$ satisfies $\Fcal[\zeta] \in \Upsilon_{\kappa,T}$.
\end{lem}
\begin{proof}
The proof is postponed to Appendix~\ref{sec:proof:lower}  in  the supplementary material.
\end{proof}
Let $x_0$ and $h_\kappa$ be as in Lemma~\ref{lem_fns_dans_Upsilon}. Since $h_\kappa$ is infinitely differentiable with square integrable derivatives, for all $\beta > 0$, there exists $L$ such that for all $b > 0$, $\int | \Fcal[h_\kappa * u_b](t) |^2 (1+t^2)^{\beta} \rmd t \leq L$.
Let $(P_K)_{K \geq 0}$ be the family of orthonormal polynomials for the scalar product $\langle f, g \rangle = \int f(x) g(x) h_\kappa(x)^2 \rmd x$ such that $\deg(P_K) = K$. Consider the following conjecture on the properties on these polynomials. 
\begin{conj}
\label{conj_polynomes}
There exists an nonnegative envelope function $F_\text{env}$ that has at most polynomial growth at infinity such that the family $(P_K)_{K \geq 1}$ satisfies $\sup_{K \geq 1} K^{(1-\kappa)/2} \| P_K h_\kappa / F_\text{env} \|_\infty < \infty$ and there exists constants $c_0, c_1, c_2$ such that for all $K$ large enough, there exists at least $c_0 K^\kappa$ intervals of length at least $c_1 K^{-\kappa}$ in $[-1,1]$ on which $|P_K h_\kappa| \geq c_2 K^{(\kappa-1)/2}$.
\end{conj}
Let us comment on the different elements of this conjecture. As discussed above, our objective is to construct two probability densities $f_0$ and $f_n$ that are as far from each other as possible while the resulting distributions of $\bfY$ are as close as possible, see equation~\eqref{LeCam}.
The boundedness of $P_K h_\kappa / F_\text{env}$ ensures that the densities we construct are nonnegative, and the assumption on the intervals is used to prove Corollary~\ref{cor_conj_polynomes}, which controls how large $\|f_0 - f_n\|$ is.

For the sake of simplicity, assume in the following that $\|F_\text{env} h_\kappa \|_{\Lbf^1(\R)} = 1$. Then there exists $c > 0$ and $\tau \geq 0$ such that $F_\text{env}(x) \leq c(1+(x/x_0)^2)^\tau$ for all $x \in \R$, thus making it possible to use Lemma~\ref{lem_fns_dans_Upsilon}. To use the lemma, it is important to note that this $c$ does not depend on the choice of $x_0$.

Another conjecture that gives a better idea of the behaviour of these functions is that there exists a shape function $F_\text{shape}$ such that
\begin{equation}
\label{eq_shape}
\sup_{K \geq 1} \left\| x \mapsto \frac{(P_K h_\kappa) (x) }{K^{(\kappa-1)/2} F_\text{shape}(K^{\kappa-1} x) } \right\|_\infty < \infty \eqsp.
\end{equation}
This function $F_\text{shape}$ diverges at $x_0$ and $-x_0$ for some finite $x_0 \geq 1$, as illustrated in Figure~\ref{fig:norminf} of Section~\ref{sec_numerical}. As $K$ grows, the peak of $P_K h_\kappa$ comes closer to this divergence point, but slowly enough that $F_\text{env}$ only grows polynomially.

\begin{cor}
\label{cor_conj_polynomes}
Assume Conjecture~\ref{conj_polynomes} is true, then there exist $c_b, c_3, c_4$ such that for $K$ large enough, for all $b \geq c_b K^\kappa$,
\begin{align*}
c_3 K^{\kappa-1}
	\geq \| P_{K} h_\kappa^2 \|_{\Lbf^2(\R)}^2 
	\geq \| (P_{K} h_\kappa^2) * u_b \|_{\Lbf^2(\R)}^2
	\geq c_4 K^{\kappa-1} \eqsp.
\end{align*}
\end{cor}

\begin{proof}
The proof is postponed to Appendix~\ref{sec:proof:lower}  in  the supplementary material.
\end{proof}

Note that in the limit $\kappa=1$, $h_\kappa$ is the indicator function of $[-1,1]$, and the orthonormal polynomials $P_K$ are the (normalized) Legendre polynomials. In this setting, Conjecture~\ref{conj_polynomes} (and therefore Equation~\ref{eq_shape}) have been proved with $F_\text{env} = F_\text{shape} = \one_{[-1,1]}$, see Lemma~1 of \cite{meister2007deconvolving}.

In the limit $\kappa = 1/2$, $h_\kappa$ is a normal density, and the functions $P_K h_\kappa$ are the Hermite functions. Approximations of Hermite funtions close to zero and near the turning points are known and corroborate our conjecture, see for instance~\cite[Section A.11]{boyd2018dynamics}: the behaviour near zero is approximately a trigonometric function times a shape function, validating equation~\eqref{eq_shape} (near zero) and the second part of Conjecture~\ref{conj_polynomes}. Near the turning points, they are best approximated by Airy functions with a scaling corresponding to $F_\text{env}(x) = O(x^{1/3})$.

Let $(\alpha_n)_{n \geq 1}$ be a sequence of nonnegative real numbers with limit zero, $(K_n)_{n\geq 1}$ a sequence of integers tending to infinity and $(b_n)_{n\geq 1}$ a sequence of real numbers tending to infinity. Define $f_0$ as the density of $\Xbf$ when for all $1 \leq j \leq d$, $s_j = \zeta_0 = (F_\text{env} h_{\kappa}) * u_b$, and $f_n$ as the density of $\Xbf$ when $S_{1}$ has density
\begin{equation}
\label{eq:def:sn}
\zeta_n= (F_\text{env} h_{\kappa} + \alpha_n P_{K_n} h_{\kappa}^2) * u_{b_n} = \zeta_0 + \alpha_n (P_{K_n} h_{\kappa}^2) * u_{b_n}
\end{equation}
and $S_{2},\ldots,S_{d}$ have density $\zeta_0$. The function $\zeta_n$ is nonnegative as soon as $\alpha_n \leq (\|P_{K_n} h_\kappa / F_\text{env} \|_\infty)^{-1}$, which is of order $K_n^{(1-\kappa)/2}$ by Conjecture~\ref{conj_polynomes}.
Its integral equals one for $K_n \geq 1$ since by definition the function $P_{K_n} h_\kappa^2$ is orthogonal to $P_0$ (which is a constant function) in $\Lbf^2(\R)$, so that the integral of $P_{K_n} h_{\kappa}^2 * u_{b_n}$ is zero. Therefore, $\zeta_n$ is a probability density. In addition, $\Fcal[\zeta_0] \in \Upsilon_{\kappa,T}$ and $\Fcal[\zeta_n] \in \Upsilon_{\kappa,T}$ follow immediately from Lemma~\ref{lem_fns_dans_Upsilon}. 
\begin{lem}
\label{lem:upsilon:lower}
The probability densities $f_0$ and $f_n$ are in $\Upsilon_{\kappa,T}$. 
\end{lem}
\begin{proof}
The proof is postponed to Appendix~\ref{sec:proof:lower}  in  the supplementary material.
\end{proof}
\begin{lem}
\label{lem:alpha}
%
For all $\kappa \in (0,1]$, $\beta > 0$ and $\cbeta> 0$, there exist $x_0 > 0$ and $c_h > 0$ such that $\Fcal[f_0]$ and $\Fcal[f_n]$ belong to $\Psi(\kappa,T,\beta,\cbeta)$ as soon as the two following assumptions are met:
\begin{gather}
\alpha_n \leq \|P_{K_n} h_\kappa\|_\infty^{-1} \eqsp, \\
\label{eq:alphan}
\alpha_n^2 \|P_{K_n} h_\kappa^2\|_{\Lbf^2(\R)}^2 \leq c_h b_n^{-2\beta} \eqsp.
\end{gather}
\end{lem}
\begin{proof}
The proof is postponed to Appendix~\ref{sec:proof:lower}  in  the supplementary material.
\end{proof}
Following \cite{meister2007deconvolving}, it is straightforward to establish that
\begin{equation*}
1-\frac{1}{2} \|(f_{0}\ast Q)^{\otimes n}-(f_{n}\ast Q)^{\otimes n} \|_{\Lbf^1(\R^d)^n}
	\geq \left(1-\frac{1}{2} \|(f_{0}\ast Q)-(f_{n}\ast Q) \|_{\Lbf^1(\R^d)}
\right)^{n}\eqsp.
\end{equation*}
Then, by \eqref{LeCam}, the minimax risk based on $n$ observations is lower bounded by $c\|f_{0}-f_{n}\|^{2}_{\Lbf^2(\R^d)}$ for some constant $c>0$ if $(\alpha_n)_{n\geq 1}$, $(b_n)_{n\geq 1}$ and $(K_n)_{n\geq 1}$ are chosen such that
\begin{equation}
\label{eq:f0minusfn}
\int_{\R^{d}} \left| (f_{0}\ast Q)(x)-(f_{n}\ast Q)(x) \right| \rmd x=O\left(\frac{1}{n}\right)\eqsp.
\end{equation}
\begin{lem}
\label{lem:choicealphabetaK}
Assume that \eqref{eq:alphan} holds and that
\begin{equation}
\label{eq:Kn}
K_n = \frac{c_h}{\kappa} \left(\frac{\log n}{\log \log n}\right)\eqsp.
\end{equation}
Then, \eqref{eq:f0minusfn} holds.
\end{lem}
\begin{proof}
The proof is postponed to Appendix~\ref{sec:proof:lower}.
\end{proof}
Therefore, the minimax risk based on $n$ observations is lower bounded by $c\|f_{0}-f_{n}\|^{2}_{\Lbf^2(\R^d)}$ for some constant $c>0$. In addition, by definition of $f_0$ and $f_n$, for all $u\in\R^d$,
\begin{equation*}
f_0(u) = \mathrm{Det}(A)^{-1}\prod_{j=1}^d \zeta_0((A^{-1}u)_j)\quad\mathrm{and}\quad 
f_n(u) = \mathrm{Det}(A)^{-1}\zeta_n((A^{-1}u)_1)\prod_{j=2}^d \zeta_0((A^{-1}u)_j)\eqsp.
\end{equation*}
Therefore, there exists a constant $c>0$ such that,
\begin{equation*}
\|f_{0}-f_{n}\|^{2}_{\Lbf^2(\R^d)}\geq c \|\zeta_0\|^{2(d-1)}_{\Lbf^2(\R)} \|\zeta_0-\zeta_n\|^{2}_{\Lbf^2(\R)} \geq c \alpha_n^2 K_n^{\kappa-1}\eqsp,
\end{equation*}
by Corollary~\ref{cor_conj_polynomes} and \eqref{eq:alphan}. Then, choosing $b_n = c_b K_n^{\kappa}$, $\alpha_n^2 \propto K_n^{-2\kappa\beta} / \|P_{K_{n}} h_\kappa^2\|_{\Lbf^2(\R)}^2 \propto K_n^{1 - \kappa - 2\kappa\beta}$ (by Corollary~\ref{cor_conj_polynomes}) and $K_n$ as in~\eqref{eq:Kn} yields
\begin{equation*}
\|f_{0}-f_{n}\|^{2}_{\Lbf^2(\R^d)}
	\geq c b_n^{-2\beta}
	\geq c K_n^{-2\kappa\beta}
	\geq c \left(\frac{\log n}{\log \log n}\right)^{-2\kappa\beta}\eqsp.
\end{equation*}
The condition $\alpha_n \leq (\|P_{K_n} h_\kappa / F_\text{env}\|_\infty)^{-1}$ corresponds to $K_n^{(1-\kappa)/2 - \kappa\beta} = O(K_n^{(1-\kappa)/2})$, which is always true.

%% file: deconvolNP.bbl
\begin{thebibliography}{}

\bibitem[Attias and Schreiner, 1998]{attias1998blind}
Attias, H. and Schreiner, C.~E. (1998).
\newblock Blind source separation and deconvolution: the dynamic component
  analysis algorithm.
\newblock {\em Neural computation}, 10(6):1373--1424.

\bibitem[Batson and Royer, 2019]{noise2self:2019}
Batson, J. and Royer, L. (2019).
\newblock Noise2self: Blind denoising by self-supervision.
\newblock {\em Proceedings of the 36th International Conference on Machine
  Learning (ICML)}.

\bibitem[Baudry et~al., 2012]{baudry2012slope}
Baudry, J.-P., Maugis, C., and Michel, B. (2012).
\newblock Slope heuristics: overview and implementation.
\newblock {\em Statistics and Computing}, 22(2):455--470.

\bibitem[Belomestny and Goldenshluger, 2019]{belomestny:goldenshluger:2019}
Belomestny, D. and Goldenshluger, A. (2019).
\newblock Density deconvolution under general assumptions on the distribution
  of measurement errors.
\newblock {\em arXiv:1907.11024}.

\bibitem[Bertin et~al., 2016]{BLR16Lepski}
Bertin, K., Lacour, C., and Rivoirard, V. (2016).
\newblock Adaptive pointwise estimation of conditional density function.
\newblock In {\em Annales de l'Institut Henri Poincar{\'e}, Probabilit{\'e}s et
  Statistiques}, volume~52, pages 939--980. Institut Henri Poincar{\'e}.

\bibitem[Boyd, 2018]{boyd2018dynamics}
Boyd, J.~P. (2018).
\newblock {\em Dynamics of the equatorial ocean}.
\newblock Springer.

\bibitem[Butucea and Tsybakov, 2008a]{butucea:tsybakov:2008a}
Butucea, C. and Tsybakov, B. (2008a).
\newblock Sharp optimality in density deconvolution with dominating bias. i.
\newblock {\em Theory of Probability and Its Applications}, 52(1):24--39.

\bibitem[Butucea and Tsybakov, 2008b]{butucea:tsybakov:2008b}
Butucea, C. and Tsybakov, B. (2008b).
\newblock Sharp optimality in density deconvolution with dominating bias. ii.
\newblock {\em Theory of Probability and Its Applications}, 52(2):237--249.

\bibitem[Campisi and Egiazarian, 2017]{campisi2017blind}
Campisi, P. and Egiazarian, K. (2017).
\newblock {\em Blind image deconvolution: theory and applications}.
\newblock CRC press.

\bibitem[Carroll and Hall, 1988]{MR997599}
Carroll, R. and Hall, P. (1988).
\newblock Optimal rates of convergence for deconvolving a density.
\newblock {\em J. Amer. Statist. Assoc.}, 83(404):1184--1186.

\bibitem[Chazal et~al., 2011]{CDSM11}
Chazal, F., Cohen-Steiner, D., and M\'erigot, Q. (2011).
\newblock Geometric inference for probability measures.
\newblock {\em Journal on Foundations of Computational Mathematics}, 11(6).

\bibitem[Chazal et~al., 2017]{chazal2017robust}
Chazal, F., Fasy, B., Lecci, F., Michel, B., Rinaldo, A., Rinaldo, A., and
  Wasserman, L. (2017).
\newblock Robust topological inference: Distance to a measure and kernel
  distance.
\newblock {\em The Journal of Machine Learning Research}, 18(1):5845--5884.

\bibitem[Chazal and Michel, 2017]{CM17}
Chazal, F. and Michel, B. (2017).
\newblock An introduction to topological data analysis: fundamental and
  practical aspects for data scientists.
\newblock {\em arXiv preprint arXiv:1710.04019}.

\bibitem[Comon, 1994]{comon:1994}
Comon, P. (1994).
\newblock Independent component analysis: a new concept?
\newblock {\em Signal Processing}, 36:287--314.

\bibitem[Comte and Lacour, 2013]{comte2013anisotropic}
Comte, F. and Lacour, C. (2013).
\newblock Anisotropic adaptive kernel deconvolution.
\newblock In {\em Annales de l'IHP Probabilit{\'e}s et statistiques},
  volume~49, pages 569--609.

\bibitem[Delaigle et~al., 2008]{MR2396811}
Delaigle, A., Hall, P., and Meister, A. (2008).
\newblock On deconvolution with repeated measurements.
\newblock {\em Ann. Statist.}, 36(2):665--685.

\bibitem[Devroye, 1989]{MR1033106}
Devroye, L. (1989).
\newblock Consistent deconvolution in density estimation.
\newblock {\em Canad. J. Statist.}, 17(2):235--239.

\bibitem[Eckle et~al., 2016]{eckle:bissantz:dette:2016}
Eckle, K., Bissantz, N., and Dette, H. (2016).
\newblock Multiscale inference for multivariate deconvolution.
\newblock {\em arXiv:1611.05201}.

\bibitem[Ehm et~al., 2004]{EGR04}
Ehm, W., Gneiting, T., and Richards, D. (2004).
\newblock Convolution roots of radial positive definite functions with compact
  support.
\newblock {\em Transactions of the American Mathematical Society}, 356(11).

\bibitem[Eriksson and Koivunen, 2004]{eriksson:koivunen:2004}
Eriksson, J. and Koivunen, V. (2004).
\newblock Identifiability, separability, uniqueness of linear {ICA} models.
\newblock {\em IEEE Signal Processing Letters}, 11:601--604.

\bibitem[Fan, 1991]{MR1126324}
Fan, J. (1991).
\newblock On the optimal rates of convergence for nonparametric deconvolution
  problems.
\newblock {\em Ann. Statist.}, 19(3):1257--1272.

\bibitem[Gassiat et~al., 2020]{gassiat:lecorff:lehericy:2019}
Gassiat, E., Le~Corff, S., and Leh\'ericy, L. (2020).
\newblock Identifiability and consistent estimation of nonparametric
  translation hidden markov models with general state space.
\newblock {\em Journal of Machine Learning Research}, 21(115):1--40.

\bibitem[Gassiat and Rousseau, 2016]{gassiat:rousseau:2016}
Gassiat, E. and Rousseau, J. (2016).
\newblock Nonparametric finite translation hidden {M}arkov models and
  extensions.
\newblock {\em Bernoulli}, 22(1):193--212.

\bibitem[Goldenshluger and Lepski, 2008]{goldenshluger2008}
Goldenshluger, A. and Lepski, O. (2008).
\newblock Universal pointwise selection rule in multivariate function
  estimation.
\newblock {\em Bernoulli}, 14(4):1150--1190.

\bibitem[Goldenshluger and Lepski, 2013]{goldenshluger2013}
Goldenshluger, A. and Lepski, O. (2013).
\newblock General selection rule from a family of linear estimators.
\newblock {\em Theory of Probability \& Its Applications}, 57(2):209--226.

\bibitem[Hyvarinen et~al., 2002]{hyvarinen:karhunen:oja:2000}
Hyvarinen, A., Karhunen, J., and Oja, E. (2002).
\newblock {\em Independent {C}omponent {A}nalysis}.
\newblock John Wiley \& Sons.

\bibitem[Johannes, 2009]{johannes:2009}
Johannes, J. (2009).
\newblock Deconvolution with unknown error distribution.
\newblock {\em The {A}nnals of {S}tatistics}, 37:2301--2323.

\bibitem[Jutten, 1991]{jutten:1991}
Jutten, C. (1991).
\newblock Blind separation of sources, part {I}: an adaptive algorithm based on
  neuromimetic architecture.
\newblock {\em Signal Processing}, 2(4):1--10.

\bibitem[Khemakhem et~al., 2020]{khemakhem:2020}
Khemakhem, I., Kingma, D., Pio~Monti, R., and Hyvarinen, A. (2020).
\newblock Variational autoencoders and nonlinear ica: A unifying framework.
\newblock {\em ArXiv:1907.04809}.

\bibitem[Kotlarski, 1967]{kotlarski:1967}
Kotlarski, I. (1967).
\newblock On characterizing the gamma and the normal distribution.
\newblock {\em Pacific Journal of Mathematics}, 20(1):69--76.

\bibitem[Krull et~al., 2019]{noise2void:2019}
Krull, A., Buchholz, T.-O., and Jug, F. (2019).
\newblock Noise2void - learning denoising from single noisy images.
\newblock {\em IEEE/CVF Conference on Computer Vision and Pattern Recognition
  (CVPR)}.

\bibitem[Kundur and Hatzinakos, 1996]{kundur1996blind}
Kundur, D. and Hatzinakos, D. (1996).
\newblock Blind image deconvolution.
\newblock {\em IEEE signal processing magazine}, 13(3):43--64.

\bibitem[Lacour and Comte, 2010]{lacour:comte:2010}
Lacour, C. and Comte, F. (2010).
\newblock Pointwise deconvolution with unknown error distribution.
\newblock {\em Comptes {R}endus {M}athematique de l'{A}cademie des {S}ciences},
  348(5-6):323--326.

\bibitem[Le~Cam, 2012]{le2012asymptotic}
Le~Cam, L. (2012).
\newblock {\em Asymptotic methods in statistical decision theory}.
\newblock Springer Science \& Business Media.

\bibitem[Li and Vuong, 1998]{MR1625869}
Li, T. and Vuong, Q. (1998).
\newblock Nonparametric estimation of the measurement error model using
  multiple indicators.
\newblock {\em J. Multivariate Anal.}, 65(2):139--165.

\bibitem[Lin and Kulasekera, 2007]{lin2007identifiability}
Lin, W. and Kulasekera, K. (2007).
\newblock Identifiability of single-index models and additive-index models.
\newblock {\em Biometrika}, 94(2):496--501.

\bibitem[Liu and Taylor, 1989]{MR1047309}
Liu, M.~C. and Taylor, R.~L. (1989).
\newblock A consistent nonparametric density estimator for the deconvolution
  problem.
\newblock {\em Canad. J. Statist.}, 17(4):427--438.

\bibitem[Massart, 2007]{massart:2003}
Massart, P. (2007).
\newblock {\em Concentration Inequalities and Model Selection : Ecole d'Et\'e
  de Probabilit\'es de Saint-Flour XXXIII - 2003}.
\newblock Berlin ; Heidelberg (DEU) ; New York : Springer.

\bibitem[Meister, 2004]{meister2004effect}
Meister, A. (2004).
\newblock On the effect of misspecifying the error density in a deconvolution
  problem.
\newblock {\em Canadian Journal of Statistics}, 32(4):439--449.

\bibitem[Meister, 2007]{meister2007deconvolving}
Meister, A. (2007).
\newblock Deconvolving compactly supported densities.
\newblock {\em Mathematical Methods of Statistics}, 16(1):63--76.

\bibitem[Meister, 2009]{meister:2009}
Meister, A. (2009).
\newblock {\em Deconvolution problems in nonparametric statistics}.
\newblock Springer.

\bibitem[Moulines et~al., 1997]{moulines1997maximum}
Moulines, E., Cardoso, J.-F., and Gassiat, E. (1997).
\newblock Maximum likelihood for blind separation and deconvolution of noisy
  signals using mixture models.
\newblock In {\em {IEEE} International Conference on Acoustics, Speech, and
  Signal Processing}, volume~5, pages 3617--3620. IEEE.

\bibitem[Pfister et~al., 2019]{pfister:2019}
Pfister, N., Weichwald, S., Buhlmann, B., and Scholkopf, B. (2019).
\newblock Robustifying independent component analysis by adjusting for
  group-wise stationary noise.
\newblock {\em Journal of Machine Learning Research}, 20:1--50.

\bibitem[Sarkar et~al., 2018]{sarkar2018bayesian}
Sarkar, A., Pati, D., Chakraborty, A., Mallick, B.~K., and Carroll, R.~J.
  (2018).
\newblock Bayesian semiparametric multivariate density deconvolution.
\newblock {\em Journal of the American Statistical Association},
  113(521):401--416.

\bibitem[Schennach and Hu, 2013]{MR3174611}
Schennach, S.~M. and Hu, Y. (2013).
\newblock Nonparametric identification and semiparametric estimation of
  classical measurement error models without side information.
\newblock {\em J. Amer. Statist. Assoc.}, 108(501):177--186.

\bibitem[Starck et~al., 2002]{starck2002deconvolution}
Starck, J.-L., Pantin, E., and Murtagh, F. (2002).
\newblock Deconvolution in astronomy: A review.
\newblock {\em Publications of the Astronomical Society of the Pacific},
  114(800):1051.

\bibitem[Stefanski and Carroll, 1990]{MR1054861}
Stefanski, L. and Carroll, R.~J. (1990).
\newblock Deconvoluting kernel density estimators.
\newblock {\em Statistics}, 21(2):169--184.

\bibitem[Stein and Shakarchi, 2003]{Stein:complex}
Stein, E. and Shakarchi, R. (2003).
\newblock {\em Complex Analysis}.
\newblock Princeton University Press, Princeton.

\bibitem[Yuan, 2011]{yuan:2011}
Yuan, M. (2011).
\newblock On the identifiability of additive index models.
\newblock {\em Statistica Sinica}, 21:1901--1911.

\end{thebibliography}
